\documentclass[a4paper,12pt]{article}

\newenvironment{proof}{\noindent {\bf Proof:}}{\hfill $\Box$}

\newtheorem{theorem}{Theorem}

\newtheorem{proposition}{Proposition}
\newtheorem{corollary}{Corollary}

\newtheorem{remark}{Remark}

\usepackage{booktabs}

\usepackage{tikz}

\usepackage{algorithm}
\usepackage[noend]{algpseudocode}
\makeatletter
\def\BState{\State\hskip-\ALG@thistlm}
\makeatother
\usepackage{algorithmicx}

\usepackage{amsmath}
\usepackage{amssymb}
\usepackage{amsfonts}
\usepackage{graphicx}
\usepackage{ dsfont }

\usepackage[normalem]{ulem}
\usepackage{color}

\usepackage{tikz}
\usepackage{lipsum}

\newcommand{\mr}[1]{\mathrm{#1}}

\newcommand{\Hc}{\mathcal{H}}
\newcommand{\Xc}{\mathcal{X}}

\newcommand{\sigalg}{\mathfrak{M}}
\newcommand{\Rb}{\mathbb{R}}
\newcommand{\Cb}{\mathbb{C}}
\newcommand{\Zb}{\mathbb{Z}}
\newcommand{\Tb}{\mathbb{T}}
\newcommand{\Nb}{\mathbb{N}}

\newcommand{\Fc}{\mathcal{F}}
\newcommand{\Zc}{\mathcal{Z}}

\newcommand{\figCaptionSize}{\small}
\newtheorem{example}{Example}

\usepackage{url}

\def\be{\begin{equation}}
\def\ee{\end{equation}}

\title{\bf Data-driven spectral analysis\ of the Koopman operator}

\begin{document}

\author{Milan Korda$^1$, Mihai Putinar$^{1,2}$, Igor Mezi{\'c}$^1$}

\footnotetext[1]{Milan Korda, Mihai Putinar and Igor Mezi{\'c} are with the University of California, Santa Barbara,\; {\tt milan.korda@engineering.ucsb.edu,  mezic@engineering.ucsb.edu}}
\footnotetext[2]{Mihai Putinar is also affiliated to the Newcastle University, Newcastle upon Tyne, UK, \: {\tt mihai.putinar@ncl.ac.uk}}

\maketitle
\begin{abstract}
\looseness-1 Starting from measured data, we develop a method to compute the fine structure of the spectrum of the Koopman operator with rigorous convergence guarantees. The method is based on the observation that, in the measure-preserving ergodic setting, the moments of the spectral measure associated to a given observable are computable from a single trajectory of this observable. Having finitely many moments available, we use the classical Christoffel-Darboux kernel to separate the atomic and absolutely continuous parts of the spectrum, supported by convergence guarantees as the number of moments tends to infinity. In addition, we propose a technique to detect the singular continuous part of the spectrum as well as two methods to approximate the spectral measure with guaranteed convergence in the weak topology, irrespective of whether the singular continuous part is present or not. The proposed method is simple to implement and readily applicable to large-scale systems since the computational complexity is dominated by inverting an $N\times N$ Hermitian positive-definite Toeplitz matrix, where $N$ is the number of moments, for which efficient and numerically stable algorithms exist; in particular, the complexity of the approach is independent of the dimension of the underlying state-space. We also show how to compute, from measured data, the spectral projection on a given segment of the unit circle, allowing us to obtain a finite approximation of the operator that explicitly takes into account the point and continuous parts of the spectrum. Finally, we describe a relationship between the proposed method and the so-called Hankel Dynamic Mode Decomposition, providing new insights into the behavior of the eigenvalues of the Hankel DMD operator. A number of numerical examples illustrate the approach, including a study of the spectrum of the lid-driven two-dimensional cavity flow.
 \end{abstract}


\begin{center}\small
\end{center} 
 
\section{Introduction}
Spectral methods have been increasingly popular in data-driven analysis of large-scale nonlinear dynamical systems. Among them, in particular, methods based on approximation of the Koopman operator have been extremely successful across a wide range of fields. This operator, originally defined almost a century ago by Koopman ~\cite{koopman:1931}, is a \emph{linear} infinite-dimensional operator that fully describes the underlying \emph{nonlinear} dynamical system\footnote{For the related and subtle question of the relationship between the so-called spectral and spatial isomorphisms of dynamical systems, see~\cite{redei2012history}. See also~\cite{denker1978unitary} for the related question on when a unitary operator on an abstract Hilbert space is inducible by the composition operator associated to a measure preserving transformation.}. An approximation of the spectrum of the Koopman operator encodes information about the dynamics of the underlying system. For example, global stability is analyzed in~\cite{mauroy2016global}, whereas \cite{mauroy2013isostables} deals with the so-called isostables and isochrones; ergodic partition and mixing properties are analyzed in~\cite{budisicetal:2012}, and~\cite{korda2018linear,Bruntonetal:2016} utilize the Koopman operator approximations for control whereas~\cite{mezic:2005} for model reduction. Recent applications include fluid dynamics \cite{SharmaMezicM16,arbabi2017study}, power grids \cite{Susukietal:2016}, neurodynamics \cite{brunton2016extracting}, energy efficiency \cite{GeorgescuandMezic:2015},  molecular physics \cite{wuetal:2016} and data fusion~\cite{williams2015data}.

Since the early work~\cite{mezic:2005}, there has been a number of algorithms proposed for approximation of the spectrum of the Koopman operator, including Fourier averages~\cite{mezic:2005} and variations of the dynamic mode decomposition (DMD), e.g.~\cite{rowley2009,williams2015data}. The benefits of averaging methods lie in their solid theoretical support  with strong convergence results existing; a limitation of this approach is the requirement of having a grasp on the eigenvalues of the operator beforehand\footnote{From a practical perspective, this downside is not so severe, since the Fourier averages can be computed extremely fast using FFT and hence one can utilize a grid search of the eigenvalues.} and the fact that these methods do not provide any information on the continuous part of the spectrum of the operator. On the other hand, the DMD-like methods do not require the knowledge of the eigenvalues beforehand but their spectral convergence properties are not as favorable~\cite{korda2018convergence} and similarly to averaging methods they do not systematically handle the continuous part of the spectrum. Another approach was pursued in~\cite{giannakis}  where  Koopman eigenfunctions are computed through a regularized advection-diffusion operator.

The present article proposes a new harmonic analysis based, data-driven, approach for approximation of the spectrum of the Koopman operator that is capable of computing both the point and continuous parts of the spectrum (with convergence guarantees), thereby generalizing the method of~\cite{mezic:2005}. We start with the observation that, in the measure-preserving ergodic setting, the moments of the spectral measure associated to a given observable are computable from a single trajectory of this observable. Therefore, in this case, the problem of approximating the spectrum of the Koopman operator reduces to that of reconstructing a measure from its moments. Since the operator is unitary, the measure is supported on the unit circle in the complex plane, which is a very well understood setting, with the earliest results going back to classical Fourier analysis. In our work, we leverage modern results from this field in the Koopman operator setting. We rely primarily on the Christoffel-Darboux kernel which allows us to approximate both the atomic part of the spectrum (i.e., the eigenvalues) as well as the absolutely continuous part. In addition, we develop a method to detect the presence of the singular continuous part of the spectrum as well as two methods to construct approximations to the measure converging weakly even for complicated spectral measures with nonzero singular continuous spectrum. The first method is based on quadrature (with the help of convex optimization) and the second one on the classical Ces\` aro summation. This allows for a detailed understanding of the spectrum derived from raw data and opens the door to approximations of the Koopman operator that explicitly take into account the continuous part of the spectrum. In this work, we approximate the operator as a sum of spectral projections onto segments of the unit circle, where the segments can be taken to be singletons if there is an eigenvalue at a given point. These projections can be readily computed from data and the approximation converges in the strong operator topology.

The framework developed in our article is simple to implement and readily applicable to high-dimensional systems since the computational complexity is fully determined only by the number of the moments $N$ and in particular is independent of the dimension of the underlying state-space. To be more precise, the complexity is governed by the inversion (or Cholesky factorization) of an $N\times N$ Hermitian positive-definite Toeplitz matrix which can be carried out with asymptotic complexity $O(N^2)$ or even $O(N \log^2(N))$ as opposed to $O(N^3)$ for a general matrix; see, e.g.,~\cite{trench1964algorithm,strang1986proposal,brent1988old}.

For the rich history and relevance of the Christoffel-Darboux kernel in the theory of orthogonal polynomials we refer to the ample eulogy of G\' eza Freud by Nevai \cite{nevai1986geza}, the more recent articles by
Simon \cite{simon2009duke} and Totik~\cite{totik2000asymptotics} as well as the comprehensive books by Simon~\cite{simonSzegoDescendants} and Levin and Lubinsky~\cite{levin2012orthogonal}. See also the seminal work of Wiener~\cite{wiener1930generalized} for a predecessor of the methods used in this work. For a comprehensive reference on approximation of Toeplitz operators, see~\cite{bottcher2013analysis}.

Finally, we would also like to make a brief connection to the spectral theory of stationary stochastic processes and signal processing, which offer a diverse set of tools for estimation of spectral properties (see, e.g., \cite{marple1987digital}). These tools are typically tailored for a particular setting (e.g., purely absolutely continuous spectrum or purely discrete spectrum) and do not provide a single comprehensive method for estimation of the individual components of the spectrum. Also, importantly, these methods do not offer any insight into the underlying state-space structures which are provided by the spectral projection techniques proposed in this work.

\section{Problem statement}
Throughout this article we  consider a  discrete-time dynamical systems of the form
\begin{equation}\label{eq:sys}
x^+ = T(x),
\end{equation}
where $x$ is the state of the system, $x^+$ the successor state and $T:\Xc\to\Xc$ the transition mapping defined on the state space $\Xc$. Henceforth we work in an invertible measure preserving setting, i.e., we assume that $T$ is a measurable bijection on the state space $\Xc$ endowed with a sigma algebra $\sigalg$ and a measure $\nu$ defined on $\sigalg$ such that
\begin{equation}\label{eq:meas_pres}
\nu(T^{-1}(A)) = \nu(A), \quad \forall A\in \sigalg.
\end{equation}
In applications, the state space $\Xc$ is typically a subset of $\Rb^n$ for finite-dimensional systems or a subset of a Banach space for infinite dimensional systems (e.g., arising in the study of partial differential equations) and the sigma algebra $\sigalg$ is typically the Borel sigma algebra. In practice, the assumption that the system is measure-preserving implies that we are interested in on-attractor, post-transient, behavior of the dynamical system.

\subsection{The Koopman operator}
A canonical object associated to the dynamical system $(\ref{eq:sys})$ is the Koopman operator $U:\Hc\to \Hc$ defined for all $f:\Xc\to \Cb$, $f\in\Hc$, by
\begin{equation}
Uf = f\circ T,
\end{equation}
where $\circ$ denotes the composition of functions. The choice of the function space $\Hc$ depends on the particular class of systems studied. In our case of measure-preserving systems, a suitable choice\footnote{The fact that $f\circ T \in L_2(\nu)$ for any $f\in L^2(\nu)$ follows from the assumption of $T$ being measure preserving with respect to the measure $\nu$.} is
\[
\Hc = L_2(\nu),
\]
the Hilbert space of complex-valued functions square integrable with respect to the preserved measure $\nu$ with the standard inner product
\[
\langle f,g\rangle = \int_{\Xc} f\bar g\,d\nu,
\]
where $\bar g$ denotes the complex conjugate of $g$. The Koopman operator $U$ is a \emph{linear} operator (acting on an infinite dimensional space) which encodes an equivalent description of the nonlinear dynamical system~(\ref{eq:sys}); indeed, from the knowledge of the action of $U$ on all functions in $\Hc$, one can recover the transition mapping $T$, uniquely up to sets of zero measure $\nu$ (to see this, take for example the collection of all measurable indicator functions $\{I_A\}_{A\in\sigalg}$ and use the fact that $U I_{A} = I_{T^{-1}(A)}$). The functions $f \in \Hc$ are referred to as \emph{observables} as they often represent physical measurements taken on the dynamical systems.

Notable in our setting, the operator $U$ is \emph{unitary}, i.e., $U^{-1} = U^\ast$, where $U^\ast$ denotes the adjoint of $U$. Indeed, for any $f
\in \Hc$ and $g
\in \Hc$ we have
\[
\langle f,U^{\ast} g\rangle = \langle Uf,g\rangle = \int_{\Xc} (f\circ T) \bar g \, d\nu = \int_{\Xc} (f\circ T) (\bar g\circ T^{-1}\circ T) \, d\nu = \int_{\Xc} f (\bar g\circ T^{-1}) \, d\nu  = \langle f, U^{-1} g\rangle,
\]
where the third equality follows from bijectivity of $T$ and the fourth from~(\ref{eq:meas_pres}). 
\subsection{Spectral resolution}
Since $U$ is unitary, the spectrum of $U$, $\sigma(U)$, lies on the unit circle $\Tb$ in the complex plane and the spectral theorem~\cite[Part II, X.2.2, Theorem~1, p.\,895]{dunford1971linear} ensures the existence of a projection-valued\footnote{By a projection-valued measure, we mean a measure with values in the space of orthogonal projection operators on~$\Hc$.} spectral measure $E$ supported on $\sigma(U)$ such that
\begin{equation}\label{eq:specExp}
U = \int_{\Tb} z \, dE(z).
\end{equation}
The relation~(\ref{eq:specExp}) is called the \emph{spectral resolution} or \emph{spectral expansion} of $U$. The measure $E$ decomposes into three mutually singular measures as
\begin{equation}
E = E_{\mr{at}} +  E_{\mr{ac}} +  E_{\mr{sc}},
\end{equation}
where the \emph{atomic part} $E_{\mr{at}}$ is supported on the at most countable set of eigenvalues of $U$ and $E_{\mr{ac}}$ and $E_{\mr{sc}}$ are the \emph{absolutely continuous} (AC) and \emph{singular continuous} (SC) parts of $E$ with their supports referred to as the absolutely continuous respectively singular continuous spectrum of $U$. The main goal of this work is to understand the individual components of the spectrum from data and use this information to construct an approximation of $U$.

In order to do so, we observe that for any $f\in \Hc$, the projection-valued measure $E$ defines an ordinary, real-valued, positive measure on $\Tb$ by
\begin{equation}\label{eq:muf_def}
\mu_f(A) := \langle E(A)f,f\rangle
\end{equation}
for all Borel sets $A\subset\Tb$. Crucially, by the spectral theorem, the moments of $\mu_f$
\begin{equation}\label{eq:mom_mu_def}
m_{k} :=  \int_{\Tb} z^k \, d\mu_f(z), \quad k\in\Zb
\end{equation}
satisfy
\begin{equation}\label{eq:mom_mu_C}
m_{k} = \langle U^kf,f \rangle,\quad k\in\Zb\,,
\end{equation}
which will be instrumental in computing the moments from data in Section~\ref{sec:momComp}.

By density of the trigonometric polynomials in the space of continuous functions defined on the unit circle, the moment sequence $(m_k)_{k\in \Zb}$ uniquely determines the measure $
\mu_f$. In fact, since the measure is real-valued, the relation $m_{-k} = \bar m_k$ holds and hence $(m_k)_{k=0}^\infty$ uniquely determines $\mu_f$.

Importantly, the sole knowledge of $\mu_f$ determines the operator $U$ provided that the function $f$ is \emph{$\ast$-cyclic}, i.e.,
\begin{equation}\label{eq:cyclic}
\Hc_f := \overline{\mr{span}\{f,Uf,U^{-1}f, U^2f, U^{-2}f, \ldots\}} = \Hc.
\end{equation}
For the sake of completeness we state and prove this known result here:
\begin{proposition}
If the observable $f$ is $\ast$-cyclic, then the measure $\mu_f$ defined in~(\ref{eq:muf_def}) fully determines the operator $U$.
\end{proposition}
\begin{proof}
First, by the Riesz representation theorem for Hilbert spaces, the operator $U$ is determined by the values of $\langle Ug,h\rangle$ for all $g,h\in \Hc$. Since $f$ is cyclic, for any such $g$ and $h$ and any $\epsilon >0$ there exists an $N > 0$ such that $\|g - \sum_{i=-N}^N \alpha_i U^i f \| < \epsilon$ and $\|h - \sum_{i=-N}^N \beta_i U^i f \| < \epsilon$ for some $\alpha_i\in\Cb$ and $\beta_i\in\Cb$. Denoting $\tilde g =\sum_{i=-N}^N \alpha_i U^i f $ and $\tilde h =\sum_{i=-N}^N \beta_i U^i f $, we have
\begin{align*}
\langle Ug,h\rangle & =  \langle U(g-\tilde{g} + \tilde g),h-\tilde{h} + \tilde h\rangle \\  
&= \langle U(g-\tilde{g}),h-\tilde{h}\rangle + \langle U(g-\tilde{g}),\tilde h\rangle + \langle U\tilde g,h - \tilde h\rangle + \langle U\tilde{g},\tilde h\rangle \\ &  = c(\epsilon)
+ \Big\langle U\sum_{i=-N}^N \alpha_i U^i f , \sum_{i=-N}^N \beta_i U^i f \Big\rangle = c(\epsilon) + \sum_{i,j}{\alpha_i}\bar\beta_j\langle U^{i+1-j}f,f\rangle \\ &=  
 c(\epsilon) + \sum_{i,j}{\alpha_i}\bar\beta_j\int_\Tb z^{i+1-j}\,d\mu_f =  c(\epsilon) + \sum_{i,j}{\alpha_i}\bar\beta_j m_{i+1-j},
\end{align*}
where $c(\epsilon) = \langle U(g-\tilde{g}),h-\tilde{h}\rangle + \langle U(g-\tilde{g}),\tilde h\rangle + \langle U\tilde g,h - \tilde h\rangle$ satisfies $| c(\epsilon)| \le \epsilon^2 + \epsilon(\|h\| + \epsilon)+ \epsilon(\|g\| + \epsilon) = 3\epsilon^2 + \epsilon(\|g\| + \|h\|)$ by the Schwarz inequality and the fact that $\|U\| = 1$ (since $U$ is unitary). As $\epsilon$ was arbitrary, the proof is complete.
\end{proof}

Therefore, provided that $f$ is $\ast$-cyclic, the  sequence of complex numbers $(m_k)_{k=0}^\infty$ fully determines the operator $U$.

\begin{remark}
If $f$ is not $\ast$-cyclic (i.e., $\Hc_f\ne \Hc$), then the measure $\mu_f$ determines the operator $U$ on $\Hc_f\subset \Hc$, which is the smallest closed subspace containing $f$ invariant under the actions of $U$ and $U^\ast$.
\end{remark}

Whether or not $f$ is $\ast$-cyclic, the information contained in $\mu_f$ is of great importance for understanding the spectrum of the operator $U$ and for its approximation. In particular, similar to $E$, the measure  $\mu_f$ decomposes as
\begin{equation}\label{eq:lebDec_muf}
\mu_f = \mu_{\mr{at}} + \mu_{\mr{ac}} + \mu_{\mr{sc}}
\end{equation}
with $\mu_{\mr{at}}$ being atomic (i.e., at most a countable sum of Dirac masses), $\mu_{\mr{ac}}$ being absolutely continuous with respect to the Lebesgue measure on $\Tb$ and $\mu_{\mr{sc}}$ being singularly continuous with respect to the Lebesgue measure on $\Tb$. The supports of $\mu_{\mr{at}}$, $\mu_{\mr{ac}}$, $\mu_{\mr{sc}}$ are, respectively, included in the supports of $E_{\mr{at}}$, $E_{\mr{ac}}$ and $E_{\mr{sc}}$, with equality of the supports if and only if $f$ is $\ast$-cyclic. In particular the locations of the atoms in $\mu_{\mr{at}}$ corresponds to eigenvalues of $U$.

In this work we show:
\begin{enumerate}
\item How the moments $(m_k)_{k=0}^N$, for any $N\in\Nb$, can be computed from data.
\item How the measure $\mu_f$ can be approximately reconstructed from these moments.
\item How the operator $U$ can be approximated using the knowledge of $\mu_f$.
 \end{enumerate}


\section{Computation of moments from data}\label{sec:momComp}
In this section we describe how to estimate the first $N+1$ moments \[
m_k = \langle U^kf,f\rangle= \int_X (f\circ T^k)\bar f\,d\nu , \quad k=0,\ldots,N,\]
given data in the form of $M$ measurements (or snapshots) of the observable $f\in \Hc$ in the form \begin{equation}\label{eq:data}
y_i=f(x_i), \quad i=1,\ldots M.
\end{equation}

We distinguish between two assumptions on the data generating process:
\subsection{Ergodic sampling}
In the first scenario we assume the measure $\nu$ is \emph{ergodic} in which case we assume that the data~(\ref{eq:data}) lie on a single trajectory, i.e., $x_{i+1} = T(x_i)$. In this case, by Birkhoff's ergodic theorem we infer, for $\nu$-almost all initial conditions $x_1$:
\begin{equation}\label{eq:ergod}
m_k = \int_\Xc (f\circ T^k)\bar f\,d\nu = \lim_{M\to\infty}\frac{1}{M-k}\sum_{i=1}^{M-k} (f\circ T^k(x_i))\bar f(x_i) = \lim_{M\to\infty}\frac{1}{M-k}\sum_{i=1}^{M-k}y_{i+k}\bar y_i
\end{equation}
and therefore for large $M$
\begin{equation}\label{eq:moment:approx}
m_k \approx \frac{1}{M-k}\sum_{i=1}^{M-k}y_{i+k}\bar y_i.
\end{equation}

In practice, the ergodic measure $\nu$ will often be the so called physical measure (see, e.g.,~\cite{froyland1998approximating}) in which case the ergodic theorem holds also for Lebesgue almost all initial conditions and hence an initial condition sampled at random from a uniform distribution over $\Xc$ will satisfy~(\ref{eq:ergod}) with probability one.

The reader is referred to~\cite{kachurovskii1996rate} for a survey on the rate of convergence of the ergodic averages~(\ref{eq:moment:approx}) as well as to \cite{das2018super} for a weighted scheme exhibiting an accelerated rate of convergence for quasi-periodic dynamical systems.

\subsection{IID sampling}
In the second scenario we assume that the samples~(\ref{eq:data}) are independently drawn at random from the distribution of $\nu$. In this case  the moments $m_k$ satisfy (\ref{eq:ergod}) and (\ref{eq:moment:approx}) by virtue of the law of large numbers. This sampling scheme requires the knowledge of the preserved measure $\nu$ and hence is less relevant in applications than the ergodic sampling. In addition, this sampling scheme requires a trajectory of length $N+1$ to be computed for each sampled initial condition in order to compute the estimates of $(m_k)_{k=0}^N$, making this sampling scheme also more computationally expensive.

\section{Reconstruction of $\mu_f$}
In this section we show how to approximately reconstruct $\mu_f$ using the truncated moment sequence $(m_k)_{k=0}^N$ computed from data in Section~\ref{sec:momComp}. For the remainder of this section we suppress the dependency on the observable $f$ and write $\mu$ for $\mu_f$. The measure $\mu$ is supported on a subset of the unit circle $\Tb$ and hence we can\footnote{\label{foot:symmetry} There is an ambiguity when representing a measure on $\Tb$ by a measure on $[0,1]$ if there is an atom at $1=e^{i2\pi}$. Throughout this paper we shall assume that the representation on $[0,1]$ satisfies $\mu(\{0\}) = \mu(\{1\})$, which eliminates the ambiguity. Therefore if there is an atom at $1=e^{i2\pi}$ with weight $w > 0$, then $\mu(\{0\}) = \mu(\{1\}) = w/2$.} regard it as a measure on $[0,1]$. We shall use the symbol $\mu$ both for a measure on $\Tb$ and its representation on $[0,1]$, the distinction always being clear from the context. Hence, the moments of $\mu$~(\ref{eq:mom_mu_C}) become the Fourier coefficients
\begin{equation}\label{eq:momFour}
m_k = \int_{[0,1]} e^{i2\pi \theta k}\, d\mu(\theta),\quad k\in\Zb.
\end{equation}


Regarding $\mu$ as a measure on $[0,1]$, the Lebesgue decomposition of $\mu$ (Eq.~(\ref{eq:lebDec_muf}))  reads
\begin{equation}\label{eq:lebDecomp}
\mu = \mu_{\mr{at}} + \mu_{\mr{ac}} + \mu_{\mr{sc}} \,,
\end{equation}
where the atomic part can be written as (with the symmetry convention $\mu_{\mr{at}}(\{0\}) = \mu_{\mr{at}}(\{1\})$; see Footnote~\ref{foot:symmetry})
\[
\mu_{\mr{at}}=\sum_{j=1}^{n_{\mr{at}}} w_j \delta_{\theta_j} 
\]
with $w_j>0$ and $n_{\mr{at}} \in \Nb \cup \{\infty\}$, and the absolutely continuous part as
\[
d\mu_{\mr{ac}} = \rho\, d\theta
\]
with the density $\rho \in L_1([0,1],d\theta)$. In what follows we describe a procedure to recover the weights $w_i$ and locations $\theta_i$ as well as the density $\rho$ from the moment data, even in the presence of the singular continuous part $\mu_{\mr{sc}}$ (Section~\ref{sec:CD}). We also show (in Section~\ref{sec:weakApprox}) how to construct approximations $\mu_N$ of $\mu$ that converge weakly to $\mu$ as $N$ tends to infinity, even in the presence of $\mu_{\mr{sc}}$. In particular, denoting
\[
F(t) := \mu ([0,t]),\quad F_{N}(t) := \mu_N ([0,t])
\]
the right-continuous (cumulative) distribution functions of $\mu$ and $\mu_N$, we will construct the approximations $F_N$ such that
\begin{equation}\label{eq:distFunConv_temp}
\lim_{N\to\infty} F_N(t) = F(t)
\end{equation}
at all points of continuity of $F$. The weak approximations $\mu_N$ will be constructed in two different ways, one purely atomic and one purely absolutely continuous. In addition to~(\ref{eq:distFunConv_temp}), the absolutely continuous approximations will satisfy
 \[
\lim_{N\to\infty} F_N(t) = \frac{F(t)+F^-(t)}{2}, \quad t\in(0,1),
\]
where  $F^{-}(t) = \mu([0,t))$ denotes the left limit of $F$ at $t$.

\subsection{Christoffel-Darboux kernel}\label{sec:CD}
The main tool we use for the recovery of the atomic and AC parts is the classical Christoffel-Darboux (CD) kernel defined for each $N \in \Nb$ and each $z\in\Cb$, $s\in \Cb$ by
\[
K_N(z,s) = \sum_{i=0}^N \bar \varphi_i(z) \varphi_i(s),
\]
where $\varphi_i$'s are the orthonormal polynomials associated to $\mu$, i.e., $\deg \varphi_j = j$ and $\int_{\Tb}\varphi_i \bar\varphi_j \, d\mu = 1$ if $i=j$ and zero otherwise. Note that the first $N$ orthonormal polynomials (and hence the kernel itself) can be determined from the first $N$ moments $(m_k)_{k=0}^N$ of the measure $\mu$. The following explicit formula is folklore (e.g., \cite[Theorem 2.1]{simon2008christoffel}):
\begin{equation}\label{eq:CDkerel_explicit}
K_N(z,s) = \psi_N(z)^{H} \mathbf{M}_N^{-1}\psi_N(s),
\end{equation}
where $\mathbf{M}_N$ is the positive semidefinite Hermitian Toeplitz \emph{moment matrix}
\begin{equation}\label{eq:momMat}
 \mathbf{M}_N= \int_{\Tb} \psi_N\psi_N^H\,d\mu  =  \begin{bmatrix}m_0 & \bar m_1 & \bar m_2&\ldots& \ldots& \bar m_{N} \\ 
									  m_1 & m_0 &  \bar m_1 & \ddots&  &  \bar m_{N-1}\\
 m_2 &  m_1 &\ddots &\ddots &\ddots  & \vdots	\\
\vdots & \ddots & \ddots & \ddots&   \bar m_1 & \bar m_2			\\
\vdots & & \ddots & m_1 & m_0 & \bar m_1					 	\\
 m_{N} & \ldots & \ldots &  m_2 & m_1 & m_0
			     \end{bmatrix}	,	   
\end{equation}
and
\begin{equation}
\psi_N(z) = \begin{bmatrix}
1,z,z^2,\ldots, z^N
\end{bmatrix}^\top,
\end{equation}
where $A^H$ denotes the Hermitian transpose of a matrix $A$ (i.e., $(A^H)_{i,j} = \bar A_{j,i}$) and $A^\top$ the ordinary transpose.

The expression~(\ref{eq:CDkerel_explicit}) makes it clear that the kernel $K_N$ is well defined only if $\mathbf{M}_N$ is invertible. This is for example the case if the density $\rho$ is strictly positive on a set of positive Lebesgue measure or if the atomic part contains at least $N+1$ atoms. We will not invoke any of these assumptions but rather build the kernel~(\ref{eq:CDkerel_explicit}) using the modified moment sequence
\[
\tilde m_k = \begin{cases}
					m_k +1 & k = 0\\
					m_k & k > 0.
\end{cases}
\]
The moment sequence $(\tilde{m}_k)_{k=0}^\infty$ corresponds to the measure $d\tilde{\mu}=d\mu + 1d\theta$ for which the associated moment matrix is always invertible (since it is the sum of a positive semidefinite matrix and the identity matrix). Whenever constructing any approximations to $\mu$ from $(\tilde{m}_k)_{k=0}^N$ we simply subtract the constant density at the final step of the approximation process. The CD kernel constructed using the modified moment sequence will be denoted by
\begin{equation}\label{eq:CDmodified}
\tilde{K}_N(z,s) = \psi_N(s)^{H} \tilde{\mathbf{M}}^{-1}_N\psi_N(z),
\end{equation}
where $\tilde{\mathbf{M}}_N$ is as in~(\ref{eq:momMat}) with $m_k$ replaced by $\tilde{m}_k$.

Numerical aspects of evaluating the matrix inversion in~(\ref{eq:CDmodified}) are discussed in Section~\ref{sec:CDnumerics}.

\subsubsection{Approximation of $\mu$ using the CD kernel}
When approximating $\mu$ using the CD kernel, relevant for us will be the diagonal values $\tilde{K}_N(z,z)$ on the unit circle, i.e., with $z = e^{i2\pi\theta}$. The following well-known variational characterization (e.g.,~\cite[Proposition 2.16.2]{simonSzegoDescendants}) lies at the heart of the approximation results:
\begin{equation}\label{eq:varCharCD}
\frac{1}{\tilde{K}_N(z_0,z_0)} = \min_{p_N}\Big\{\int_\Cb |p_N(z)|^2 d\tilde \mu(z) \;\;\Big\vert\;\; p_N(z_0)=1,\; \mr{deg}(p_N) \le N \Big\}
\end{equation}
for all $z_0 \in \Cb$, where the optimization in~(\ref{eq:varCharCD}) is over all complex polynomials of degree at most $N$, i.e., over all $p_N \in \mr{span}\{1,z,\ldots, z^N\}$.

The first classical result pertains to the atomic part of the measure $\mu$:
\begin{theorem}[Point spectrum]\label{thm:atoms}
If $\mu$ is a positive measure on $\Tb$ and $(m_k)_{k=0}^N$ its moments defined by~(\ref{eq:mom_mu_C}), then for all $\theta \in [0,1]$
\begin{equation}\label{eq:atomConv}
\lim_{N\to\infty}\Big[ \frac{1}{\tilde{K}_N(e^{i2\pi \theta},e^{i2\pi \theta})} - \frac{1}{N+1} \Big] = \mu(\{e^{i2\pi\theta}\}) = \mu_{\mr{at}}(\{e^{i2\pi\theta}\}).
\end{equation}
\end{theorem}
\begin{proof}
The claim follows from~(\ref{eq:varCharCD}) and the classical result~(e.g., \cite[Theorem 2.2.1]{simon2009orthogonal}) implying that $\tilde{K}_N(e^{i2\pi \theta},e^{i2\pi \theta})^{-1} \to \tilde{\mu}(\{e^{i2\pi\theta}\})$, and from the facts that $\mu_{\mr{at}}(\{e^{i2\pi\theta}\})=\mu(\{e^{i2\pi\theta}\}) = \tilde{\mu}(\{e^{i2\pi\theta}\})$ and $(N+1)^{-1} \to 0$.
\end{proof}

Two remarks are in order. First of all, the factor $(N+1)^{-1}$ does not influence the limit but improves the accuracy of the estimate for finite $N$ by compensating for the effect of adding $1d\theta$ to the measure $\mu$. This will become clear from Theorem~\ref{thm:densConv} below. Second, the limit~(\ref{eq:atomConv}) holds for \emph{all} $\theta \in[0,1]$, not almost all.

Theorem~\ref{thm:atoms} asserts that we can extract the atomic part of the measure $\mu$ by studying the behavior of $\tilde{K}_N^{-1}$ for large $N$. Whenever $\tilde{K}_N^{-1}(e^{i2\pi \theta},e^{i2\pi \theta})$ tends to zero, $\theta$ is not in the support of the atomic part $\mu_{\mr{at}}$; if, on the other hand, $\tilde{K}_N^{-1}(e^{i2\pi \theta},e^{i2\pi \theta})$ converges to a nonzero value, then $\theta$ is in the support of $\mu_{\mr{at}}$ and the weight on the Dirac mass at $\theta$ is equal to the limit of $\tilde{K}_N^{-1}(e^{i2\pi \theta},e^{i2\pi \theta})$.

The following theorem describes how to exploit the CD kernel for recovering the density $\rho$ of the AC part:
\begin{theorem}[Density of AC part]\label{thm:densConv}
If $\mu = \mu_{\mr{at}}+\mu_{\mr{ac}}+\mr{\mu_\mr{sc}}$ with $d\mu_{\mr{ac}} =\rho\,d\theta$ is a positive measure supported on $[0,1]$ and $(m_k)_{k=0}^N$ its Fourier coefficients~(\ref{eq:momFour}), then for Lebesgue almost all $\theta \in [0,1]$
\begin{equation}\label{eq:densConv}
\lim_{N\to\infty}\Big[ \frac{N+1}{\tilde{K}_N(e^{i2\pi \theta},e^{i2\pi \theta})} - 1 \Big] = \rho(\theta).
\end{equation}
\end{theorem}
\begin{proof}
The result follows from Theorem~1 of~\cite{mate1991szego} in view of the fact that the density of the AC part of $\tilde{\mu}$ is equal to $\tilde{\rho} = \rho + 1$ and satisfies Szeg\" o's integrability condition $\int_0^1 \log(\tilde{\rho}(\theta))\,d\theta > -\infty$ for any nonnegative $\rho$ and in view of~(\ref{eq:varCharCD}).
Under the integrability condition, one has $(N+1)\tilde{K}_N(e^{i2\pi \theta},e^{i2\pi \theta})^{-1} \to \tilde{\rho}(\theta)$ Lebesgue almost everywhere, which is equivalent to~(\ref{eq:densConv}).
\end{proof}

\begin{remark}\label{rem:convNK}
Theorem~\ref{thm:densConv} provides a recovering method of the density $\rho$. We remark that~(\ref{eq:densConv}) holds for Lebesgue almost every $\theta \in [0,1]$. Precise characterization of when~(\ref{eq:densConv}) holds is given in~Theorem~4 of~\cite{mate1991szego}. This theorem in particular implies that~(\ref{eq:densConv}) holds if $\theta$ is a Lebesgue point of the density $\rho$ and lies outside of the support of the singular parts $\mu_{\mr{at}}$ and $\mu_{\mr{sc}}$. On the other hand, Theorem~\ref{thm:atoms} implies that the limit in~(\ref{eq:densConv}) is infinite if $\theta$ lies in the support of the $\mu_{\mr{at}}$. As far as our knowledge goes, a precise characterization of the limiting behavior for $\theta$ belonging to the support of $\mu_{\mr{sc}}$ remains an open question.   Consult \cite{simonSzegoDescendants,totik2000asymptotics} for the state of the art on this matter.
\end{remark}


\begin{example}[Measure with AC part + atoms]\label{ex:ac+at}
We demonstrate the approximations using the CD kernel on the measure $\mu = \mu_{\mr{at}} + \mu_{\mr{ac}}$ with $\mu_{\mr{at}} = 0.05\delta_{0} + 0.05\delta_{1} + 0.1\delta_{0.2}+0.1\delta_{0.6} + 0.1\delta_{0.8}$ and with the density of the AC part $\rho(\theta) = 4I_{[0.3,0.7]}(\theta)$, where $I_{[0.3,0.7]}$ denotes the indicator function of the interval $[0.3,0.7]$. For ease  of notation we set
\begin{equation}\label{eq:zeta_N}
\zeta_N(\theta) :=\frac{N+1}{\tilde{K}_N(e^{i2\pi\theta},e^{i2\pi\theta})}  - 1.
\end{equation}
Figure~\ref{fig:artifNK} depicts the approximation $\zeta_N$, i.e., the approximation to the AC part of Theorem~\ref{thm:densConv},  whereas  Figure~\ref{fig:artif1K} depicts the $\zeta_N / (N+1) $, i.e., the approximation to the atomic part of Theorem~\ref{thm:atoms}. Notice in particular the different scale of the vertical axis between the two figures. Notice the rapid convergence of $\zeta_N(\theta)$ to $\rho(\theta)$ for $\theta$  outside the support of $\mu_{\mr{at}}$. Notice also the rapid convergence of $\zeta_N(\theta)/(N+1)$ to $\mu(\{\theta\})= \mu_{\mr{at}}(\{\theta\})$ for all $\theta \in [0,1]$. Notice in particular that because of periodicity, one has  $\zeta_N(0) / (N+1) = \zeta_N(1) / (N+1)$ and this value converges to $\mu(\{e^{i2\pi 0}\})$ which in our case is $0.05+0.05 =0.1$.

\begin{figure*}[th]
\begin{picture}(140,160)
\put(20,0){\includegraphics[width=70mm]{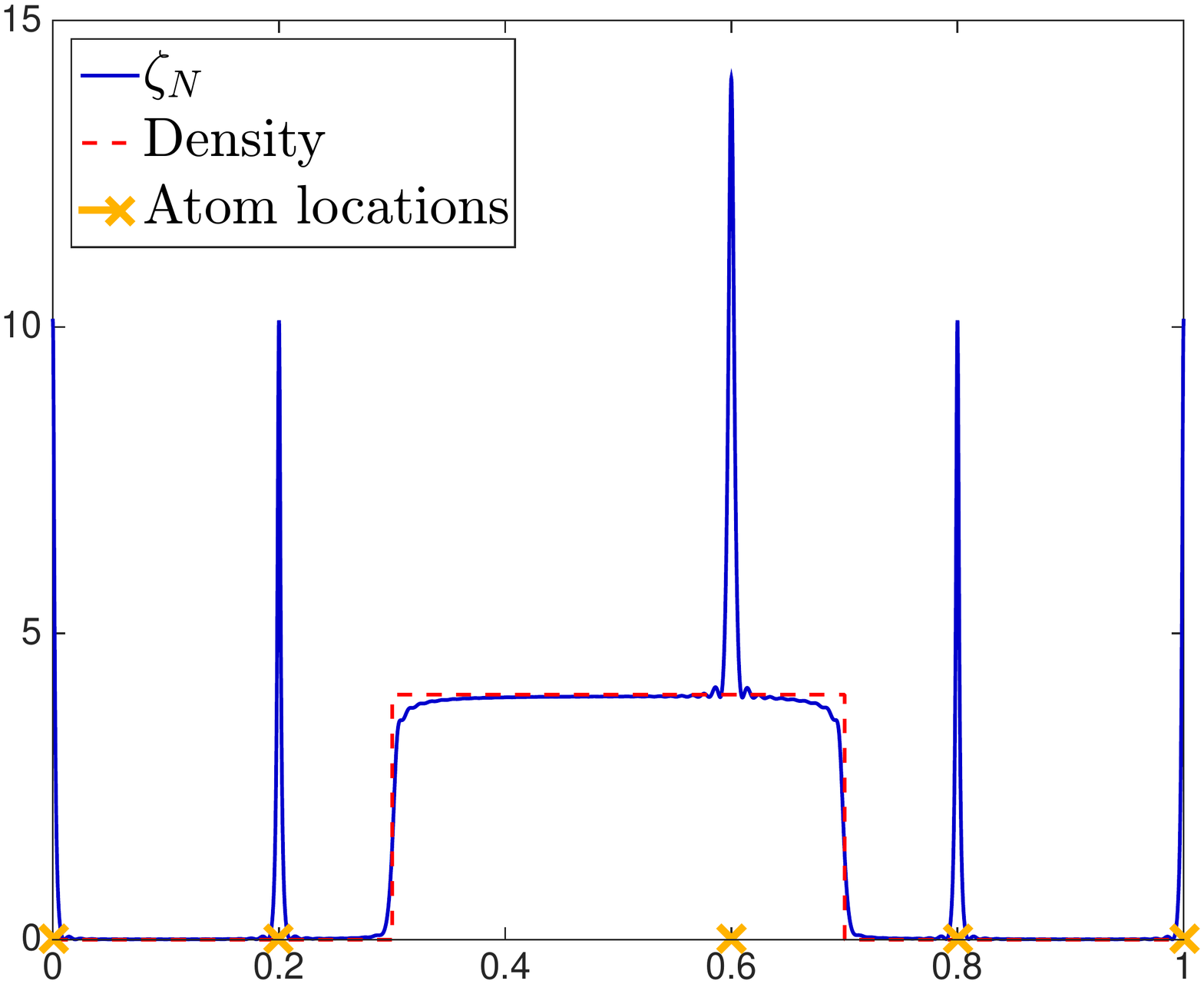}}
\put(235,0){\includegraphics[width=70mm]{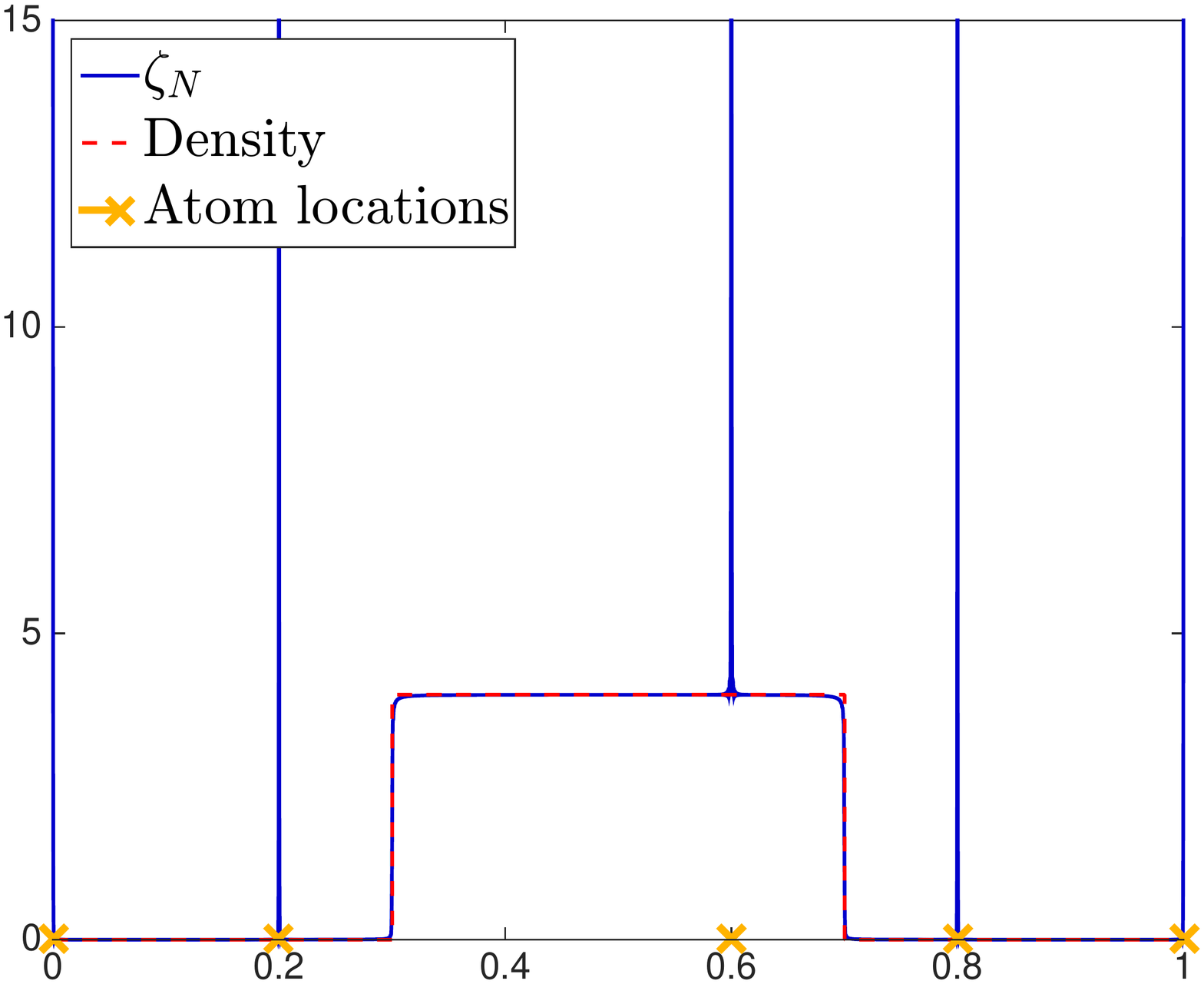}}

\put(120,0){\footnotesize $\theta$}
\put(340,0){\footnotesize $\theta$}
\put(150,125){\footnotesize $N=100$}
\put(364,125){\footnotesize $N=1000$}

\end{picture}
\caption{\figCaptionSize Example~\ref{ex:ac+at} -- Approximation of the absolutely continuous part of the spectrum using the CD kernel (Theorem~\ref{thm:densConv}).}
\label{fig:artifNK}
\end{figure*}

\begin{figure*}[th]
\begin{picture}(140,160)
\put(20,0){\includegraphics[width=70mm]{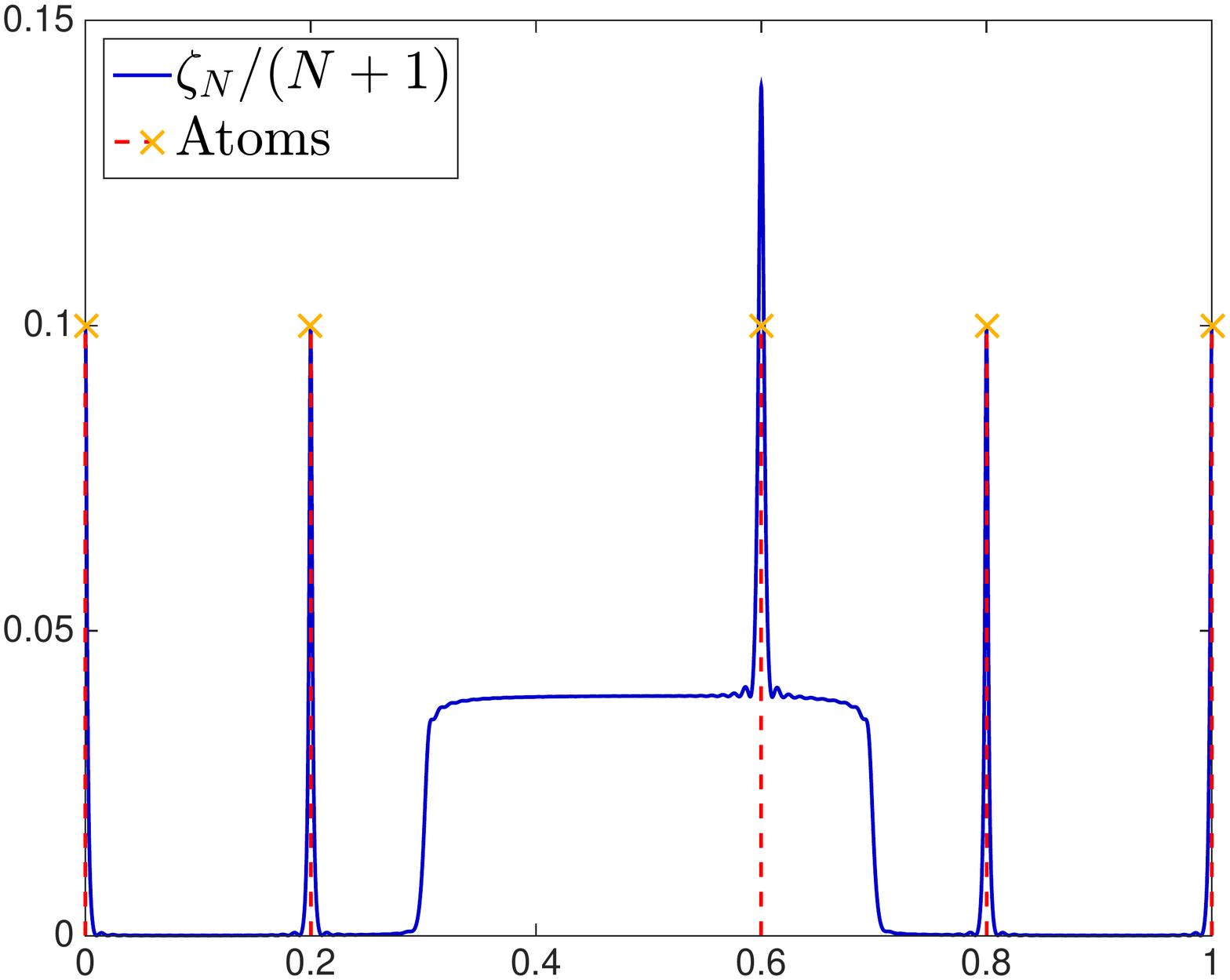}}
\put(235,0){\includegraphics[width=70mm]{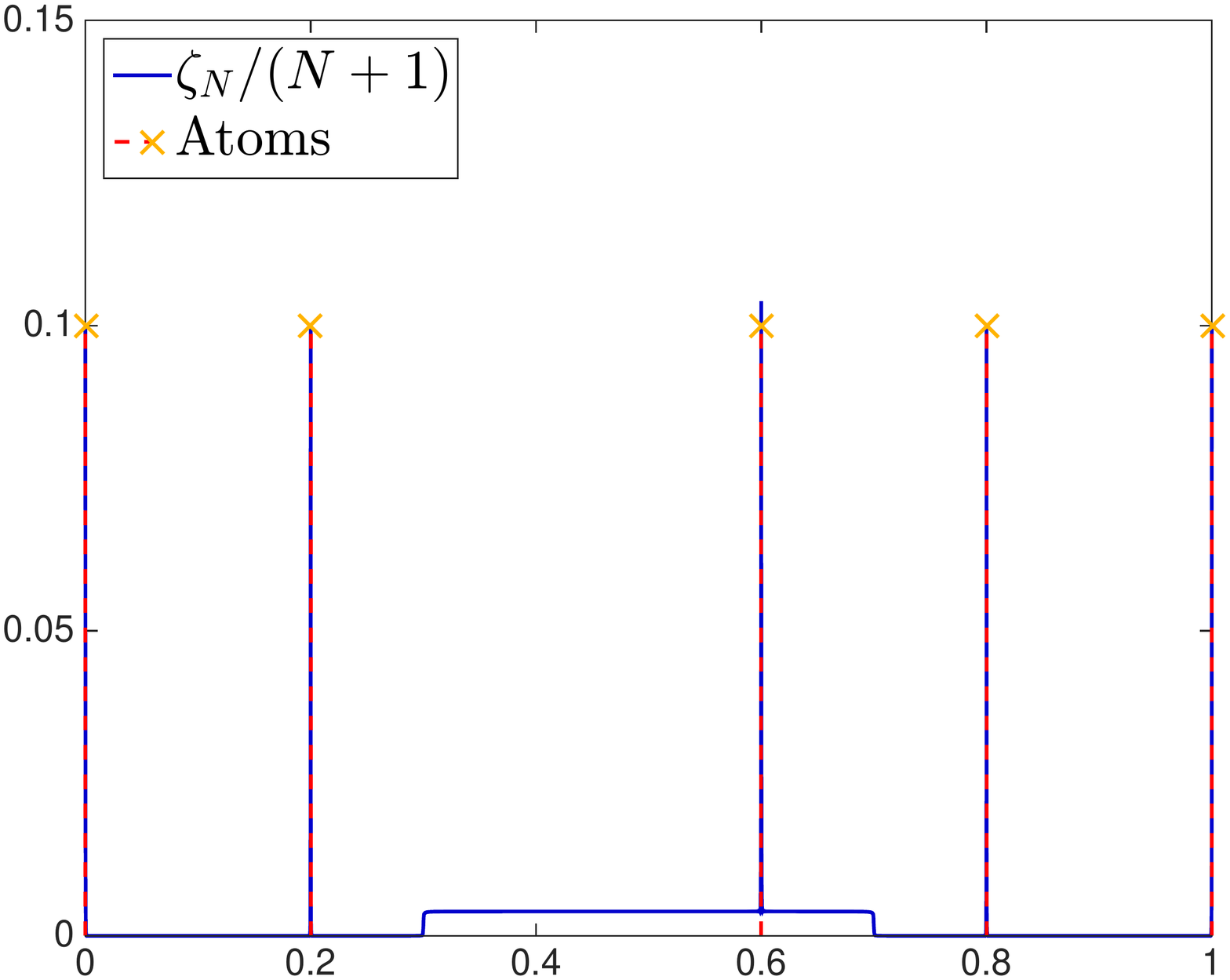}}

\put(120,0){\footnotesize $\theta$}
\put(340,0){\footnotesize $\theta$}
\put(150,125){\footnotesize $N=100$}
\put(364,125){\footnotesize $N=1000$}

\end{picture}
\caption{\figCaptionSize Example~\ref{ex:ac+at} -- Approximation of the atomic part of the spectrum using the CD kernel (Theorem~\ref{thm:atoms}).}
\label{fig:artif1K}
\end{figure*}

\end{example}

\begin{example}[Measure with SC part]\label{ex:sc_cd}
In this example we investigate the effect of the singular continuous part~$\mu_{\mr{sc}}$. We use the same atomic and AC parts as in the previous example, only add the singular continuous part $\mu_{\mr{sc}}$ equal to the Cantor measure on $[0,1]$ whose moments
\begin{equation}\label{eq:cantorMom}
\int_{[0,1]} e^{i2\pi \theta k} \,d\mu_{\mr{sc}}(\theta) = e^{i\pi\theta k}\prod_{n=1}^\infty \cos\Big(\frac{2\pi\theta k}{3^n}\Big)
\end{equation}
are derived readily from its characteristic function known in closed form (see, e.g., \cite{lukacs1970characteristic}). Figure~\ref{fig:artifNK_cantor} and Figure~\ref{fig:artif1K_cantor}  are analogous to Figures~\ref{fig:artifNK} and~\ref{fig:artif1K} for this situation\footnote{Note that the support of the Cantor measure depicted in Figures~\ref{fig:artifNK_cantor} and~\ref{fig:artif1K_cantor}  is a rough numerical approximation since the true support is the fractal Cantor set.}. Extraction of the atomic part is not particularly hampered as the conclusion of Theorem~\ref{thm:atoms} holds for all $\theta \in [0,1]$.  This is witnessed by Figure~\ref{fig:artif1K_cantor} where we observe $\zeta_N(\theta)/(N+1)$ converging to zero whenever $\theta$ is not in the support of the atomic part. On the other hand, extraction of the density $\rho$ is more difficult in this case since the conclusion of Theorem~\ref{thm:densConv} holds only almost everywhere (see Remark~\ref{rem:convNK}). Nevertheless, the behavior of $\zeta_N(\theta)$ for different values of $N$ is still a guideline for disentangling the contributions of the AC and SC parts. Whenever $\theta$ lies outside of the support of $\mu_{\mr{sc}}$ and $\mu_{\mr{at}}$, we expect convergence to $\rho(\theta)$ whereas otherwise we expect a divergent behavior. This is witnessed by Figure~\ref{fig:artifNK_cantor}, where we observe this behavior; notice in particular the different rate of divergence for $\theta$ belonging to the support of the SC part versus for $\theta$ being in the support of the atomic part (see Remark~\ref{rem:convNK}). We shall re-investigate this example using the tools described in the next section which will further help disentangling the AC and SC parts.

\begin{figure*}[th]
\begin{picture}(140,160)
\put(20,0){\includegraphics[width=70mm]{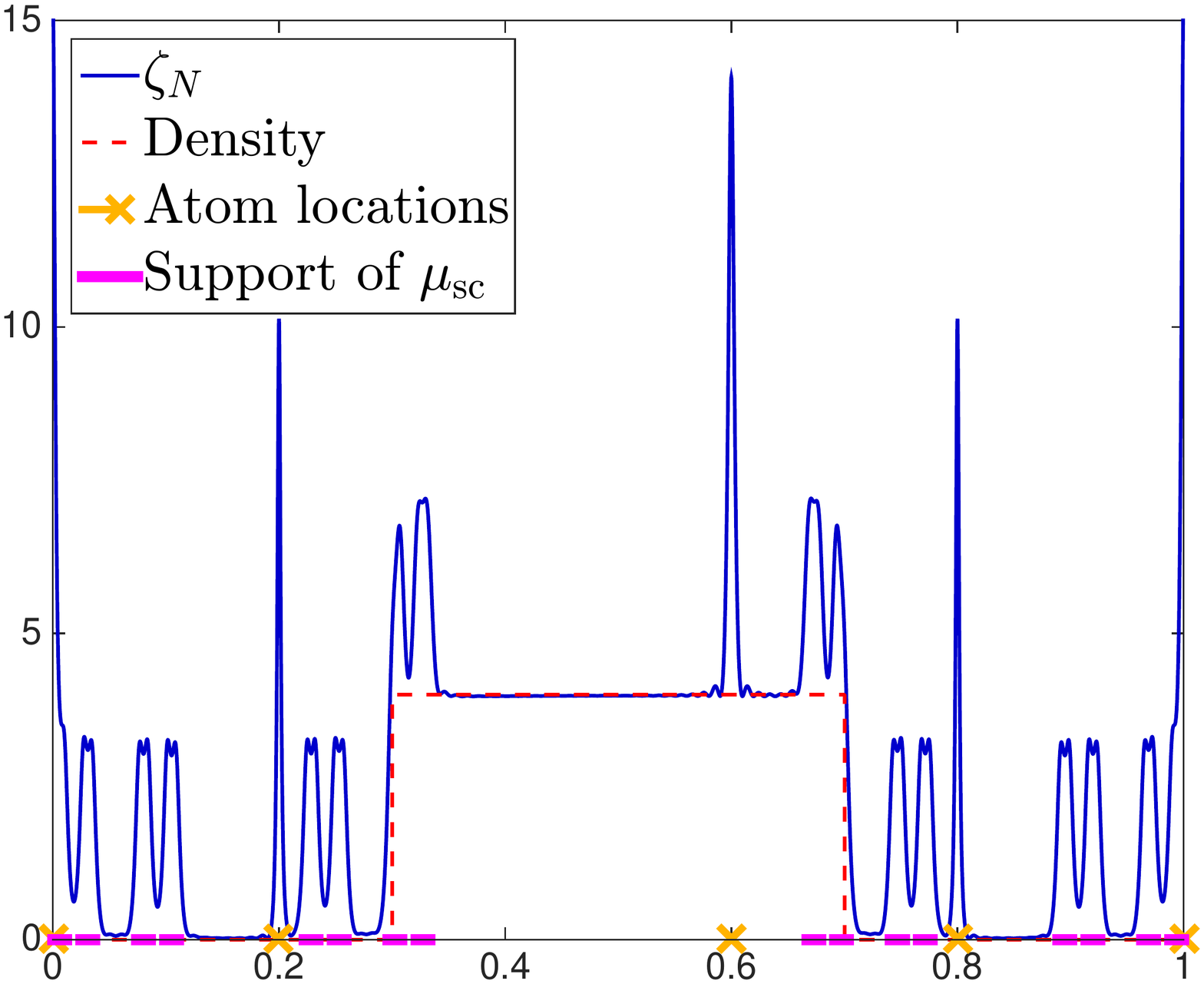}}
\put(235,0){\includegraphics[width=70mm]{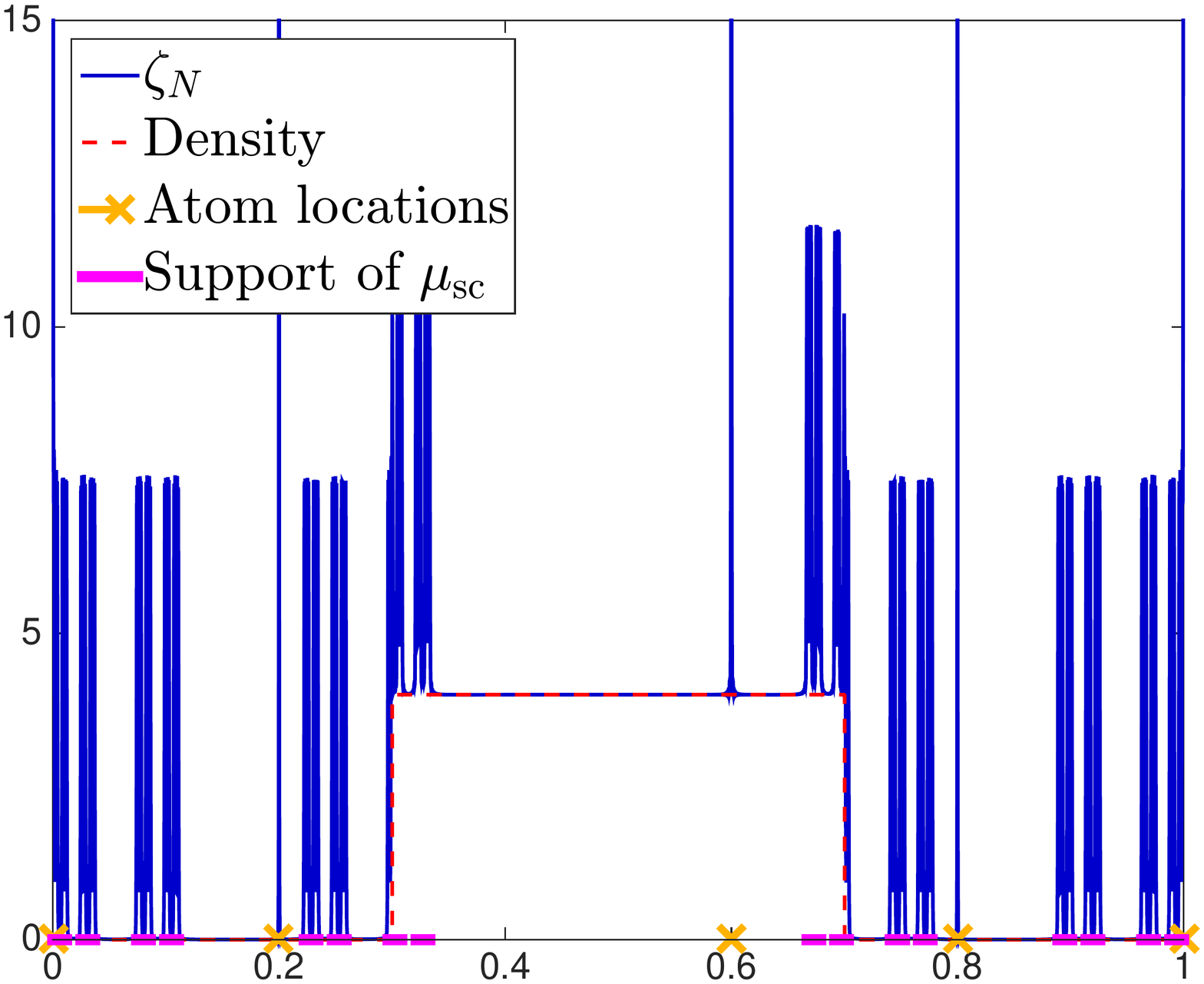}}

\put(120,0){\footnotesize $\theta$}
\put(340,0){\footnotesize $\theta$}
\put(150,125){\footnotesize $N=100$}
\put(364,125){\footnotesize $N=1000$}

\end{picture}
\caption{\figCaptionSize Example~\ref{ex:sc_cd} --  Approximation of the absolutely continuous part of the spectrum using the CD kernel (Theorem~\ref{thm:densConv}).}
\label{fig:artifNK_cantor}
\end{figure*}

\begin{figure*}[th]
\begin{picture}(140,160)
\put(20,0){\includegraphics[width=70mm]{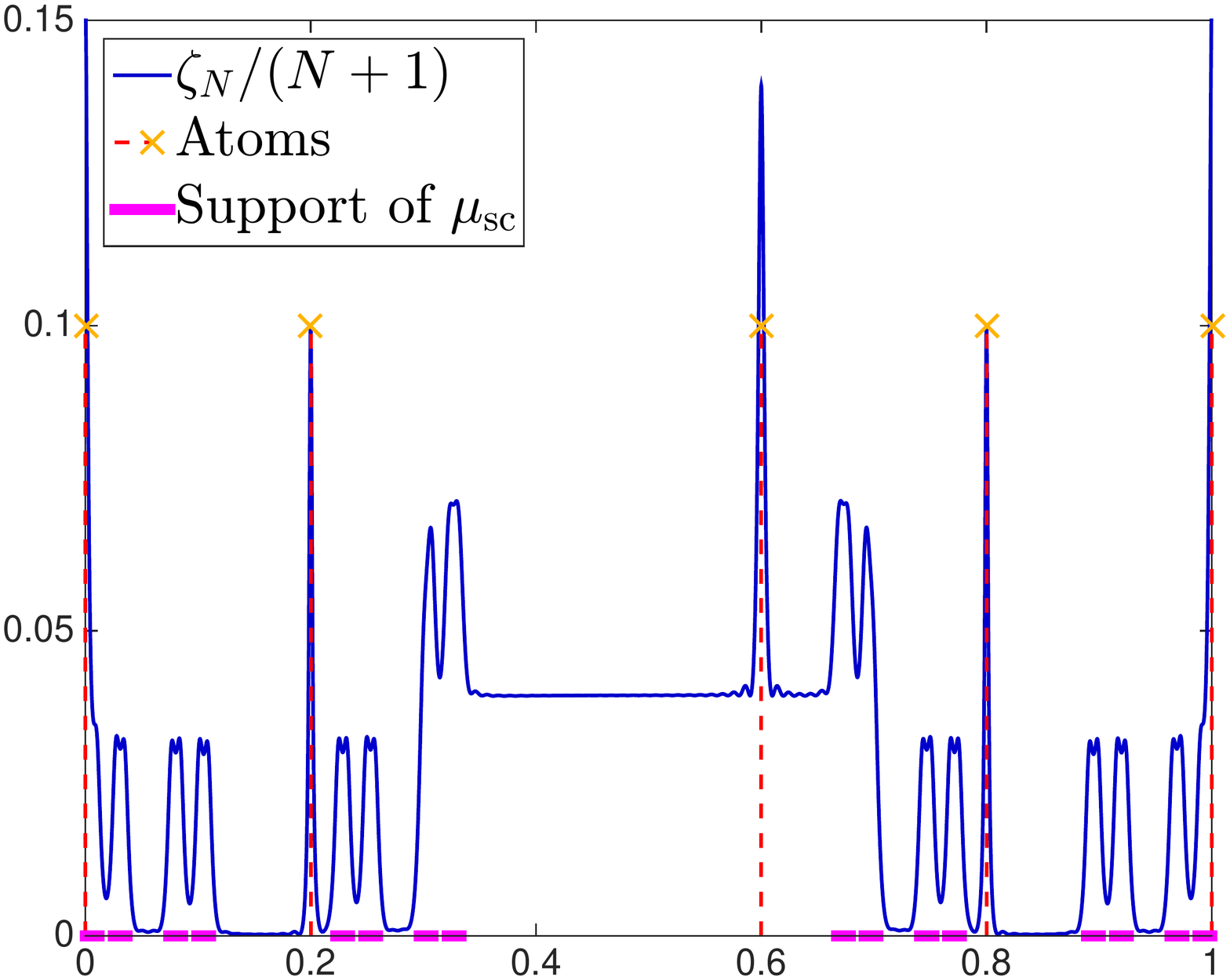}}
\put(235,0){\includegraphics[width=70mm]{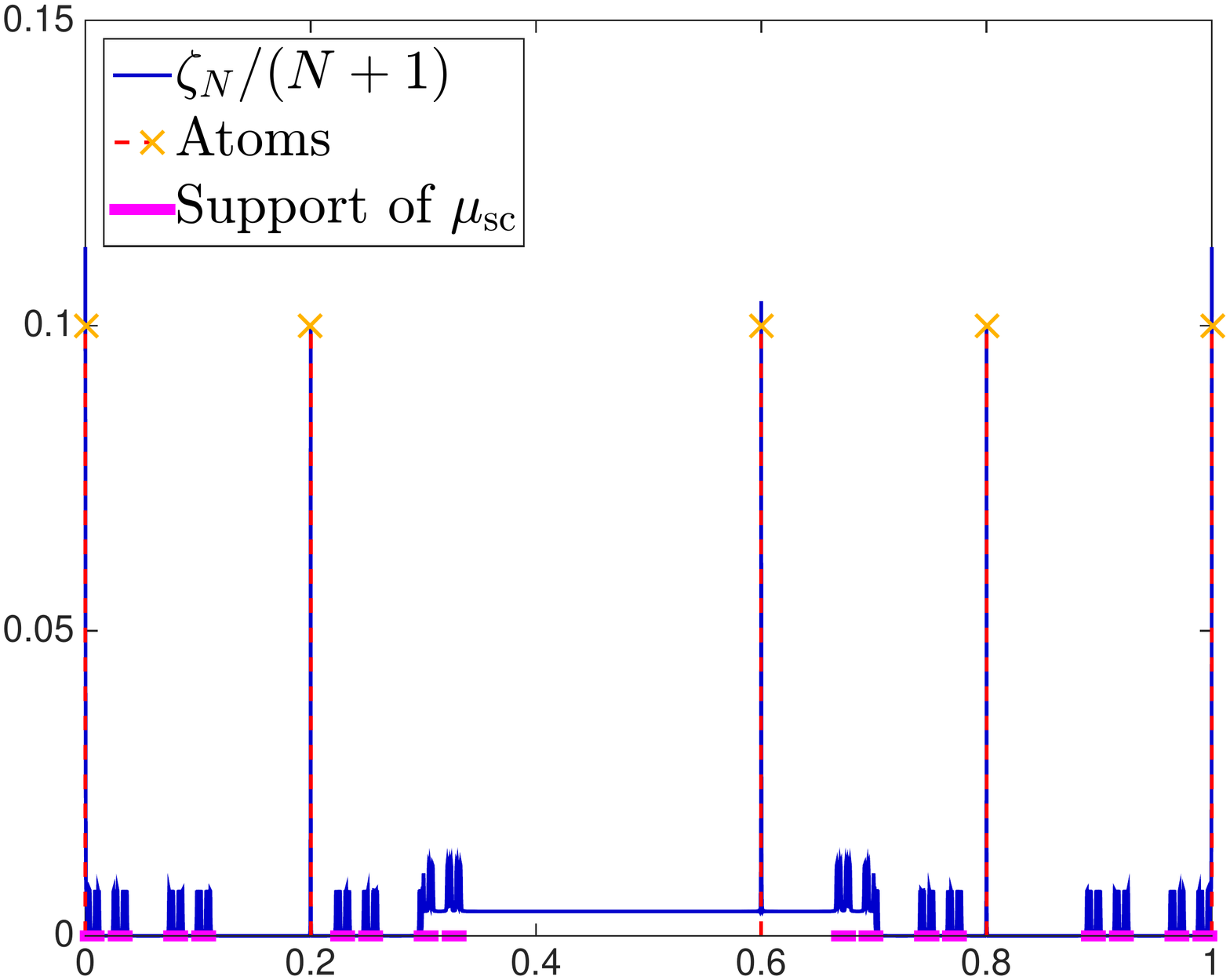}}

\put(120,0){\footnotesize $\theta$}
\put(340,0){\footnotesize $\theta$}
\put(150,125){\footnotesize $N=100$}
\put(364,125){\footnotesize $N=1000$}

\end{picture}
\caption{\figCaptionSize Example~\ref{ex:sc_cd} -- Approximation of the atomic part of the spectrum using the CD kernel (Theorem~\ref{thm:atoms}).}
\label{fig:artif1K_cantor}
\end{figure*}

\end{example}

\subsubsection{Numerical computation of the CD kernel}\label{sec:CDnumerics}
In this section we discuss numerical aspects of evaluation of the modified CD kernel~(\ref{eq:CDmodified}). Fortunately, the matrix $\tilde{\mathbf{M}}_N$ is positive definite Hermitian Toeplitz matrix which can be inverted in $O(N^2)$ or even $O(N\log^2(N))$ floating point operations in a numerically stable way, as opposed to $O(N^3)$ for a general matrix. See, e.g.,~\cite{trench1964algorithm,strang1986proposal,brent1988old}. Hence the CD kernel can be accurately evaluated even for a very large number of moments (e.g., $N \gg 1000$). Whether or not one should precompute the inverse or rather compute the Cholesky factorization of $\tilde{\mathbf{M}}_N$ (also with $O(N^2)$ complexity~\cite{stewart1997cholesky}) depends on the number of point evaluations of $\tilde{K}_N(z,s)$ one expects to carry out. In our applications we required evaluation of $\tilde{K}_N(e^{i2\pi\theta},e^{i2\pi\theta})$ on a very fine grid of $\theta$ in which case direct inversion was preferable.

If one requires also direct access to the orthonormal polynomials $\tilde\varphi_i$ of $\tilde{\mu}$, then Cholesky factorization $\tilde{\mathbf{M}}_N = L^H  L$ is preferred since we have
\begin{equation}\label{eq:cholesky}
\begin{bmatrix}\tilde\varphi_0(z) \\ \tilde\varphi_1(z)\\ \vdots \\ \tilde\varphi_N(z)   \end{bmatrix} =  L^{-H} \begin{bmatrix} 1 \\ z \\ \vdots \\z^N  \end{bmatrix},
\end{equation}
where we used the fact that $\tilde{\mathbf{M}}_N = \int_{\Tb} \psi_N\psi_N^H\,d\tilde{\mu}$ (see~(\ref{eq:momMat})).

The orthonormal polynomials~(\ref{eq:cholesky}) have a number of interesting properties, one of which being the relation of their zeros to the eigenvalues of the dynamic mode decomposition approximation of the Koopman operator $U$. This is explored in Section~\ref{sec:dmd_relation}.

\subsection{Weak approximation of $\mu$}\label{sec:weakApprox}
In Section~\ref{sec:CD} we showed how the atomic and absolutely continuous parts of $\mu$ can be recovered. In this section we show how to construct approximations $\mu_N$ to $\mu$  such that $\mu_N$ converges weakly\footnote{A sequence of measures $(\mu_N)_{N=1}^\infty$ converges to $\mu$ weakly if $\int f \, d\mu_N \to \int f\,d\mu $ for all continuous bounded functions.} to $
\mu$. This is equivalent to saying that the distribution function $F_N(t) = \mu_N([0,t])$ of $\mu_N$ converges pointwise to $F = \mu([0,t])$ at every point of continuity of~$F$. Since the support of our measures is compact, all continuous test functions are automatically bounded, hence the notion of weak convergence we refer to throughout this article coincides with the classical weak-$\ast$ convergence.

We describe two  methods of constructing $\mu_N$. The first approximation results in the approximation $\mu_N$ being absolutely continuous with respect to the Lebesgue measure with a smooth density $\rho_N$, whereas the second approximation results in a purely atomic approximation  $\mu_N$.

\subsubsection{Ces\`aro sums}
The first method uses the classical Ces\`aro summation. Given moments $(m_k)_{k=-N}^N$, with $m_{-k} = \bar m_k$, the $n$-th partial sum of the Fourier series is defined by
\begin{equation}\label{eq:Sndef}
S_n(\theta) := \sum_{k=-n}^n e^{i 2\pi \theta k} \bar m_k.
\end{equation}
Then we define the density $\rho_N^{\mr{CS}}$ to be the Ces\`aro sum
\begin{equation}\label{eq:cesaro_means}
\rho_N^{\mr{CS}}(\theta) := \frac{1}{N+1}\sum_{n=0}^N S_n(\theta)
\end{equation}
and we set for any $0\le a\le b\le 1$
\begin{equation}\label{eq:mu_Nf_cesaro}
\mu_N^{\mr{CS}}([a,b]) := \int_{a}^b \rho_N^{\mr{CS}}(\theta)\,d\theta.
\end{equation}
Hence the cumulative distribution function of $\mu_{N}^{\mr{CS}}$ is
\begin{equation}\label{eq:F_N_cesaro}
F_N^{\mr{CS}}(t) = \int_0^t \rho_N^{\mr{CS}}(\theta)\,d\theta.
\end{equation}

\begin{theorem}\label{thm:cdf_conv_CS}
Let $\rho_N^{\mr{CS}}$ be the Ces\`aro sum~(\ref{eq:cesaro_means}) associated to the moment sequence~(\ref{eq:momFour}) of a positive measure $\mu$ on $[0,1]$ and let $\mu_N^{\mr{CS}}$ and $F_N^{\mr{CS}}$ be defined by~(\ref{eq:mu_Nf_cesaro}) and (\ref{eq:F_N_cesaro}). Then $F_N^{\mr{CS}}(0) = 0$, $F_N^{\mr{CS}}(1) = \mu([0,1])$ and
\[
\lim_{N\to\infty} F_N^{\mr{CS}}(t) = \frac{F(t)+F^-(t)}{2},\quad t\in (0,1),
\]
where $F = \mu([0,t])$ is the right-continuous distribution function of $\mu$ and $F^-(t) = \mu([0,t))$ its left limit at $t$. In particular, $\mu_N^{\mr{CS}}$ converges weakly to $\mu$.
\end{theorem}
\begin{proof}
Using~(\ref{eq:Sndef}), (\ref{eq:cesaro_means}) and (\ref{eq:momFour}), we get
\begin{align*}
\rho_N^{\mr{CS}}(x) &=  \frac{1}{N+1}\sum_{n=0}^N S_n(x)= \int_{[0,1]} \frac{1}{N+1}\sum_{n=0}^N \sum_{k=-n}^n e^{i2\pi k (x-\theta) }\,d\mu(\theta)\\ &= \int_{[0,1]} G_{N+1}(2\pi(x-\theta))\,d\mu(\theta)
\end{align*}
where
\begin{equation}\label{eq:fejerKernel}
G_N(x)  =  \frac{1}{N}\sum_{n=0}^{N-1} \sum_{k=-n}^n e^{i k x } = \frac{1}{N}\Big(\frac{\sin \frac{Nx}{2}}{\sin\frac{x}{2}}\Big)^2
\end{equation}
is the $N$-th Fej\'er kernel. Therefore
\begin{align*}
F_N^{\mr{CS}}(t) &=\int_0^1 I_{[0,t]} \rho_N^{\mr{CS}}(x) dx= \int_0^1 \int_{[0,1]}  I_{[0,t]}G_{N+1}(2\pi(x-\theta))\,d\mu(\theta) dx\\ &= \int_{[0,1]}\int_0^1   I_{[0,t]}G_{N+1}(2\pi(x-\theta))\, dx\, d\mu(\theta) = \int_{[0,1]}  g_{N,t}(\theta)\,d\mu(\theta),
\end{align*}
where $g_{N,t}(\theta) := (I_{[0,t]}\ast G_{N+1})(\theta)$ is the convolution of the Fej\'er kernel with the indicator function of the interval $[0,t]$. Here we used the Fubini theorem and the symmetry of the Fej\'er kernel. By basic properties of the Fej\'er kernel~(e.g.,~\cite[Ch. III, p. 88,\,89]{zygmundTrigo}) we have for any $t\in (0,1)$
\[
g_{\infty,t}(\theta) := \lim_{N\to\infty}g_{N,t}(\theta) = \begin{cases} 1/2 & \theta = 0 \\
												   1 & \theta \in (0,t) \\
												   1/2 & \theta = t \\
												   	 0 & \theta \in (t,1) \\
												   	 1/2 & \theta = 1										  \end{cases}
\]
and $g_{\infty,t}(\theta) = 0$ for all $\theta\in[0,1]$ for $t = 0$ and $g_{\infty,t}(\theta) = 1$ for all $\theta\in[0,1]$ for $t = 1$. Since $|I_{[0,t]}| \le 1$, we also have $|g_{N,t}| \le 1$ and hence by the dominated convergence theorem we have 
\[
\lim_{N\to\infty} F_N^{\mr{CS}}(t) = \int_{[0,1]} g_{\infty,t}(\theta) d\mu(\theta) = \begin{cases} \mu([0,t)) + \frac{1}{2}\mu(\{t\}) & t\in (0,1)\\
0 & t = 0\\
\mu([0,1]) & t = 1,
 \end{cases}
\]
where in the first line we used the fact that $\mu(\{0\}) = \mu(\{1\})$ (see Footnote~\ref{foot:symmetry}). Therefore, for any $t\in(0,1)$ \[
\frac{F(t)+F^-(t)}{2} = \frac{\mu([0,t])+ \mu([0,t))}{2} = \mu([0,t)) + \frac{1}{2}\mu(\{t\}) = \lim_{N\to\infty} F_N^{\mr{CS}}(t)
\]
as desired. This also implies convergence of $F_N^{\mr{CS}}(t)$ to $F(t)$ in every point of continuity of $F$ and hence weak convergence of $\mu_N^{\mr{CS}}$ to $\mu$.
\end{proof}

\subsubsection{Quadrature}
In this section we develop purely atomic approximations $\mu_N$ that converge weakly to $\mu$. For this we use the so called quadrature, i.e., we seek a measure $\mu_N$ in the form
\begin{equation}\label{eq:muN_quad}
\mu_N^{\mr{Q}} = \sum_{j=0}^{n_\mr{q}} \gamma_j \delta_{\eta_j},
\end{equation}
where $\gamma_j \ge 0$ and $\eta_j\in [0,1]$ are nonnegative weights and atom locations, respectively. The weights and locations are selected such that
\[
\int_{[0,1]} h_N\, d\mu_N^{\mr{Q}}(\theta) = \int_{[0,1]} h_N\, d\mu(\theta) 
\]
for all $h_N\in \mr{span}\big\{ e^{-i2\pi N\theta},\ldots,e^{i2\pi N\theta}  \big\}$, i.e., for all trigonometric polynomials of degree no more than $N$. This is equivalent to the moment matching condition
\begin{equation}\label{eq:momentMatching}
m_k = \int_{[0,1]} e^{i2\pi k \theta}\,d\mu_N^{\mr{Q}}(\theta) = \sum_{j=0}^{n_{\mr{q}}} \gamma_j e^{i2\pi k \eta_j}\quad \forall\; k\in\{0,\ldots,N\},
\end{equation}
where we use the fact that $m_{-k} = \bar{m}_k$. There is a well developed theory on choosing the locations $\eta_j$ such that nonnegative weights $\gamma_j$ satisfying~(\ref{eq:momentMatching}) exist. These locations are given by the zeros of the so called  paraorthogonal polynomials that can be readily computed from the orthonormal polynomials $\varphi$; see \cite[Ch. 2.15]{simonSzegoDescendants} and \cite[Sec. 7]{jones1989moment}. It is highly relevant for our study that, as $N\to\infty$,  these zeros become uniformly distributed on the unit circle~\cite[Theorem~2.15.4]{simonSzegoDescendants}. Therefore, since we typically work with large $N$, in order to avoid the numerically slightly troublesome process of polynomial root finding, we use directly a uniform grid of points
\begin{equation}\label{eq:grid_points}
\eta_j =  \frac{j}{n_{\mr{q}}},\quad j = 0,\ldots, n_{\mr{q}}
\end{equation}
with $n_{\mr{q}} > N$.
In order to ensure symmetry (see Footnote~\ref{foot:symmetry}), we impose the additional constraint
\[
\gamma_0 = \gamma_{n_{\mr{q}}}.
\]
Given the atom locations~(\ref{eq:grid_points}), the weights $\gamma_j$ are determined by solving the quadratic program (QP)
\begin{equation}\label{opt:quadrature_QP}
\begin{array}{lll}
\varepsilon^\ast_N =  \min\limits_{\gamma_0,\ldots,\gamma_{n_{\mr{q}}}} & \sum_{k=0}^N (m_k - \sum_{j=0}^{n_{\mr{q}}} \gamma_j e^{i2\pi k \eta_j})^2 \\
\hspace{1.35cm} \mathrm{s.t.} &  \gamma_j \ge 0, \quad j = 0,\ldots, n_{\mr{q}} \vspace{0.2mm}\\
&  \gamma_0 = \gamma_{n_\mr{q}} \vspace{0.2mm},
\end{array}
\end{equation}
where in the objective function we penalize the discrepancy in the moment matching condition~(\ref{eq:momentMatching}). In particular, whenever $n_{\mr{q}}$ is such that the optimal value $\varepsilon^\ast_N$ of~(\ref{opt:quadrature_QP}) is zero, the moment matching condition is satisfied exactly. Alternatively, one can solve a linear programming (LP) feasibility problem by casting the moment matching condition~(\ref{eq:momentMatching}) as a constraint. In this case, however, if $n_{\mr{q}}$ is such that exact matching cannot be achieved, then the LP does not provide us an approximation to $\mu_N$ whereas the QP~(\ref{opt:quadrature_QP}) provides an approximation whether or not an exact matching can be achieved. Computational complexity of solving a QP is comparable to that of an LP with many mature solvers existing for both (e.g., MOSEK, GUROBI, CPLEX and many others) and hence we prefer the QP~(\ref{opt:quadrature_QP}).

The following result is immediate.
\begin{theorem}\label{thm:cdf_conv_Q}
Let for each $N$ the measure $\mu_N^{\mr{Q}}$ be of the form~(\ref{eq:muN_quad}) with $(\gamma_j)_{j=0}^{n_{\mr{q}}}$ being an optimal solution to~(\ref{opt:quadrature_QP}). If the associated optimal values $\varepsilon_N^\ast$ converge to zero as $N\to \infty$, then $\mu_N^{\mr{Q}}$ converges weakly to $\mu$. In particular the distribution function of $\mu_N^{\mr{Q}}$
\begin{equation}
F_N^{\mr{Q}}(t) := \mu_N^{\mr{Q}}([0,t])=\sum_{\substack{j=0\\ \eta_j\le t}}^{n_\mr{q}} \gamma_j
\end{equation}
\end{theorem}
converges to the distribution function $F(t)=\mu([0,t])$ in every point of continuity of $F$.

\begin{proof}
The condition $\varepsilon_N^\ast \to 0$ implies pointwise convergence of the moment sequences of $\mu_N$ to the moment sequence of $\mu$, which implies weak convergence of $
\mu_N$ to $\mu$ by compactness of $\Tb$.
\end{proof}

\begin{example}[Distribution functions, quantifying singularity]\label{ex:artif_cdf}
In this example we compare the distribution function obtained using the Ces\`aro sums and using quadrature. We use the same measure as in Example~\ref{ex:sc_cd}, i.e.,  $\mu_{\mr{at}} = 0.05\delta_{0} + 0.05\delta_{1} + 0.1\delta_{0.2}+0.1\delta_{0.6} + 0.1\delta_{0.8}$,  the density of the AC part is $\rho(\theta) = 4I_{[0.3,0.7]}(\theta)$ and the SC part is equal to the Cantor measure with moments given by~(\ref{eq:cantorMom}). Figure~\ref{fig:artif_cdf} depicts the comparison of $F_{N}^{\mr{CS}}$ and $F_{N}^{\mr{Q}}$, where the quadrature weights were obtained using~(\ref{opt:quadrature_QP}) with $n_\mr{q} = 10N$, which was sufficient for exact moment matching (i.e., $\varepsilon_N^\ast = 0$ in~(\ref{opt:quadrature_QP})). We observe a very good accuracy of both methods, with the piecewise constant $F_N^{\mr{Q}}$ being able to capture the fine features of $F$ slightly better than the smooth $F_N^{\mr{CS}}$, especially for $N=1000$. Note that we have also plotted the function
\[
F_{\zeta_N}(t) = \int_0^t \zeta_N(\theta) \, d\theta,
\]
where $\zeta_N$ is defined in~(\ref{eq:zeta_N}). This function is expected to be a good approximation to $F$ only if $\mu$ is absolutely continuous. In our example both singular parts of $\mu$ are non-zero and hence we expect a discrepancy. Note that the overall shape of $F_{\zeta_N}$ is similar to $F$ although the exact magnitude of increase at points of singularity is underestimated by $F_{\zeta_N}$. In fact, the ratio $F(1) / F_{\zeta_N}(1) = m_0 / F_{\zeta_N}(1)$ can be used to quantify the singularity of $\mu$. If $F(1) / F_{\zeta_N}(1)=1$, then  $\mu$ is AC; if $F(1) / F_{\zeta_N}(1) > 1$, then  $\mu$ contains a singular part. In fact, in view of Theorem~\ref{thm:densConv}, one expects $\lim_{N\to\infty} F(1) / F_{\zeta_N}(1) = \mu([0,1])/\mu_{\mr{ac}}([0,1])$. In addition, instead of the global quantity $F(1) / F_{\zeta_N}(1)$, one can also quantify singularity of $\mu$ locally by $(F(b)-F(a)) / (F_{\zeta_N}(b)-F_{\zeta_N}(a))$ for every $0\le a < b \le 1$. Since $F$ is not known, this quantity can be approximated by replacing $F$ with $F_N^{\mr{CS}}$ or $F_N^{\mr{Q}}$. In Figure~\ref{fig:artif_singul_quat} we plot
\begin{equation}\label{eq:singul_indicator}
\Delta_N(t_k) = \frac{F_N^{\mr{CS}}(t_{k+1})-F_N^{\mr{CS}}(t_k)}{F_{\zeta_N}(t_{k+1})-F_{\zeta_N}(t_k)} - 1
\end{equation}
for $t_k = 10^{-3}k, k =0,\ldots 10^3-1$. As expected we observe $\Delta_N(t_i) \approx 0$ when $t_i $ is outside the support of the singular parts and  $\Delta_N(t_i) > 0$ otherwise, although it appears that a large number of moments is needed to get an accurate estimate of the supports, especially the support of the SC part.


\begin{figure*}[h!]
\begin{picture}(140,180)
\put(20,0){\includegraphics[width=80mm]{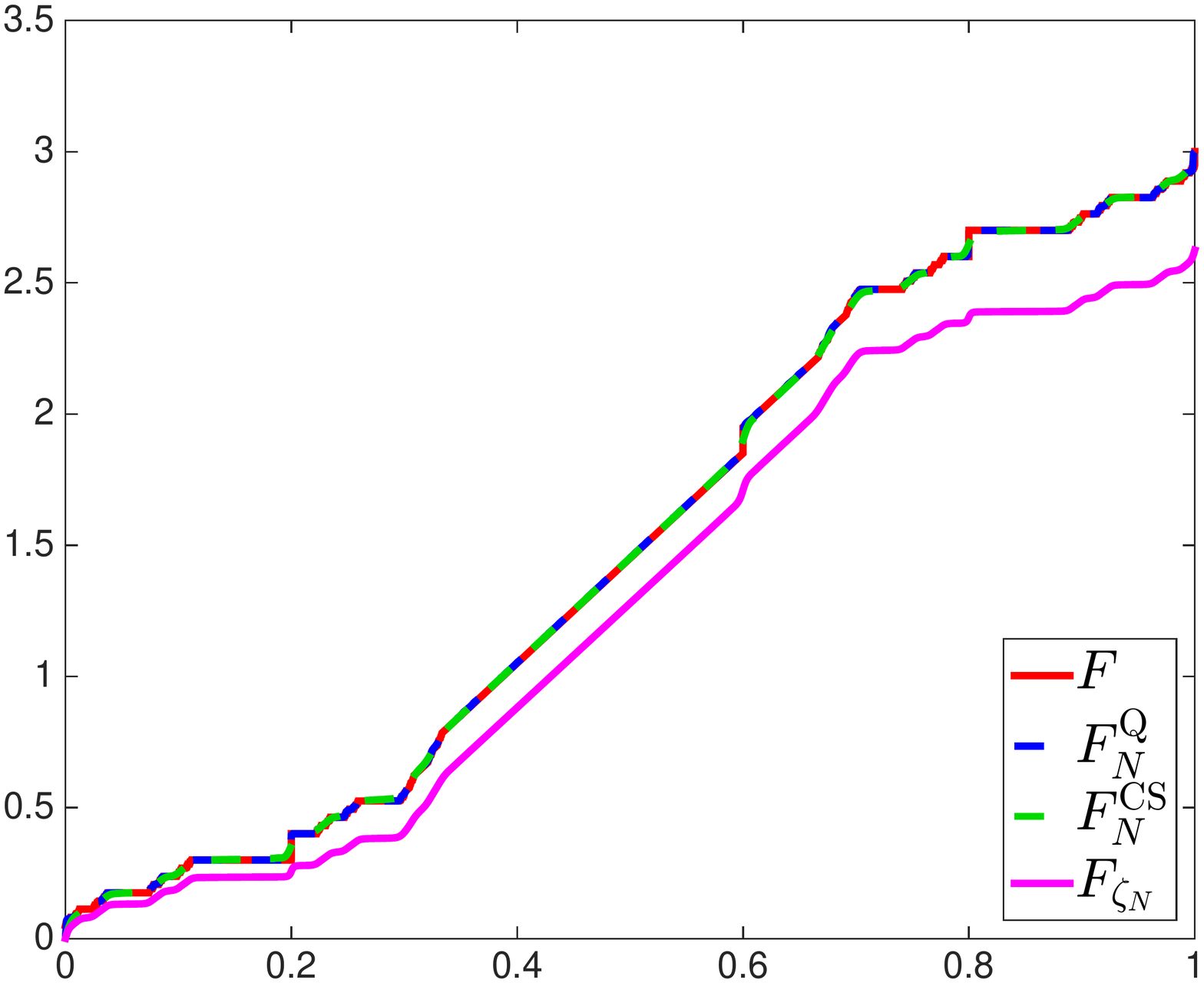}}
\put(235,0){\includegraphics[width=80mm]{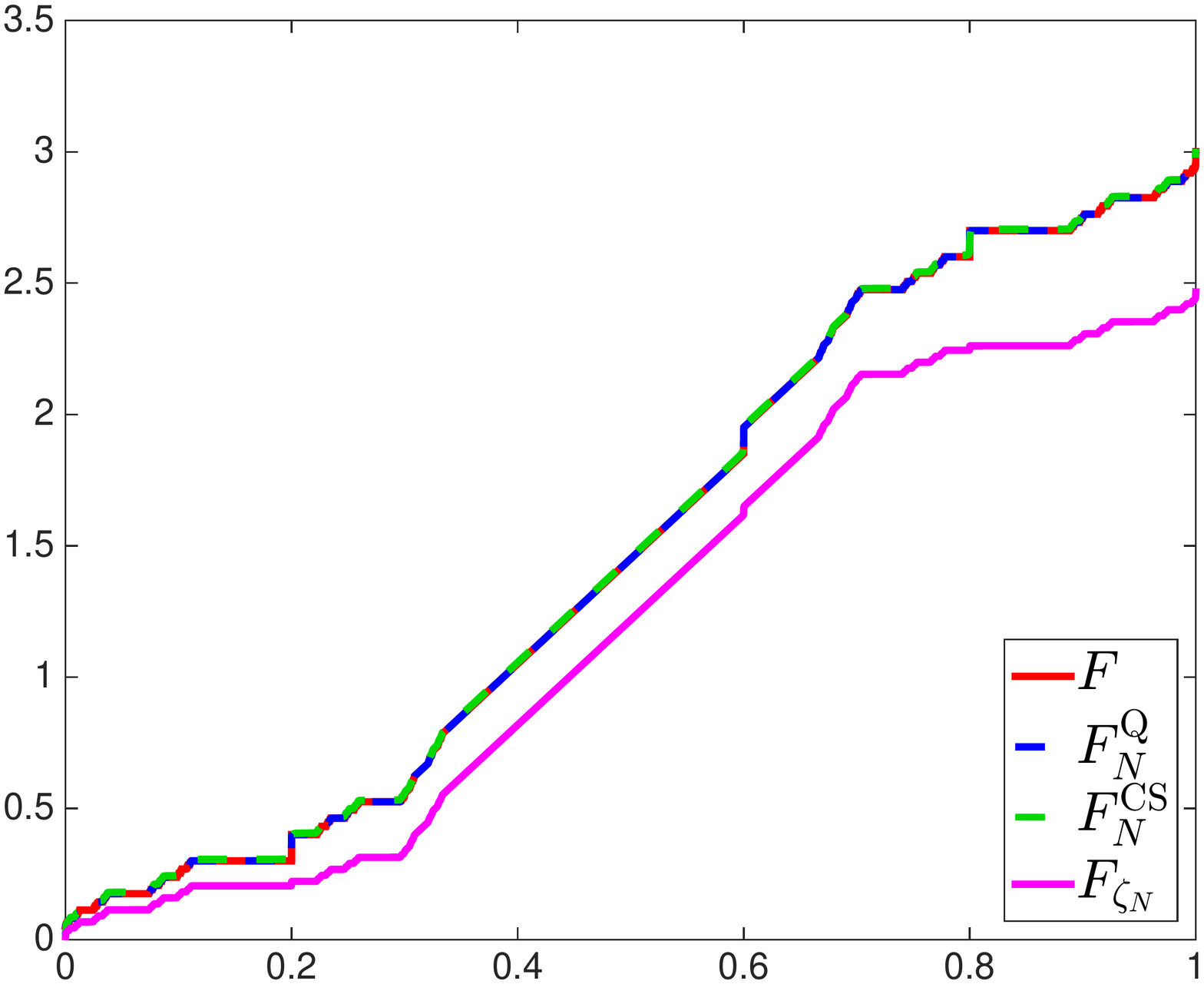}}
\put(55,85){\includegraphics[width=30mm]{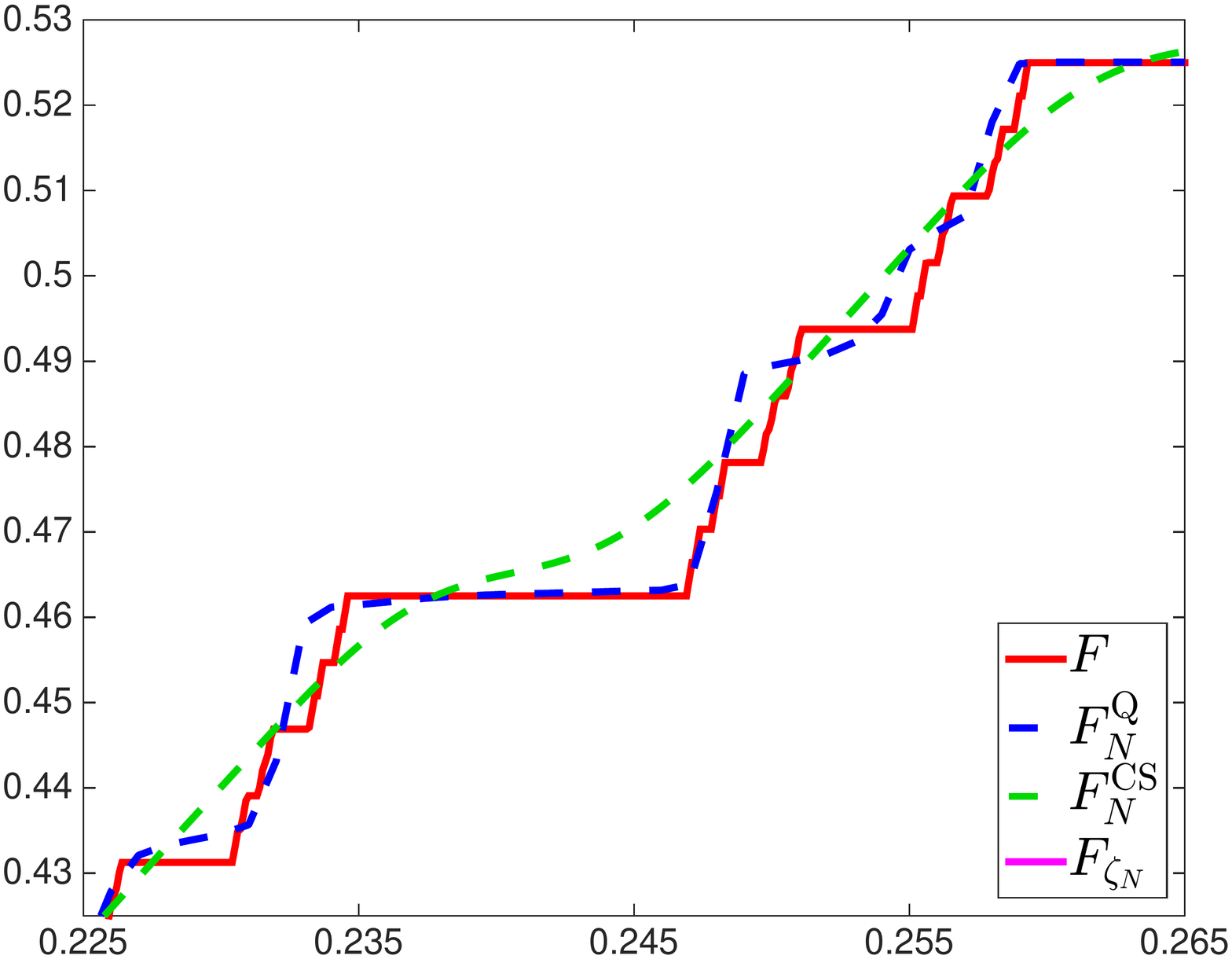}}
\put(270,85){\includegraphics[width=30mm]{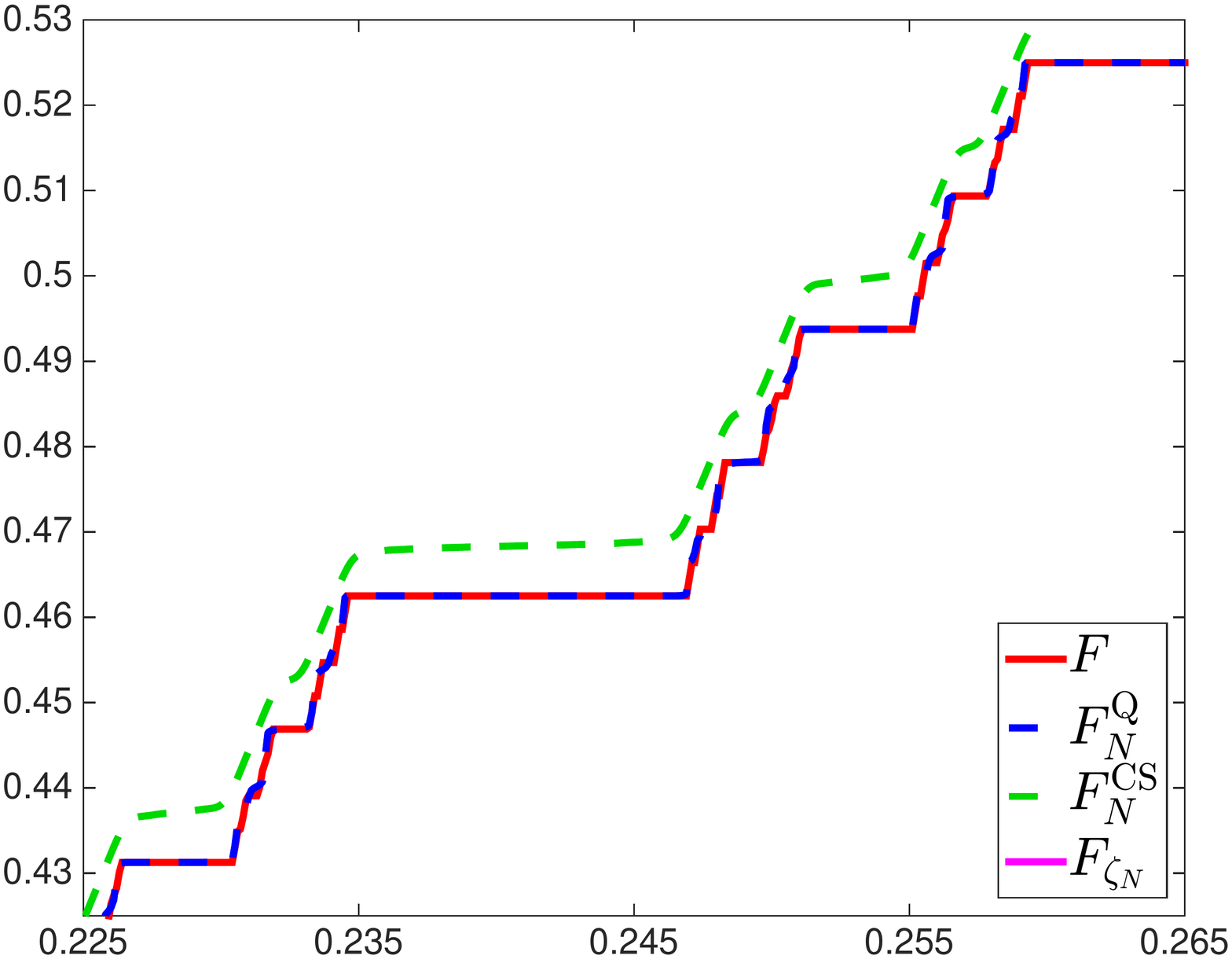}}

\put(130,0){\footnotesize $t$}
\put(350,0){\footnotesize $t$}
\put(170,145){\footnotesize $N=100$}
\put(384,145){\footnotesize $N=1000$}

\put(91, 90){\vector(0, -1){48}}
\put(306, 90){\vector(0, -1){48}}


\end{picture}
\caption{\figCaptionSize Example~\ref{ex:artif_cdf} -- Approximation of the distribution function (Theorems~\ref{thm:cdf_conv_CS} and~\ref{thm:cdf_conv_Q})} 
\label{fig:artif_cdf}
\end{figure*}

\begin{figure*}[th] 
\begin{picture}(140,180)
\put(20,0){\includegraphics[width=80mm]{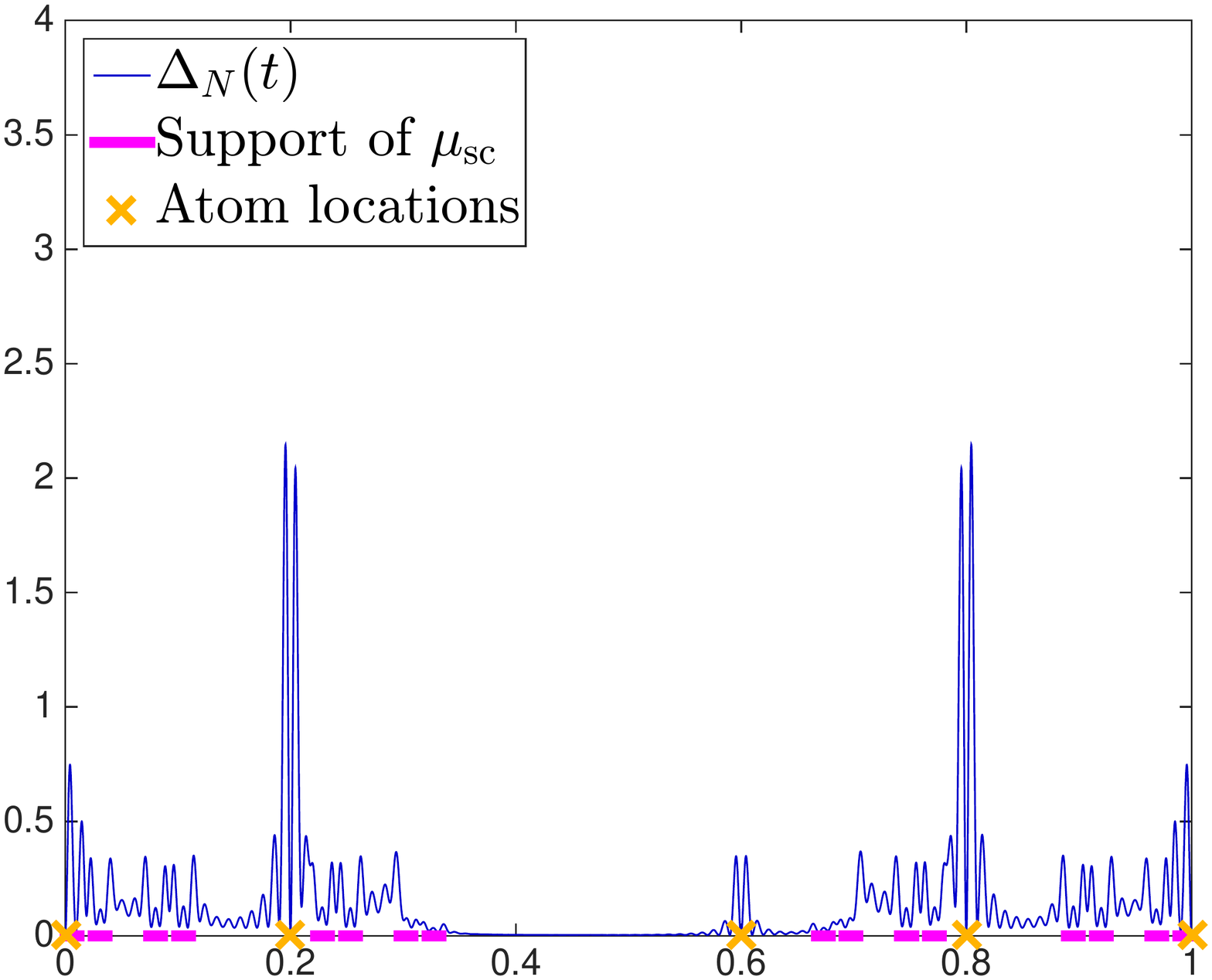}}
\put(235,0){\includegraphics[width=80mm]{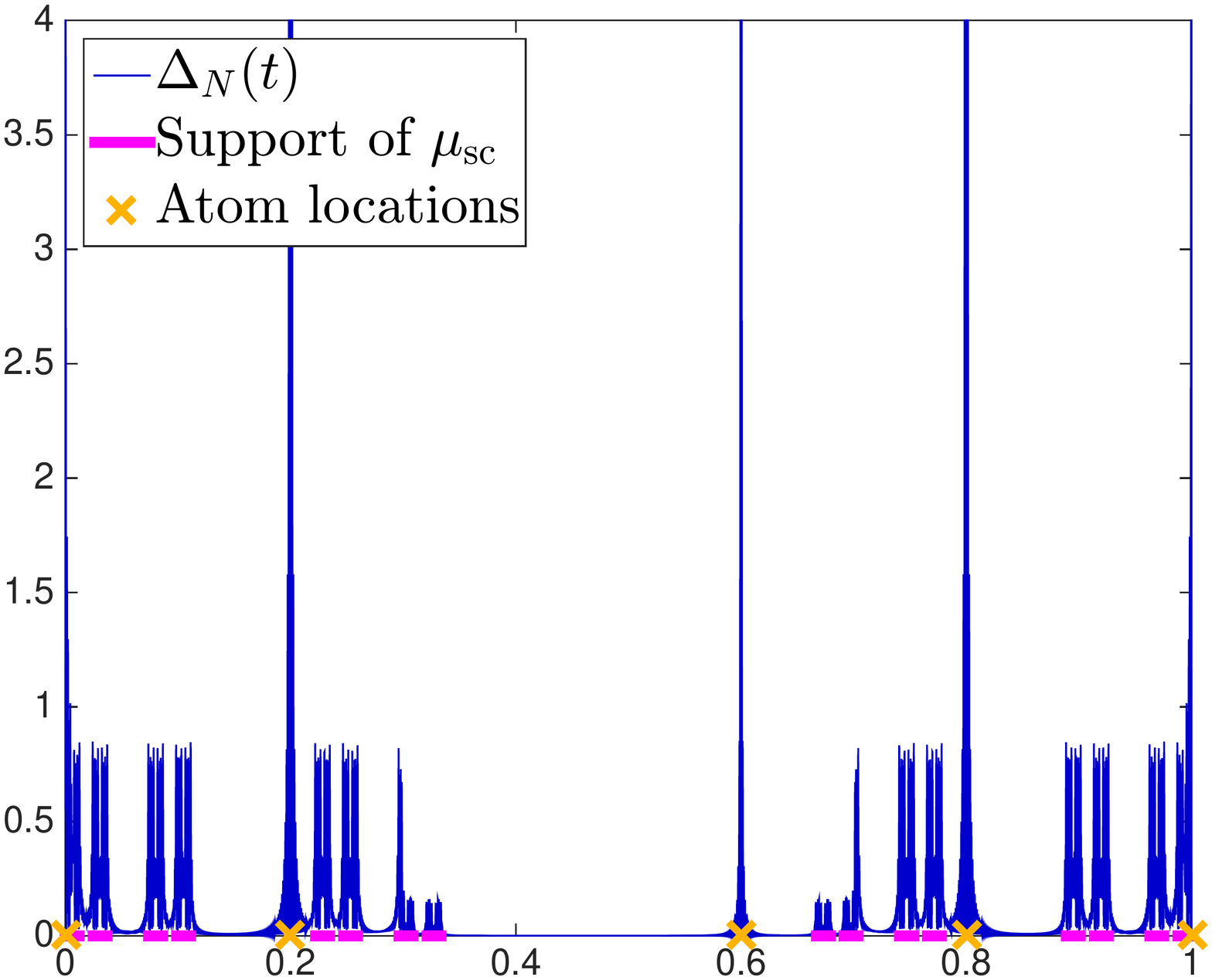}}

\put(130,0){\footnotesize $t$}
\put(350,0){\footnotesize $t$}
\put(170,145){\footnotesize $N=100$}
\put(384,145){\footnotesize $N=1000$}

\end{picture}
\caption{\figCaptionSize Example~\ref{ex:artif_cdf} -- Quantifying singularity of the measure. The singularity indicator~$\Delta_N$ is defined in~(\ref{eq:singul_indicator}).}
\label{fig:artif_singul_quat}
\end{figure*}

\end{example}

\cleardoublepage

\section{Approximation of the Koopman operator}
In this section we describe a method to construct an approximation of the Koopman operator $U$ from data. This can be done in a number of different ways, e.g., using the finite section method (also referred to as the finite central truncation). A numerical algorithm to carry out this approximation for the Koopman operator is called the Extended Dynamic Mode Decomposition (EDMD)~\cite{williams2015data} (see Section~\ref{sec:dmd_relation} for a relation of a particular version of EDMD to the CD kernel analysis developed in the previous sections).

Here, we take a slightly different path, constructing an approximation that explicitly takes into account the contributions of the atomic and continuous parts of the spectrum. The starting point is the representation of $U$ as the spectral integral
\[
U =\int_{[0,1]} e^{i2\pi\theta} dE(\theta)
\]
provided by the spectral theorem. The approximation $U_K$ of $U$ is then given by
\begin{equation}\label{eq:UK}
U_K = \sum_{j=1}^K e^{i2\pi\theta_j}P_{A_j}
\end{equation}
where $P_{A_j} := E(A_j)$ is the spectral projection on the set $A_j\subset [0,1]$ and $\theta_j \in A_j$. The sets $A_j$ are chosen such that they form a disjoint partition of $[0,1]$, i.e., $A_j \cap A_k = \emptyset$ if $j\ne k$ and $\cup_{j=1}^K A_j = [0,1]$. In what follows we discuss
\begin{enumerate}
\item Convergence of $U_K$ to $U$ as $K$ tends to infinity.
\item Computation of the spectral projections $P_{A_j}$ from data.
\item Choice of the partition $(A_j)_{j=1}^K$.
\end{enumerate}

The results of this section can be seen as a generalization of the results of~\cite{MezicandBanaszuk:2004, mezic:2005} that considered the case of $A_j$ being a singleton, i.e., $A_j=\{\theta_j\}$, in which case $P_{A_j}$ is the projection on the eigenspace associated to $\theta_j$ provided that $\theta_j$ is an eigenvalue of $U$; otherwise $P_{A_j} = 0$. Here we treat the fully general case of projections on eigenspaces as well as subsets of the continuous spectrum.

\subsection{Convergence of $U_K$ to $U$}
Given a set $A\subset [0,1]$ we define its diameter as $ \mr{diam}(A) = \sup(A) - \inf(A)$ and we note that $\mr{diam}(A) \in [0,1]$. The following result shows that if the diameter of the sets comprising the disjoint partition in~(\ref{eq:UK}) tends to zero, then $U_K$ converges to $U$ in the strong operator topology.

\begin{theorem}
Let $(A_{j,K})_{j=1}^K$ be a sequence of disjoint partitions of $[0,1]$ (i.e.,  $A_{j,K}\cap A_{l,K} = \emptyset$ if $j\ne l$ and $\cup_{j=1}^K A_{j,K} = [0,1]$) satisfying 
\begin{equation}\label{eq:diamConv}
\lim_{K\to\infty}\max_{j=1,\ldots,K} \mr{diam}(A_{j,K}) = 0
\end{equation}
and let $\theta_{j,K} \in A_{j,K}$. Then the operators $U_K$ defined by~(\ref{eq:UK}) with $A_j$ and $\theta_j$ replaced by $A_{j,K}$ and $\theta_{j,K}$ convergence to $U$ in the strong operator topology, i.e.,
\begin{equation}
\lim_{K\to\infty} \| U_K f - Uf\|_{L_2(\nu)} = 0
\end{equation}
for all $f \in L_2(\nu)$.

\end{theorem}
\begin{proof}
Writing $P_{A_{j,K}} = \int_{[0,1]} I_{A_{j,K}}(\theta)\,dE(\theta)$ with  $I_{A_{j,K}}$ being the indicator function of the set $A_{j,K}$, we get
\[
U_K = \sum_{j=1}^K e^{i2\pi\theta_j}P_{A_{j,K}} = \int_{[0,1]}\sum_{j=1}^K e^{i2\pi\theta_{j,K}} I_{A_{j,K}}(\theta)\;dE(\theta) = \int_{[0,1]} g_K(\theta)\, dE(\theta),
\]
where
\[
g_K(\theta) = \sum_{j=1}^K e^{i2\pi\theta_{j,K}} I_{A_{j,K}}(\theta).
\]
 Now we observe that $g_K(\theta)$ converges pointwise to the function $e^{i2\pi\theta}$. To see this, fix $\theta \in [0,1]$ and $\epsilon > 0$. By~(\ref{eq:diamConv}) there exists a $K\in\Nb$ such that $\max_{j=1,\ldots,K} \mr{diam}(A_{j,K}) < \epsilon$. Since the partition is disjoint there exists  one and only one $A_{j,K}$ such that $\theta \in A_{j,K}$. Therefore
\[
|g_K(\theta) - e^{i2\pi\theta}| = | I_{A_{j,K}}(\theta)e^{i2\pi\theta_{j,K}} - e^{i2\pi\theta}| = | e^{i2\pi\theta_{j,K}} - e^{i2\pi\theta}| \le 2\pi |\theta - \theta_{j,K}| \le 2\pi\epsilon.
\]
Hence indeed $\lim_{K\to\infty} g_K(\theta) = e^{i2\pi\theta}$. The proof of the theorem is finished by observing that $|g_K(\theta)| \le 1$ (since $A_{K,j}$ are disjoint) and invoking \cite[Corollary 3.27]{kowalski2009spectral}.
\end{proof}

\subsection{Computation of spectral projections $P_{A_{j,K}}$ from data}\label{sec:projComp}
In this section we show how the spectral projections $P_{A_{j,K}}f$ of a given observable $f$ can be computed from data in the form of samples of $f$ on a single trajectory of the dynamical system~(\ref{eq:sys}). Throughout this section we assume that the set $A_{j,K} \subset [0,1]$ is given and we drop the subscripts. The goal is therefore to compute $P_A f$ from data, with $A\subset [0,1]$. The idea is to approximate the indicator function $I_A$ of the set $A$ using trigonometric polynomials and apply the spectral theorem. Indeed, if the trigonometric polynomials
$
p_N(\theta) = \sum_{k=-N}^N \alpha_{k,N} e^{i2\pi\theta k},
$
 $\alpha_{k,N}\in\Cb$, satisfy 
\[
\lim_{N\to\infty} p_N(\theta) = I_A(\theta),
\]
then
\begin{align}\label{eq:projLimit}\nonumber
P_A &= \int_{[0,1]} I_A(\theta)\,dE(\theta) = \int_{[0,1]} \lim_{N\to\infty} p_N(\theta)\,dE(\theta) =  \lim_{N\to\infty} \int_{[0,1]} p_N(\theta)\,dE(\theta)  \\ &=  \lim_{N\to\infty}\sum_{k=-N}^N \alpha_{k,N}\int_{[0,1]} e^{i2\pi\theta k} \, dE(\theta) = \lim_{N\to\infty}\sum_{k=-N}^N \alpha_{k,N}U^k,
\end{align}
where the first limit is understood pointwise and the remaining limits in the sense of convergence in the strong operator topology. The exchange of limit and integration is justified by~\cite[Corollary 3.27]{kowalski2009spectral}. The fact that $U^k = \int_{[0,1]} e^{i2\pi\theta k} \, dE(\theta)$ is a direct consequence of the spectral theorem. The following theorem summarizes these developments.
\begin{theorem}\label{thm:projAbstract}
Let $\alpha_{k,N}$ be such that $\lim_{N\to\infty} \sum_{k=-N}^N \alpha_{k,N} e^{i2\pi\theta k} = I_A(\theta)$ for all $\theta \in [0,1]$. Then
\begin{equation}\label{eq:L2projConv}
\lim_{N\to\infty} \Big\| P_A g -  \sum_{k=-N}^N \alpha_{k,N} U^k g \Big\|_{L_2(\nu)} = 0
\end{equation}
for all $g \in L_2(\nu)$.
\end{theorem}

In what follows we discuss how to choose the coefficients  $\alpha_{k,N}$ such that $\sum_{k=-N}^N \alpha_{k,N} e^{i2\pi\theta k} \to I_A(\theta)$ for all $\theta \in [0,1]$ and how to approximate the sum in~(\ref{eq:L2projConv}) from data. We start with the latter.

\subsubsection{Numerical computation of spectral projections $P_{A}f$}\label{eq:specProjNum}
According to Theorem~\ref{thm:projAbstract}, given coefficients $\alpha_{k,N}$ such that $\sum_{k=-N}^N \alpha_{k,N} e^{i2\pi\theta k} \to I_A(\theta)$ for all $\theta \in [0,1]$, we can approximate the projection $P_A f$ by 
\[
P_A f \approx P_{A,N}f : = \sum_{k=-N}^N \alpha_{k,N}U^k f = \sum_{k=-N}^N \alpha_{k,N} \cdot f \circ T^k
\]
with $P_{A,N}f$ converging to $P_A$ in the $L_2(\nu)$ norm. Given data in the form of samples of $f$
\begin{equation}\label{eq:data_yj}
y_j = f(x_j),\quad j = 1,\ldots M
\end{equation}
 evaluated on a single trajectory of the dynamical system~(\ref{eq:sys}), i.e., $x_{j+1} = T(x_j)$, we can evaluate the approximate projection $P_{A,N}f$ along the points on this trajectory. Indeed, provided that $2N < M$, we can evaluate  $(P_{A,N}f)(x_j)$ for $j \in \{N,\ldots M-N\}$ using
\begin{equation}\label{eq:projApproxSample}
(P_{A,N}f)(x_j) = \sum_{k=-N}^N \alpha_{k,N} f(x_{j+k}) = \sum_{k=-N}^N \alpha_{k,N} \,y_{j+k}.
\end{equation}

Equation~(\ref{eq:projApproxSample}) is readily implementable given the data~(\ref{eq:data_yj}) and the sequence of complex numbers $\alpha_{k,N}$. Two natural questions arise about this numerical scheme.

\begin{remark}[Density of a single trajectory]
The first question asks whether a single trajectory is sufficient to represent $P_{A,N} f$. Provided that the state-space $X$ is topological, the sigma algebra $\sigalg$ is the Borel sigma algebra and the measure $\nu$ is ergodic with the property that $\nu(G) > 0$ for every open set $G\subset X$, then for $\nu$ almost all initial conditions the trajectory of the dynamical system~(\ref{eq:sys}) will be dense in $X$. Therefore, under these conditions, we can evaluate the approximate projections $P_{A,N}$ on a dense set of points from a single trajectory of~(\ref{eq:sys}). If, on the other hand, the measure $\nu$ is not ergodic, then the set of points on which $P_{A,N}$ is evaluated will be confined to a single component of the ergodic partition of $X$.
\end{remark}

\begin{remark}[Pointwise convergence]\label{rem:pointWiseConv}
 Theorem~\ref{thm:projAbstract} guarantees $L_2(\nu)$ convergence of $P_{A,N} f$ to $P_A f$. The question of whether $\nu$-almost everywhere convergence holds is more subtle and may depend on the choice of coefficients $\alpha_{k,N}$. In Theorem~\ref{thm:projPointWiseConv} we prove that this convergence holds for the coefficients proposed in Section~\ref{sec:coefChoice}.
\end{remark}

\subsubsection{Choice of the coefficients $\alpha_{k,N}$}\label{sec:coefChoice}
\label{coef}
Now we show how to select the parameters $\alpha_{k,N} \in \Cb$ such that $\lim_{N\to\infty} \sum_{k=-N}^N \alpha_{k,N} e^{i2\pi\theta k} = I_A(\theta)$ for all $\theta \in [0,1]$. We restrict our attention to $A$ being either a singleton $A = \{\theta_0\}$ or an interval $A_j = [a,b)$. We note that  for neither of these sets the choice of the coefficients $\alpha_{k,N}$ is unique.

\paragraph{Singleton} For the singleton $\{\theta_0\}$, one possible choice is
\begin{equation}\label{eq:projEigenvalCoef}
\alpha_{k,N}^{\{\theta_0\}} = \begin{cases}\frac{1}{N+1} e^{-i2\pi  k \theta_0} & k \in \{0,\ldots, N\}\\
0 & \mr{otherwise},
\end{cases}
\end{equation}
which was proposed in~\cite{mezic:2005}. Another possible choice is the double-sided version of~(\ref{eq:projEigenvalCoef}) which is non-zero for $k\in\{-N,\ldots,N\}$ and has the $1/(2N+1)$ coefficient in front.

\paragraph{Interval} For the interval $[a,b)$, one possible choice is
\begin{equation}\label{eq:projInterval}
\alpha_{k,N}^{[a,b)} = \frac{1}{2} \alpha_{k,N}^{\{a\}} + \beta_{k,N}^{[a,b)} - \frac{1}{2} \alpha_{k,N}^{\{b\}},
\end{equation}
where
\begin{equation}\label{eq:cesaro_coefs_beta}
\beta_{k,N}^{[a,b)} =\begin{cases}  \frac{N-|k|}{N} \frac{i}{2\pi k}(e^{-i2\pi b k}  - e^{-i2\pi a k} ) & k \ne 0, \\ 
b-a & k = 0
\end{cases}
\end{equation}
 are the coefficients of the Ces\`aro sum of the degree-$N$  Fourier series approximation to the indicator function of $[a,b)$ (i.e., the coefficients of the convolution of the Fej\'er kernel~(\ref{eq:fejerKernel}) with the indicator function of $[a,b)$). These coefficients have the advantage of the indicator function approximation being nonnegative and less oscillatory than Fourier series approximation. The coefficients~$\beta_{k,N}^{[a,b)}$ satisfy
\[
\lim_{N\to\infty} \sum_{k=-N}^N \beta_{k,N}^{[a,b)} e^{i2\pi\theta k} = \frac{1}{2}I_{\{a\}}(\theta) + I_{(a,b)}(\theta) + \frac{1}{2}I_{\{b\}}(\theta)
\]
and hence the need for the corrective terms $\frac{1}{2} \alpha_{k,N}^{\{a\}}$ and $-\frac{1}{2} \alpha_{k,N}^{\{b\}}$ in~(\ref{eq:projInterval}) so that \[\lim_{N\to\infty} \sum_{k=-N}^N \alpha_{k,N}^{[a,b)} e^{i2\pi\theta k} =  I_{[a,b)}(\theta). \]

We remark that, any finite union of intervals and singletons can be obtained by combining~(\ref{eq:projEigenvalCoef}) and (\ref{eq:projInterval}).

The following theorem establishes $\nu$-almost everywhere convergence of the spectral projection approximations of Section~\ref{eq:specProjNum} with the choice of coefficients proposed in this section (see Remark~\ref{rem:pointWiseConv}).

\begin{theorem}\label{thm:projPointWiseConv}
Let $\alpha_{k,N}^{\{\theta_0\}}$ and $\alpha_{k,N}^{[a,b)}$ be given by~(\ref{eq:projEigenvalCoef}), respectively (\ref{eq:projInterval}) and let $g \in L_2(\nu)$ be given. Then for $\nu$-almost all $x \in X$
\begin{equation}\label{eq:pointWiseConvSingleton}
\lim_{N\to\infty}\sum_{k=-N}^N \alpha_{k,N}^{\{\theta_0\}} g(T^k(x)) = (P_{\{\theta_0\}}g)(x)
\end{equation}
and
\begin{equation}\label{eq:pointWiseConvInterval}
\lim_{N\to\infty}\sum_{k=-N}^N \alpha_{k,N}^{[a,b)} g(T^k(x)) = (P_{[a,b)}g)(x)
\end{equation}
\end{theorem}
\begin{proof}
The relation~(\ref{eq:pointWiseConvSingleton}) follows immediately from Theorem~\ref{thm:projAbstract} and the Wiener-Wintner ergodic theorem~\cite{wiener1941harmonic}. In order to prove~(\ref{eq:pointWiseConvInterval}), we first remark that it suffices to prove that the limit 
\begin{equation}\label{eq:betaNProj}
\lim_{N\to\infty}\sum_{k=-N}^N \beta_{k,N}^{[a,b)} g(T^k(x))
\end{equation}
exists for $\nu$-almost all $x \in X$; 
this follows immediately from Theorem~\ref{thm:projAbstract} and from the fact that the coefficients $\alpha_{k,N}^{[a,b)}$ are of the form~(\ref{eq:projInterval}) with $\nu$-almost everywhere convergence of the projections onto the singletons $\{a\}$ and $\{b\}$ guaranteed by~(\ref{eq:pointWiseConvSingleton}).

In order to prove the existence of the limit~(\ref{eq:betaNProj}) we use the recent generalization of the Wiener-Wintner theorem \cite[Corollary 7.2]{lacey2008hilbertTransform} establishing the existence of the limit
\begin{equation}\label{eq:HilbertTransConv}
\lim_{N\to\infty}\sum_{0< |k| \le  N} \frac{e^{i2\pi\theta k}}{k}g(T^k(x))
\end{equation}
for $\nu$-almost all $x \in X$ and all $\theta \in [0,1]$. In order to apply this result we observe that for $k\ne 0$ we have 
\[
\beta_{k,N} =  \frac{i}{2\pi } \Big( \frac{N-|k|}{N}   \frac{e^{-i2\pi b k}}{k}  - \frac{N-|k|}{N}\frac{e^{-i2\pi a k} }{k}\Big)
\]
with $a,b\in[0,1]$. Therefore it suffices to prove the $\nu$-almost everywhere  existence of
\[
\lim_{N\to\infty}\sum_{0< |k| \le  N} \frac{N- |k|}{N}\frac{e^{i2\pi\theta k}}{k}g(T^k(x))
\]
for $\theta \in [0,1]$. We have
\begin{align*}
\lim_{N\to\infty}\sum_{0< |k| \le  N} \frac{N- |k|}{N}\frac{e^{i2\pi\theta k}}{k}g(T^k(x)) & = \lim_{N\to\infty} \frac{1}{N}\sum_{n=0}^{N-1}\sum_{0< |k| \le  n} \frac{e^{i2\pi\theta k}}{k}g(T^k(x)),
\end{align*}
which is nothing but the Ces\` aro sum of~(\ref{eq:HilbertTransConv}) and hence it exists (and is equal to~(\ref{eq:HilbertTransConv})).
\end{proof}

\begin{remark}We remark that the null-set at which the limits in (\ref{eq:pointWiseConvSingleton}) and (\ref{eq:pointWiseConvInterval})  do not exist
 can be chosen independent of $\theta_0$ and $a$ and $b$.
 \end{remark}

\subsection{Choice of the partition $(A_j)_{j=1}^K$}
In this section we discuss how to choose the partition $(A_j)_{j=1}^K$ based on spectral properties of $U$. In order to obtain convergence in strong topology the only assumptions we need to satisfy is that of disjointness  of the partition $(A_j)_{j=1}^K$ and of the diameter of $A_j$  tending to zero (Theorem~\ref{thm:projAbstract}). Therefore, we get the following immediate corollary pertaining to a generic partition of $[0,1]$ to intervals:

\begin{corollary}\label{cor:UapproxConv}
Let $0 = a_{1,K} < a_{2,K} \ldots < a_{K,K} <  a_{K+1,K} =1$ satisfy \[
\lim_{K\to\infty} \max_{j\in\{1,\ldots,K\}} |a_{j+1,K}-a_{j,K}| = 0.\]
Then, for $\alpha_{k,N}^{[a_{j,K},a_{j+1,K})}$ defined by~(\ref{eq:projInterval}) and $\theta_{j,K} = (a_{j,K} + a_{j+1,K})/2$, it holds
\begin{equation}\label{eq:corolaryEquation}
\lim_{K\to\infty}\lim_{N\to\infty}  \Big\|  \sum_{j=1}^{K} e^{i2\pi\theta_{j,K}}\sum_{k=-N}^N \alpha_{k,N}^{[a_{j,K},a_{j+1,K})} U^k g - U g \Big\|_{L_2(\nu)} = 0
\end{equation}
for all $g \in L_2(\nu)$.
\end{corollary}

The question we want to address is on how to choose the interval endpoints $a_{j,K}$ in an informed way based on the spectral analysis of $U$ from Sections~\ref{sec:CD} and \ref{sec:weakApprox}. In particular, under what conditions one should consider a more general partition $(A_{j,K})_{j=1}^K$ with some of the $A_{j,K}$ being singleton and how to choose $\theta_{j,K}$ better than the interval midpoints. The information available to us is (an approximation of) the measure $\mu_f$ and the goal is to construct a partition such that the approximation $U_K$ is accurate on the  cyclic subspace $\mathcal{H}_f$  associated to $f$ (defined in (\ref{eq:cyclic})). In general, fixing the value of $K$, the partition should be chosen fine in the regions where $\mu_f$ is large and coarse where $\mu_f$ is small. This is automatically achieved if the masses $\mu_f(A_{j,K})$ of the partition elements are the same. However, if the measure $\mu_f$ has a non-zero atomic part, a partition with all masses being equal may not exist. Therefore we suggest the following procedure, where we let
\[
\mu_f = \sum_{k=1}^\infty w_k \delta_{\theta_k} + \mu_{\mr{ac}} + \mu_{\mr{sc}}.
\]

\begin{enumerate}
\item (Singletons) Define the singletons of the partition to be the locations $\theta_{k_j}$ of those atoms of $\mu_f$ for which the weight $w_{k_j} $ satisfies $w_{k_j}\ge  \mu_f([0,1]) / K = m_0 / K$. Assume there is $K_{\mr{at}}$ of such atoms and define $\bar\mu_f = \mu_f - \sum_{j=1}^{K_{\mr{at}}} w_{k_j}\delta_{k_j}$. This step can be carried using the CD kernel approximation to the atomic part of $\mu_f$ (Theorem~\ref{thm:atoms}).
\item (Intervals) If $K_{\mr{at}} < K$, define $K -K_{\mr{at}}+1$ interval endpoints such that \[\max_{j}\bar\mu_f([a_j,a_{j+1})) - \min_{j}\bar\mu_f([a_j,a_{j+1}))\] is minimized. The idea is that the intervals of the partition should have the same $\bar\mu_f$-measure. However, achieving exactly the same measure is not possible in general if atoms are still present in $\bar\mu_f$, hence the minimization of the variation. If $\bar\mu_f$ is atomless, then simply $a_j = F^{-1}_{\bar\mu_f}(b_j)$, where $(b_j)_{j=1}^{K-K_{\mr{at}}+1}$ constitutes a \emph{uniform} partition of $[0,\bar\mu_f([0,1])]$ to $K-K_{\mr{at}}$ intervals and $F_{\bar\mu_f}(t) = \bar\mu_f([0,t]) $ is the distribution function of $\bar\mu_f$. Finally, we subtract the atom locations from the intervals to which they belong (to prevent double counting). This step can be carried out by first subtracting from $m_k$ the moments of the atoms extracted in the first step, thereby obtaining the moments of $\bar\mu_f$. Subsequently an approximation to $F_{\bar\mu_f}$ can be constructed using the methods of Section~\ref{sec:weakApprox}.
\item (Choice of $\theta_{j,K}$) The frequencies $\theta_{j,K} \in A_{j,K}$ representing each element of the partition $A_{j,K}$ are chosen to be the conditional expectations $\frac{1}{\mu_f(A_{j,K})}\int_{A_{j,K}} \theta \, d\mu_f(\theta)  $. The conditional expectation can be approximated using the weak approximations to $\mu_f$ from Section~\ref{sec:weakApprox}. Note that, of course, if $A_{j,K} $ is a singleton $\{\theta_0\}$, then $\theta_{j,K} = \theta_0$.
\end{enumerate}

\section{Relation to Dynamic mode decomposition}\label{sec:dmd_relation}
In this section we briefly describe an interesting relation of the proposed method to the so-called Hankel Dynamic Mode Decomposition (Hankel DMD)~\cite{arbabi2017ergodic}, which is a variation of the classical DMD algorithm for spectral analysis of dynamical systems~\cite{schmid:2010}. The crucial fact for the argument presented is the following consequence of the spectral theorem\footnote{The consequence of the spectral theorem we are referring to here is the isomorphism between $H_f\subset L_2(\nu)$ and $L_2(\mu_f)$ where $U^k f \in H_f$ is identified with $z^k \in L_2(\mu_f)$, and the unitary equivalence of $U:H_f\to H_f$ and $M_z: L_2(\mu_f)\to L_2(\mu_f)$. }: the operators (sometimes called the \emph{finite central truncations} of the respective infinite matrices)
\[
U_N = P_N U P_N \quad \mr{and}\quad \tilde{U}_N = \pi_N M_z \pi_N
\]
have \emph{the same spectrum}. Here, $P_N$ and $\pi_N$ denote the  $L_2(\nu)$ respectively $L_2(\mu_f)$ orthogonal projections onto \[
\Fc_N = \mr{span}\{f,Uf,\ldots, U^{N-1}f\} \quad  \mr{and}\quad \Zc_N=\mr{span}\{1,z,\ldots,z^{N-1}\},
\]
respectively, and $M_z : L_2(\mu_f)\to L_2(\mu_f)$ denotes the multiplication-by-$z$ operator, i.e., $M_z \xi = z\xi$ for any function $\xi \in L_2(\mu_f)$.  In fact, more is true: the finite-dimensional operators $U_N$ and $\tilde{U}_N$ (restricted to $\Fc_N$ and $\Zc_N$) have \emph{identical matrix representations} in the bases $(f,Uf,\ldots,U^{N_0-1}f)$, respectively $(1,z,\ldots,z^{N_0-1})$, where $ N_0 \in \{1,\ldots, N\}$ is the largest power such that $(f,Uf,\ldots,U^{N_0-1}f)$ is linearly independent. We would like to emphasize here that the operator $U_N$ acts on the space $L_2(\nu)$, where $\nu$ is a measure on the abstract state space $X$ whereas $\tilde{U}_N$ acts on $L_2(\mu_f)$, where $\mu_f$ is a measure defined on $\Cb$. Therefore, remarkably, the operator $\tilde{U}_N$, acting on a concrete, well-understood, space $L_2(\mu_f)$, contains all information about $U_N$, which acts on an abstract, intangible, space $L_2(\nu)$.

In order to use this fact to understand DMD in terms of the scope of the current work, we use two known facts. First, the operator $U_N$ (restricted to $\Fc_N$) is precisely the Hankel dynamic mode decomposition operator, in the limit as the number of samples used in the Hankel DMD goes to infinity (see~\cite{arbabi2017ergodic,korda2018convergence}). Second, the monic degree-$N$ orthogonal polynomial $\Phi_N$ with respect to $\mu_f$ is equal to the characteristic polynomial of $\tilde{U}_N$ (and hence of $U_N$); see~\cite[Theorem~1.2.6]{simonSzegoDescendants}. 

This leads to the following theorem:
\begin{theorem}\label{thm:dmd}
Suppose that orthogonal polynomials up to degree $N$ associated to $\mu_f $ exist\footnote{The existence and uniqueness (up to scaling) of orthogonal polynomials up to degree $N$ for $\mu_f$ is assured if the support of $\mu_f$ contains at least $N+1$ distinct points. In particular, if the support of $\mu_f$ contains infinitely many points, there exists a unique set of monic orthogonal polynomials $\{\Phi_k\}_{k=0}^\infty$ with $\mr{\deg}\,\Phi_k =k$.}. Then the monic degree-$N$ orthogonal polynomial for $\mu_f$, $\Phi_N$, is equal to the characteristic polynomial of the Hankel DMD operator $U_N$. In particular, the zeros of $\Phi_N$ are equal to the eigenvalues of $U_N$, including multiplicities.
\end{theorem}

We remark that the monic orthogonal polynomial $\Phi_N$ can be readily obtained by normalizing the leading coefficients of the orthonormal polynomials $\varphi_N$ for $\mu_f$ obtainable using the Cholesky decomposition of the moment matrix of $\mu_f$ defined in~(\ref{eq:momMat}) as in~(\ref{eq:cholesky}) with $\tilde{\mu}$ replaced by $\mu = \mu_f$.

\subsection{Properties of Hankel DMD eigenvalues}
Theorem~\ref{thm:dmd} allows us to study the behavior of the eigenvalues of the Hankel DMD operator $U_N$ by studying the zeros of $\Phi_N$ whose behavior is well understood. For example, we have the following slightly surprising corollary:
\begin{corollary}\label{cor:inUnCircle}
If the subspace spanned by $\{f,Uf,U^2f,\ldots\}$ is infinite-dimensional, then all eigenvalues of the Hankel DMD operator $U_N$ lie strictly inside the unit circle; in particular, no eigenvalue of $U_N$ lies on the boundary of the unit circle (where all the eigenvalues of $U$ lie).
\end{corollary}
\begin{proof}
The assumption of the span of $\{f,Uf,U^2f,\ldots\}$ being infinite-dimensional implies that $\{1,z,z^2,\ldots\}$ is a linearly independent sequence in $L_2(\mu_f)$ and hence a full sequence of orthogonal polynomials for $\mu_f$ exists. By Theorem~\cite[Theorem~1.8.4]{simonSzegoDescendants}, the zeros of these orthogonal polynomials, and hence (by Theorem~\ref{thm:dmd}) the eigenvalues of $U_N$, all lie strictly inside the unit circle.
\end{proof}

Another interesting corollary concerns the asymptotics of the distribution of eigenvalues as~$N$ tends to infinity. The corollary works under the assumption of $\mu_f$ being \emph{regular} in the sense of~\cite[p. 121]{simonSzegoDescendants}; a sufficient condition for regularity is for the density $\rho$ of the absolutely continuous part of $\mu_f$ to satisfy the \emph{Szeg\H o condition}
\begin{equation}\label{eq:regDens}
\int_0^1\log(\rho(\theta))\,d\theta > -\infty.
\end{equation}
 This is implied for instance by $\rho$ being strictly positive and bounded away from zero.

Let now $\mu_{\mr{DMD}}^N$ be the normalized counting measure supported on the eigenvalues of $U_N$, including multiplicities, i.e., 
\begin{equation}\label{eq:countingMeas}
\mu_{\mr{DMD}}^N = \frac{1}{N}\sum_{j = 0}^{n_\lambda} a_{\lambda_{j,N}}\delta_{\lambda_{j,N}},
\end{equation}
where $\lambda_{1,N},\ldots, \lambda_{n_\lambda,N}$, $n_\lambda \le N$, are the eigenvalues of $U_N$ and $a_{\lambda_{j,N}}$ denotes their algebraic multiplicity.

\begin{corollary}\label{cor:uniformDist}
If $\mu_f$ is regular in the sense of~\cite[p. 121]{simonSzegoDescendants} (which is implied by condition~(\ref{eq:regDens})), then the normalized counting measures~$\mu_{\mr{DMD}}^N$ converge weakly to the uniform distribution on the unit circle.
\end{corollary}
\begin{proof}
Follows from Theorem~\ref{thm:dmd} and from~\cite[Theorems 2.15.1 and 2.15.4]{simonSzegoDescendants} which prove this asymptotic behavior of the zeros of $\Phi_N$.
\end{proof}

This corollary has the following implication: Whenever the operator $U$ has a continuous spectrum and the observable $f$ is such that the density of the absolutely continuous part of $\mu_f$ is bounded away from zero, then the eigenvalues of $U_N$ will be, in the limit as $N\to\infty$, distributed uniformly on the unit circle, irrespective of the point and singular continuous parts of the spectrum.

\section{Numerical examples}
\subsection{Cat map}
This example analyzes the spectrum of Arnold's cat map
\[
\begin{array}{llll}
x_1^+ & = &  2x_1 + x_2 & \;\mr{mod}\;1 \\ 
x_2^+ & = & x_1 + x_2& \; \mr{mod}\;1.
\end{array}
\]
We use two different observables
\begin{align*}
f_1 &=  e^{i2\pi (2x_1+x_2)} + \frac{1}{2}e^{i2\pi (5x_1 + 3x_2)}, \\  
f_2 &=  e^{i2\pi (2x_1+x_2)} + \frac{1}{2}e^{i2\pi (5x_1 + 3x_2)}+\frac{1}{4}e^{i2\pi (13x_1 + 8x_2)}
\end{align*}
for which the associated measure $\mu_f$ is known analytically~(see~\cite{govindarajan2018approximation}) to be $\mu_{f_1}=\rho_{f_1} d\theta$ and $\mu_{f_2}=\rho_{f_2} d\theta$ with
\begin{align*}
\rho_{f_1} &= \frac{5}{4} + \cos(2 \pi\theta) \\
\rho_{f_2} & = \frac{ 21}{16} + (5/4)\cos(2\pi\theta) + \frac{1}{2}\cos(4 \pi \theta).
\end{align*}

We use~(\ref{eq:moment:approx}) with $M=10^5$ to approximately compute the first $N=100$ moments. Figure~\ref{fig:catMap_density} shows the approximations $\zeta_N(\theta)$, defined in~(\ref{eq:zeta_N}), of the densities. We observe a very good match for both observables. In order to numerically verify the absence of the singular parts of the spectra we also plot in Figure~\ref{fig:catmap_singul} the distribution function approximations and the singularity measure. In particular we observe that $F_{\zeta_N}$ matches very closely the other distribution function approximations (as well as the true distribution function), indicating the absence of the singular part of the spectrum. This is confirmed by the singularity indicator $\Delta_N$, defined in~(\ref{eq:singul_indicator}), being almost identically zero. For space reasons we show these plots only for the first observable, the results being almost identical for the second. Next, in Figure~\ref{fig:catMap_projection} we show the approximation of the spectral projection $P_{[a,b]}f$ for the observable $ f = e^{i2\pi (2x_1+x_2)}$ and interval $[a,b] = [0.125,0.375]$. At present, no analytical expression is known for this projection but the results seem to be in accordance\footnote{Due to a different scaling, the corresponding interval $[a,b]$ in~\cite{govindarajan2018approximation} is $[\pi/4,3\pi/4]$.} with numerical approximations obtained in~\cite{govindarajan2018approximation} using a very different method. We note that the prominent diagonal pattern in Figure~\ref{fig:catMap_projection} is aligned with the eigenvector $[1, -(1 +\sqrt{5})/2]$ associated to the stable eigenvalue of the matrix $\begin{bmatrix}2 & 1\\ 1 & 1\end{bmatrix}$ defining the dynamics; this direction corresponds to the stable foliation of the hyperbolic dynamical system.  Finally, in Figure~\ref{fig:catMap_zeros} we compare the eigenvalues of the Hankel DMD operator with the zeros of the $N$-th orthogonal polynomial associated to $\mu_f$ with $f = f_2$. In accordance with the results of Section~\ref{sec:dmd_relation}, the eigenvalues and the zeros almost coincide, the discrepancy being due to a finite number of samples ($M = 10^5$) taken. In addition, as predicted by Corollary~\ref{cor:inUnCircle} and \ref{cor:uniformDist}, the eigenvalues lie strictly inside the unit circle and become uniformly distributed on the unit circle in the limit as $N\to\infty$.


\begin{figure*}[th]
\begin{picture}(140,160)
\put(20,-5){\includegraphics[width=70mm]{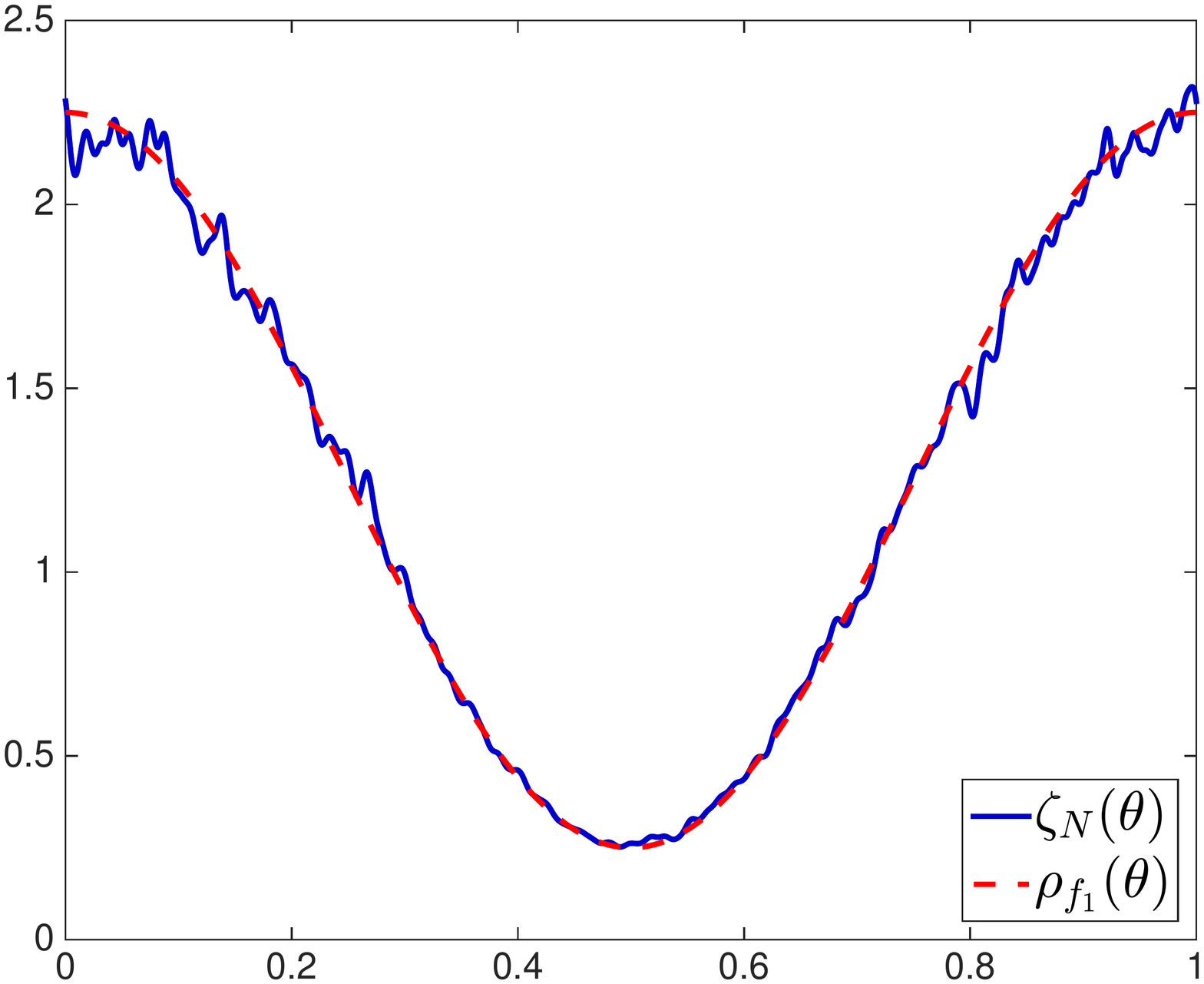}}
\put(235,-5){\includegraphics[width=70mm]{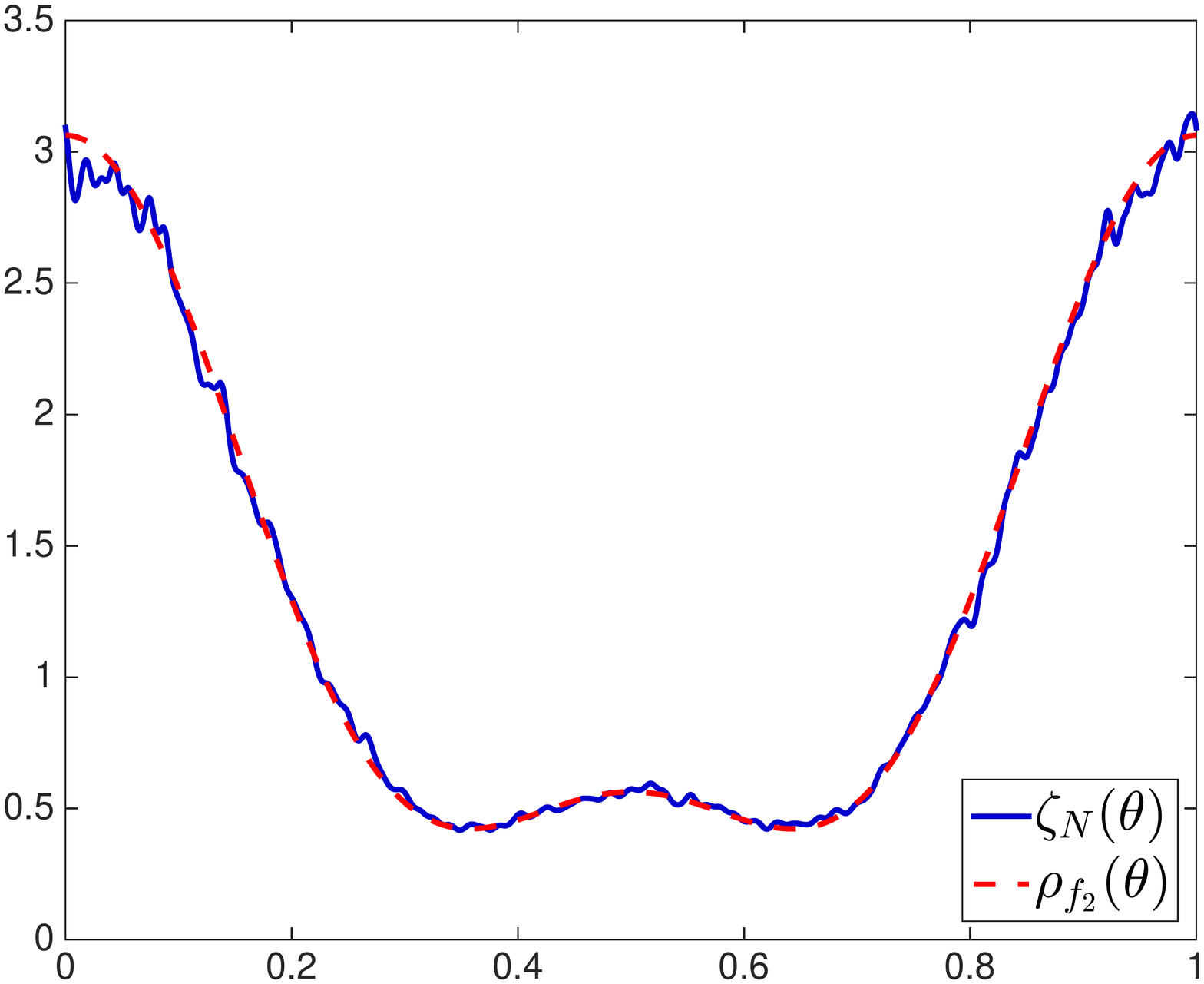}}

\put(120,-5){\footnotesize $\theta$}
\put(340,-5){\footnotesize $\theta$}

\end{picture}
\caption{\figCaptionSize Cat map -- Approximation of the densities by the CD kernel. Left: observable $f_1$. Right: observable $f_2$.}
\label{fig:catMap_density}
\end{figure*}

\begin{figure*}[th]
\begin{picture}(140,150)
\put(20,-5){\includegraphics[width=70mm]{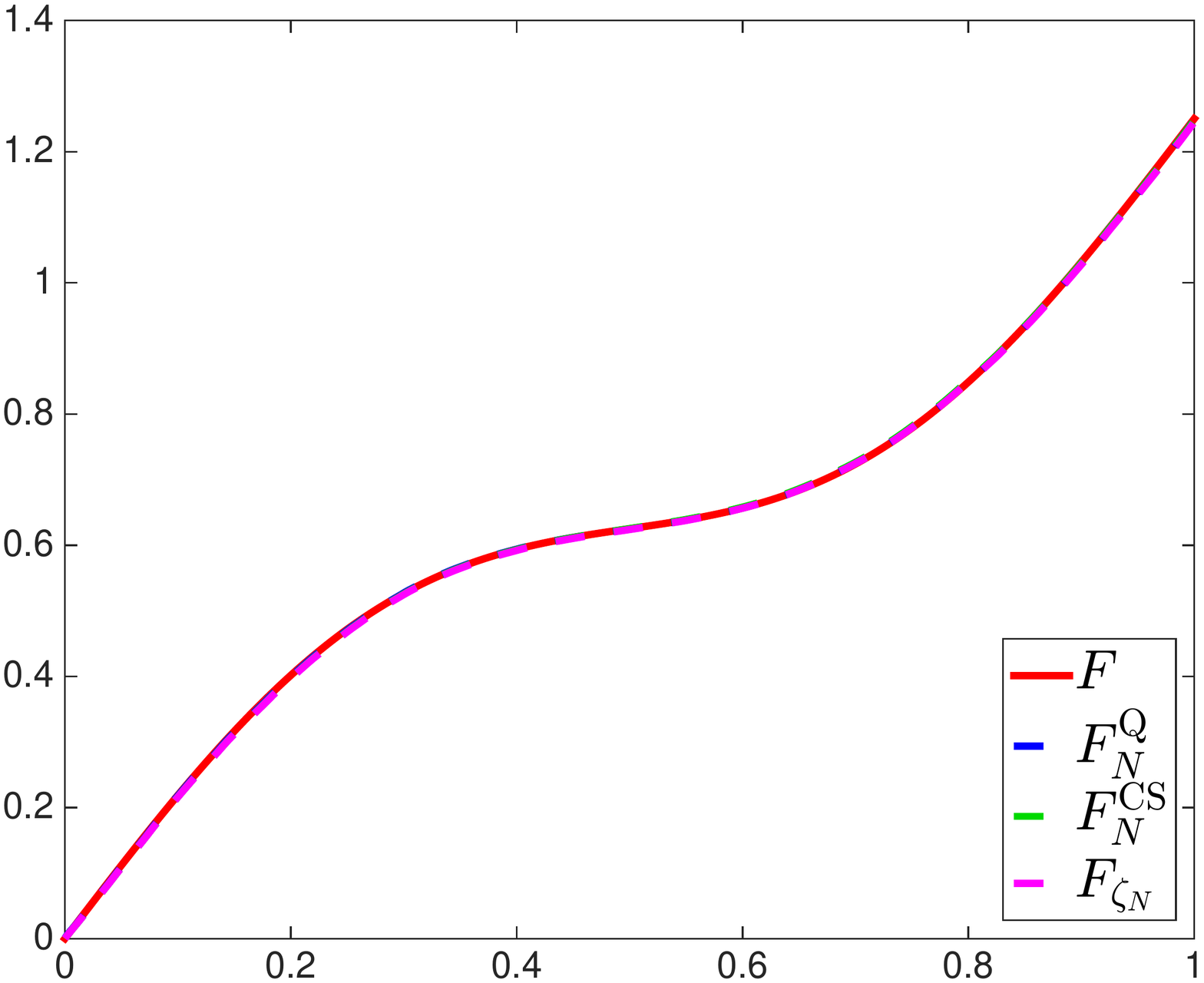}}
\put(235,-5){\includegraphics[width=70mm]{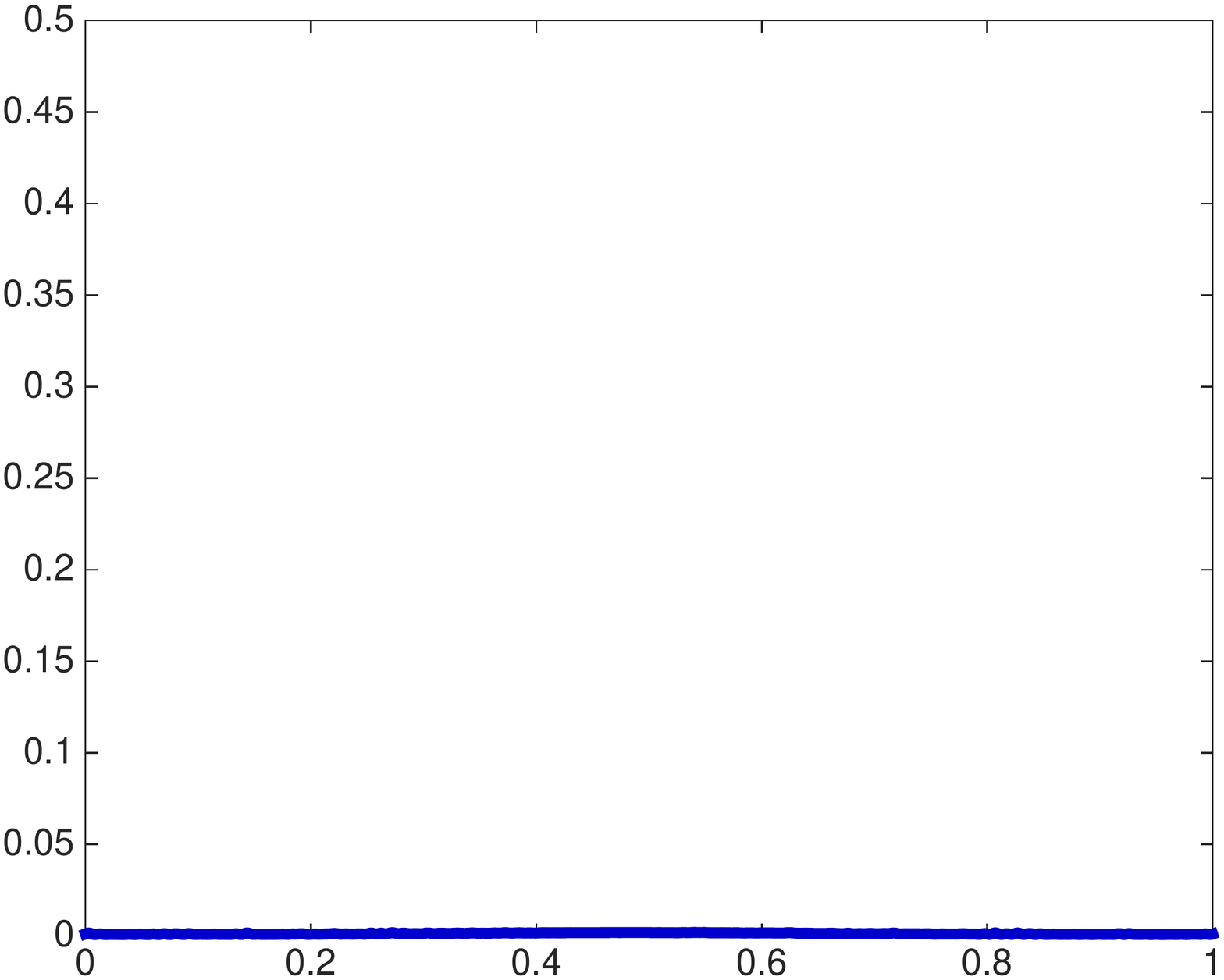}}

\put(120,-5){\footnotesize $\theta$}
\put(340,-5){\footnotesize $\theta$}

\end{picture}
\caption{\figCaptionSize Cat map -- Left: Approximation of the distribution function of $\mu_{f_1}$. Right: Singularity indicator $\Delta_N$ defined in~(\ref{eq:singul_indicator}). Both plots pertain to the observable $f_1$.}
\label{fig:catmap_singul}
\end{figure*}

\begin{figure*}[th]
\begin{picture}(140,180)
\put(10,0){\includegraphics[width=80mm]{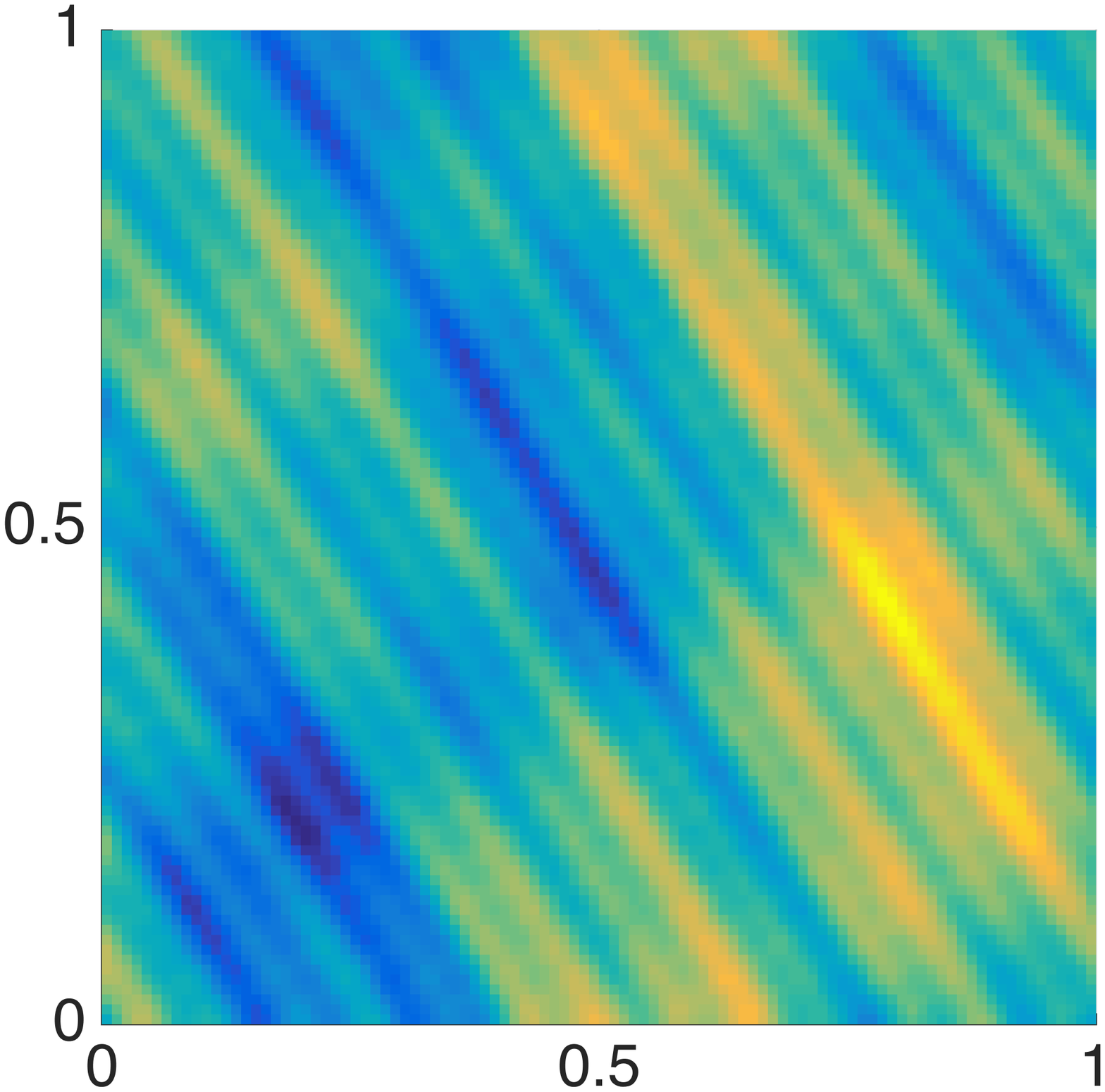}}
\put(225,0){\includegraphics[width=80mm]{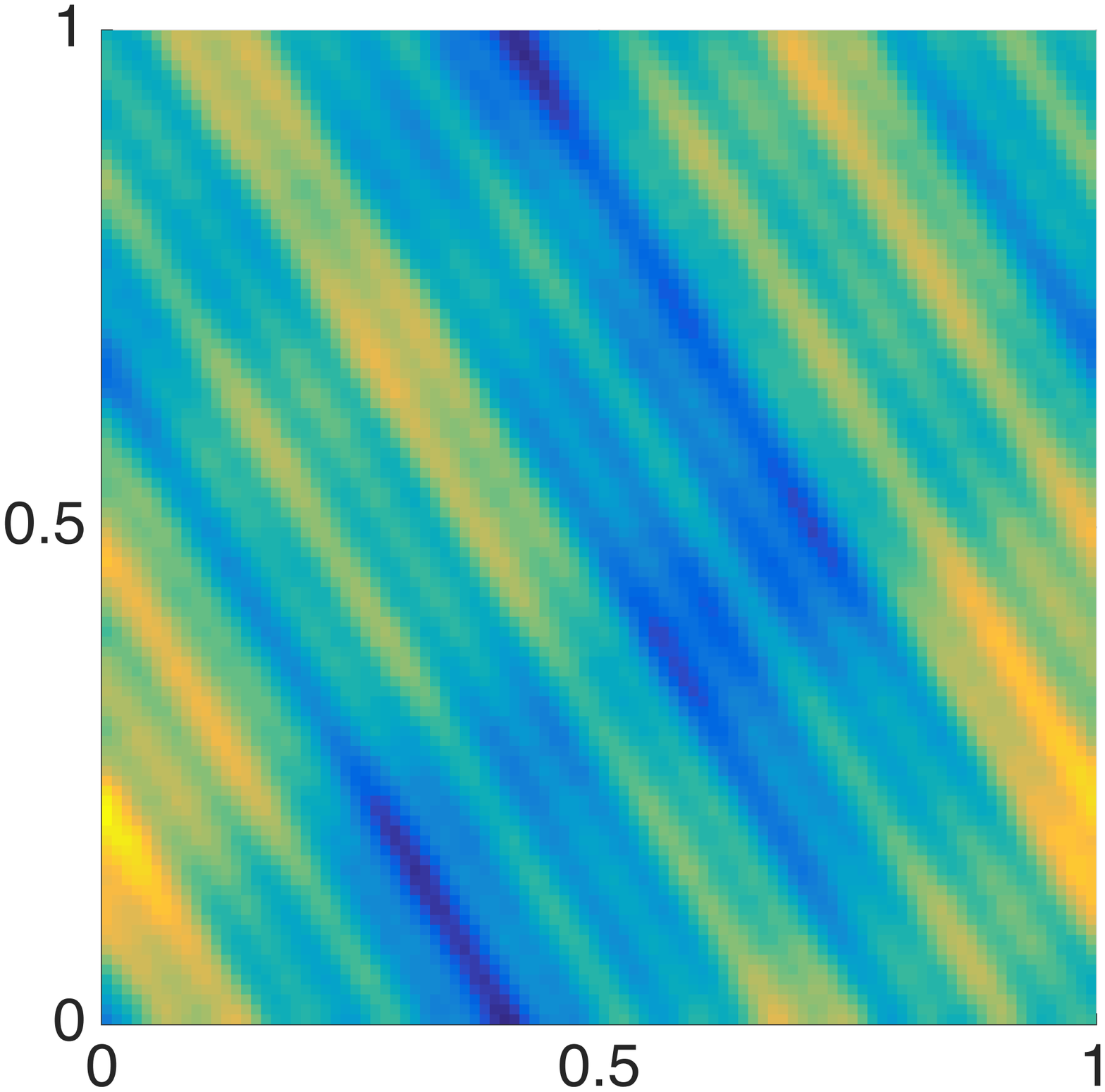}}

\put(123,-2){\small $x_1$}
\put(338,-2){\small $x_1$}
\put(20,88){\small $x_2$}
\put(237,88){\small $x_2$}
\put(105,170){\small Real part}
\put(310,170){\small Imaginary part}


\end{picture}
\caption{\figCaptionSize Cat map -- Approximation of the projection $P_{[a,b]}f$ with $f = e^{i2\pi (2x_1+x_2)}$ and $[a,b] = [0.125,0.375]$ and $N=100$, $M = 10^5$.}
\label{fig:catMap_projection}
\end{figure*}

\begin{figure*}[h!]
\begin{picture}(140,160)
\put(20,0){\includegraphics[width=70mm]{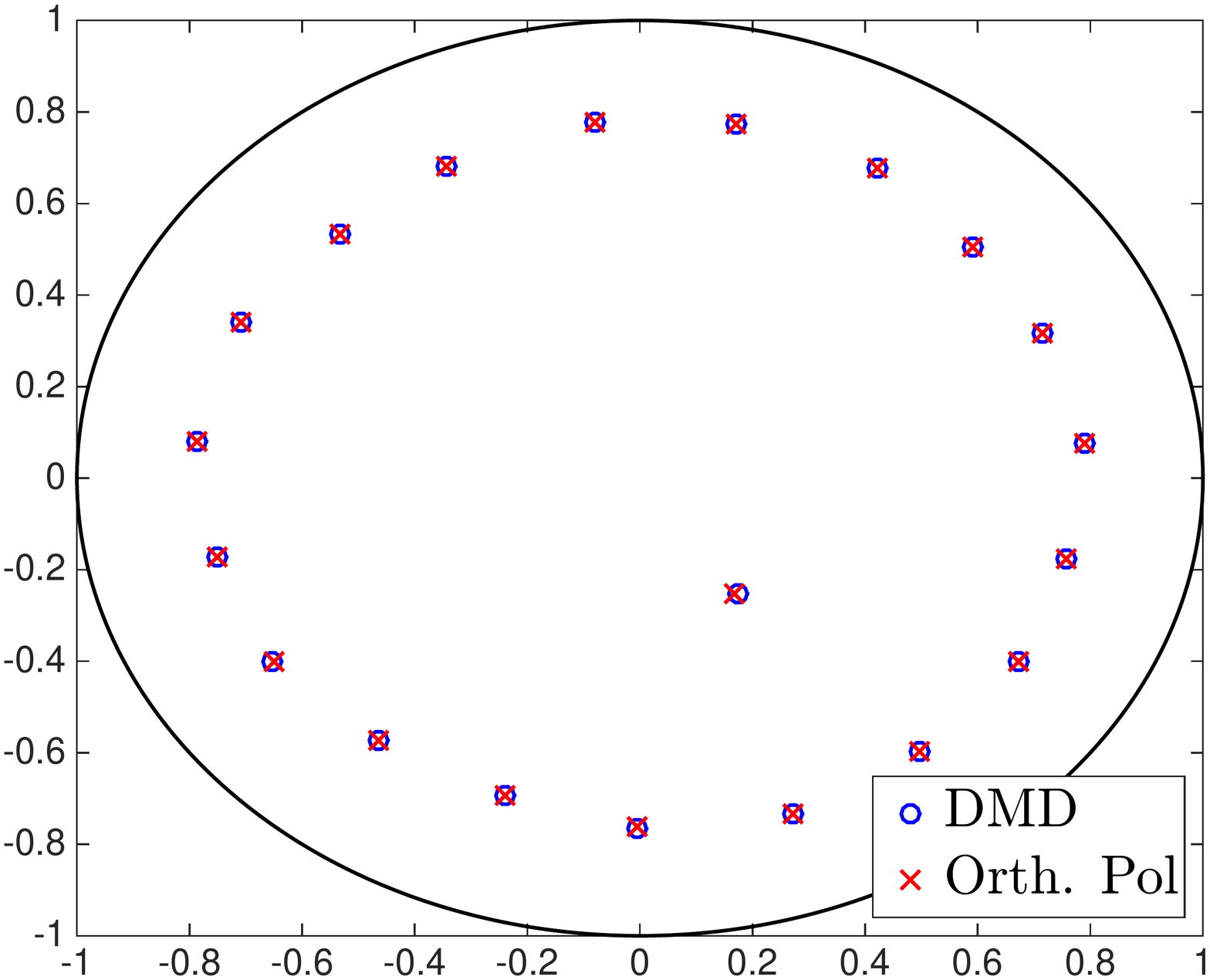}}
\put(235,0){\includegraphics[width=70mm]{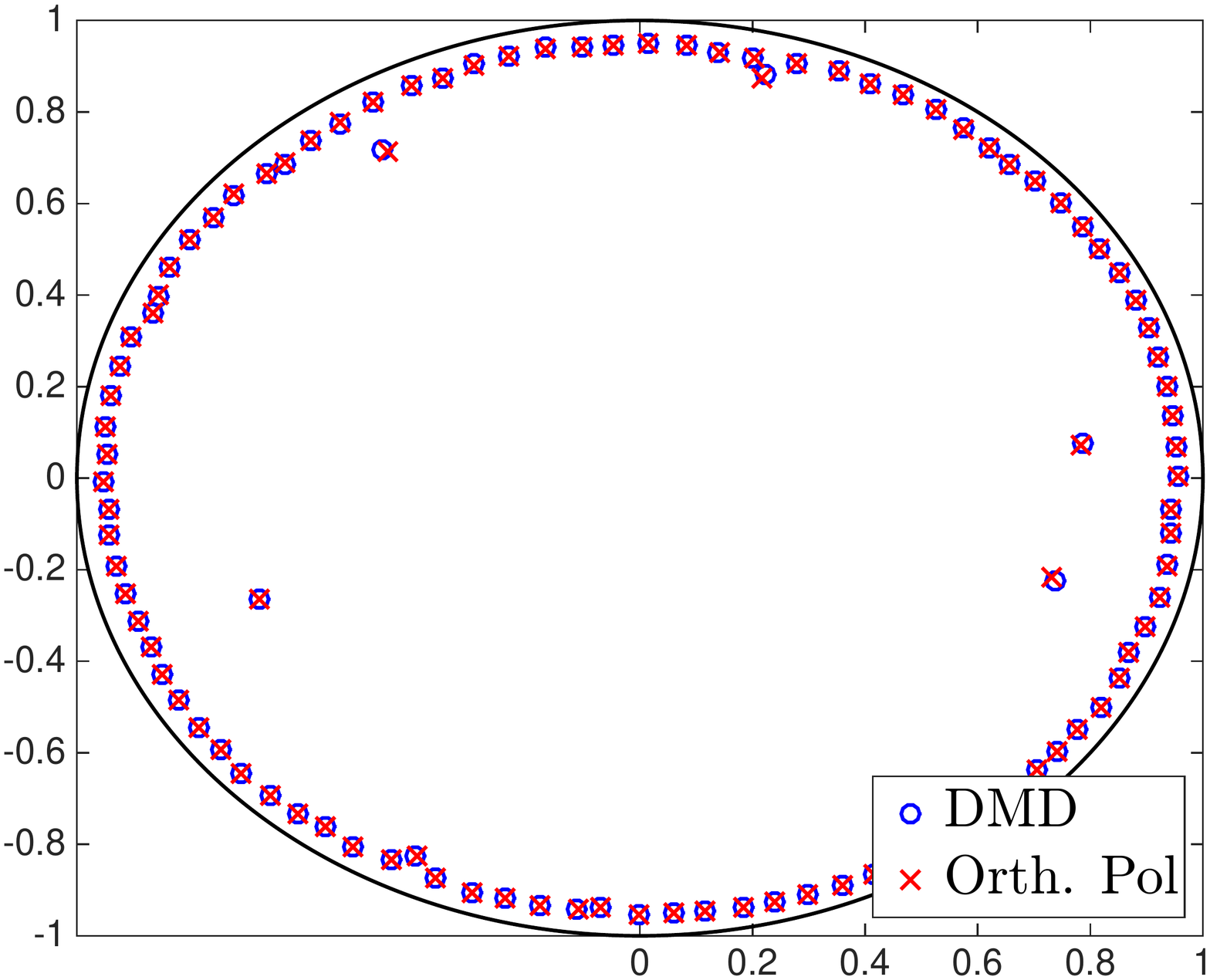}}

\put(105,0){\footnotesize Real part}
\put(320,0){\footnotesize Real part}
\put(20,45){\rotatebox{90}{\footnotesize Imaginary part}}
\put(235,45){\rotatebox{90}{\footnotesize Imaginary part}}
\put(105,147){\small $N = 20$}
\put(320,147){\small $N = 100$}


\end{picture}
\caption{\figCaptionSize Cat map -- Eigenvalues of the Hankel DMD operator and zeros of the $N$-th orthogonal polynomial of $\mu_f$. As predicted by Section~\ref{sec:dmd_relation}, the two sets of zeros almost coincide (the discrepancy being due to a finite number of samples taken) and lie strictly inside the unit circle $\Tb$, while approaching a uniform distribution on $
\Tb$ as $N$ increases. Both plots pertain to the observable $f_2$.}
\label{fig:catMap_zeros}
\end{figure*}

\begin{figure*}[!th]
\begin{picture}(140,380)
\put(0,0){\includegraphics[width=55mm]{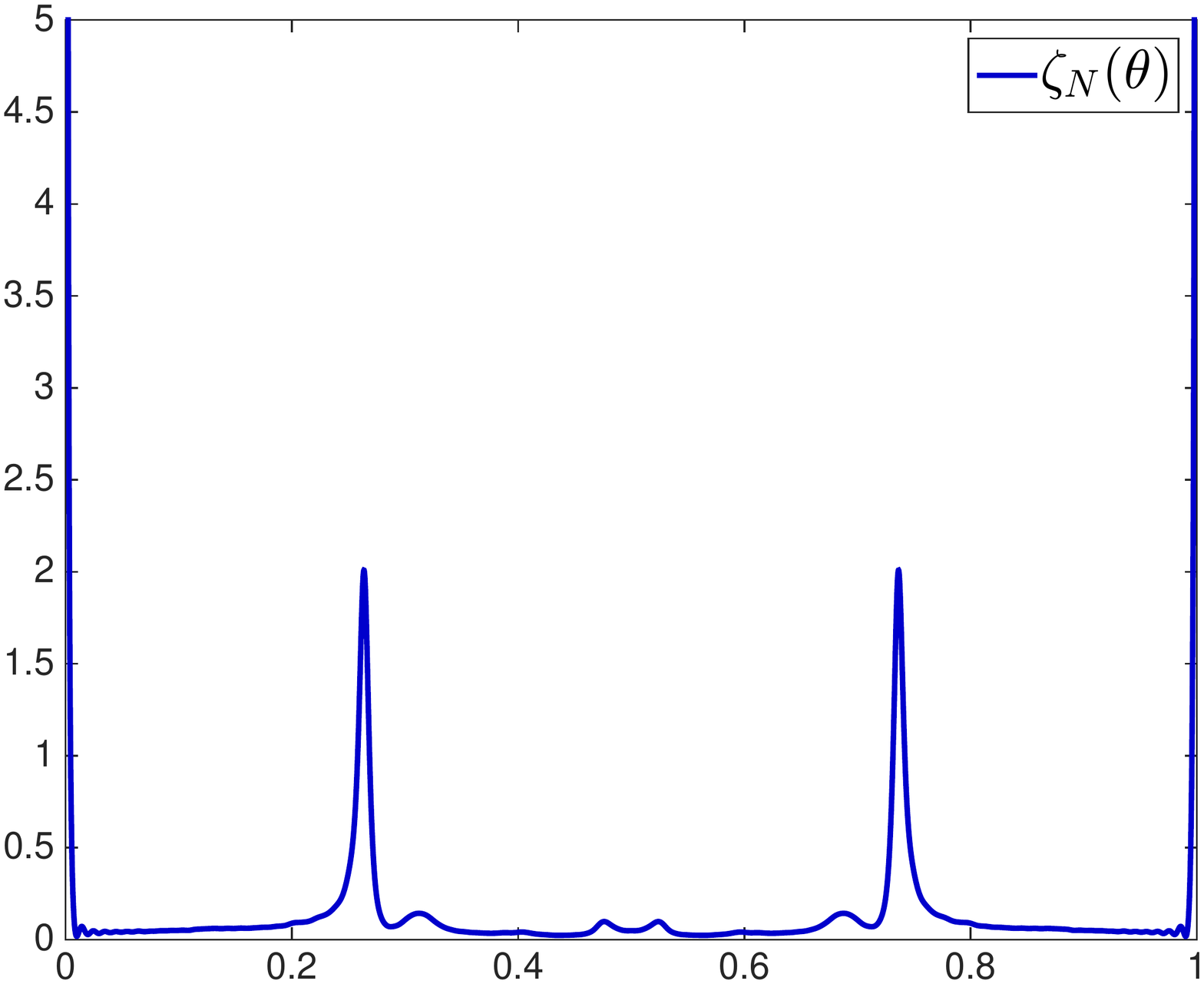}}
\put(150,0){\includegraphics[width=55mm]{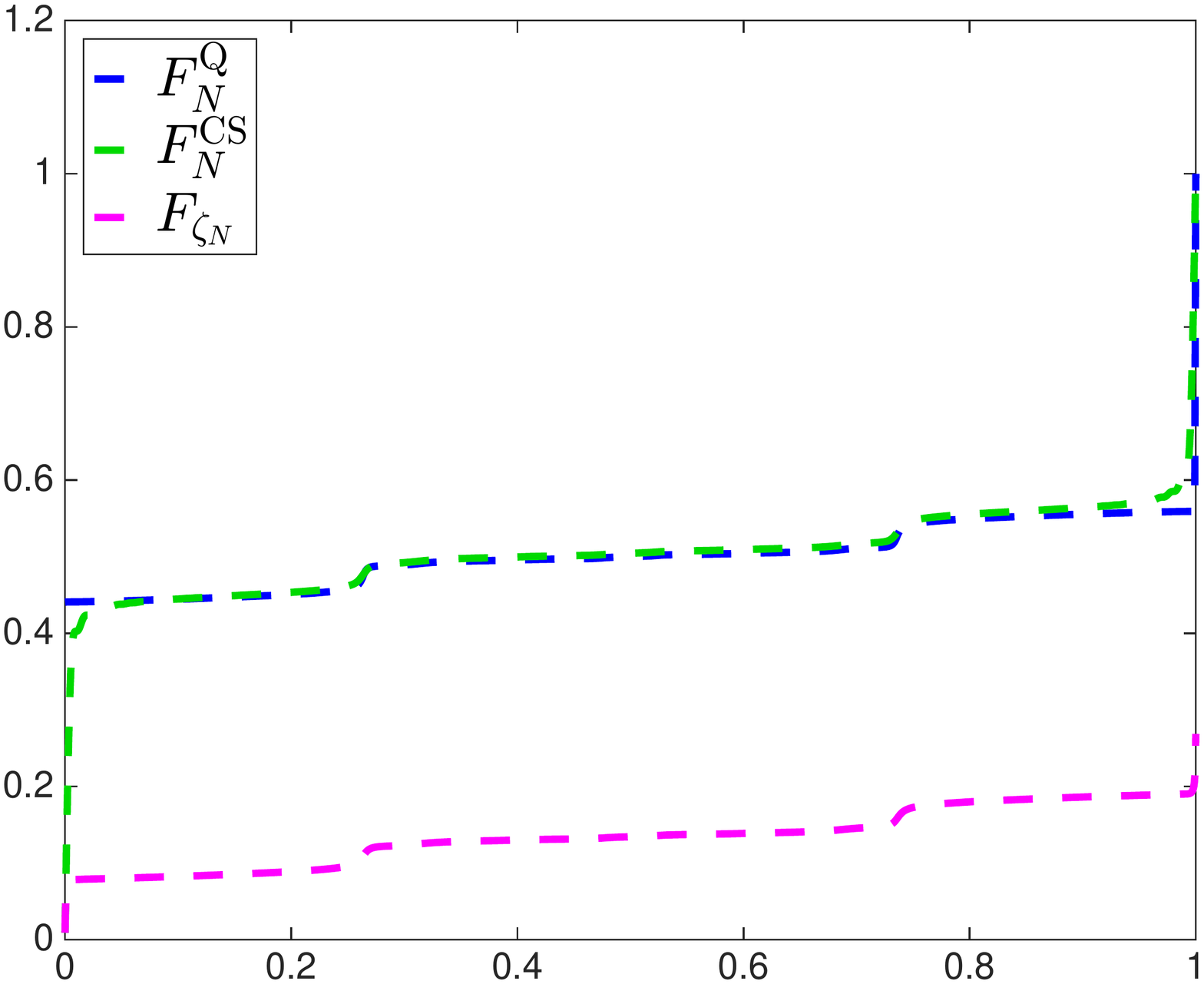}}
\put(300,0){\includegraphics[width=55mm]{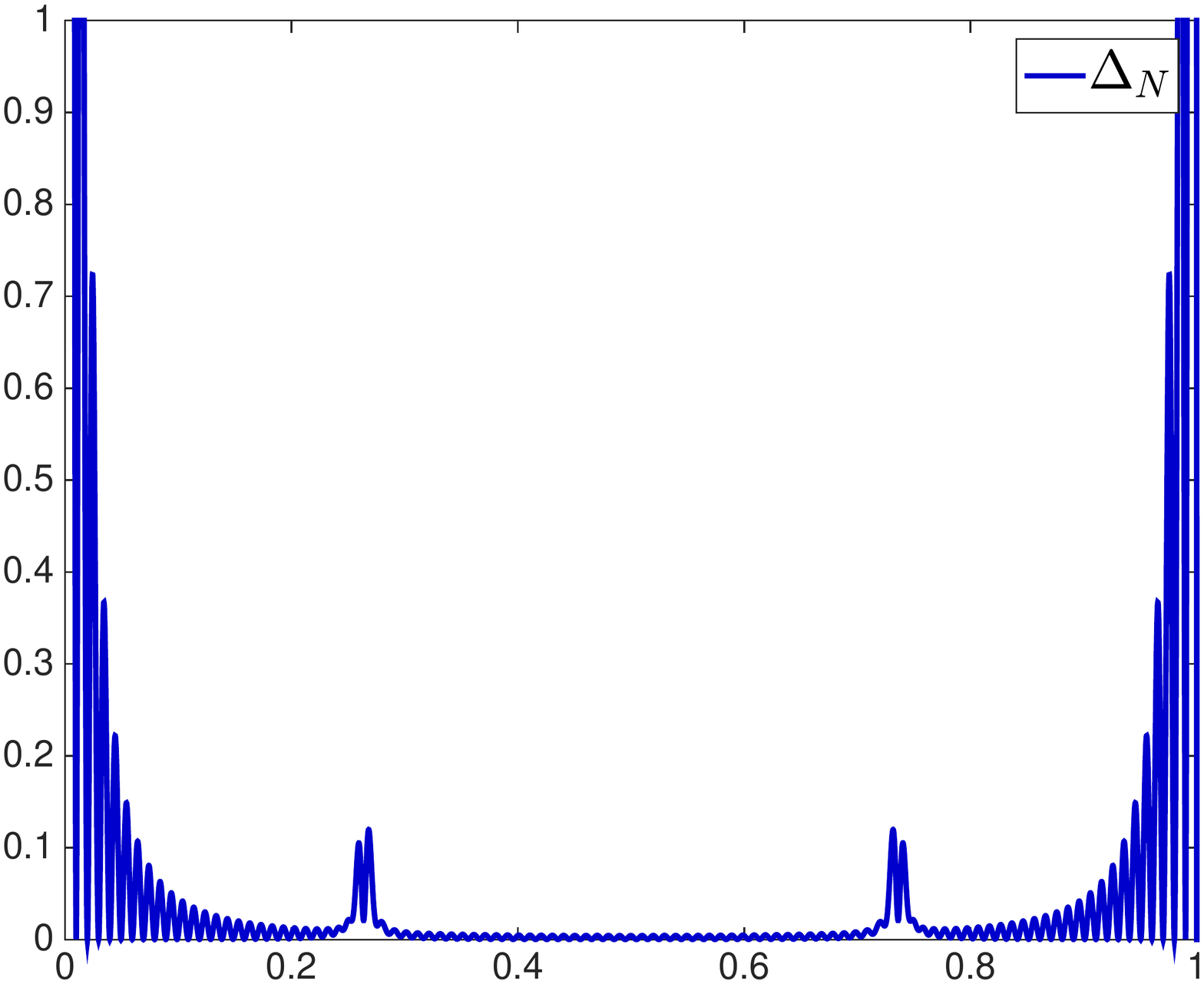}}

\put(0,120){\includegraphics[width=55mm]{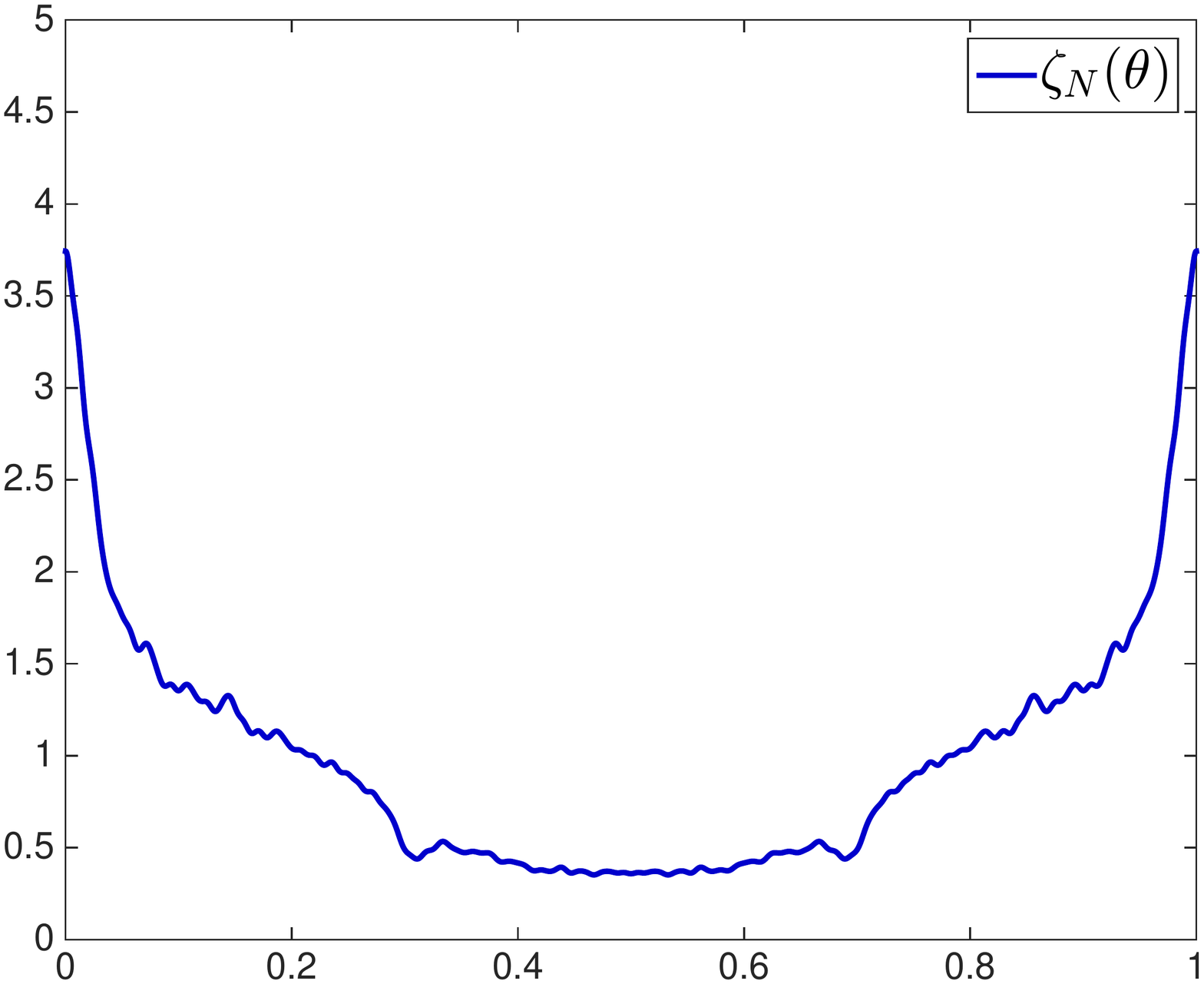}}
\put(150,120){\includegraphics[width=55mm]{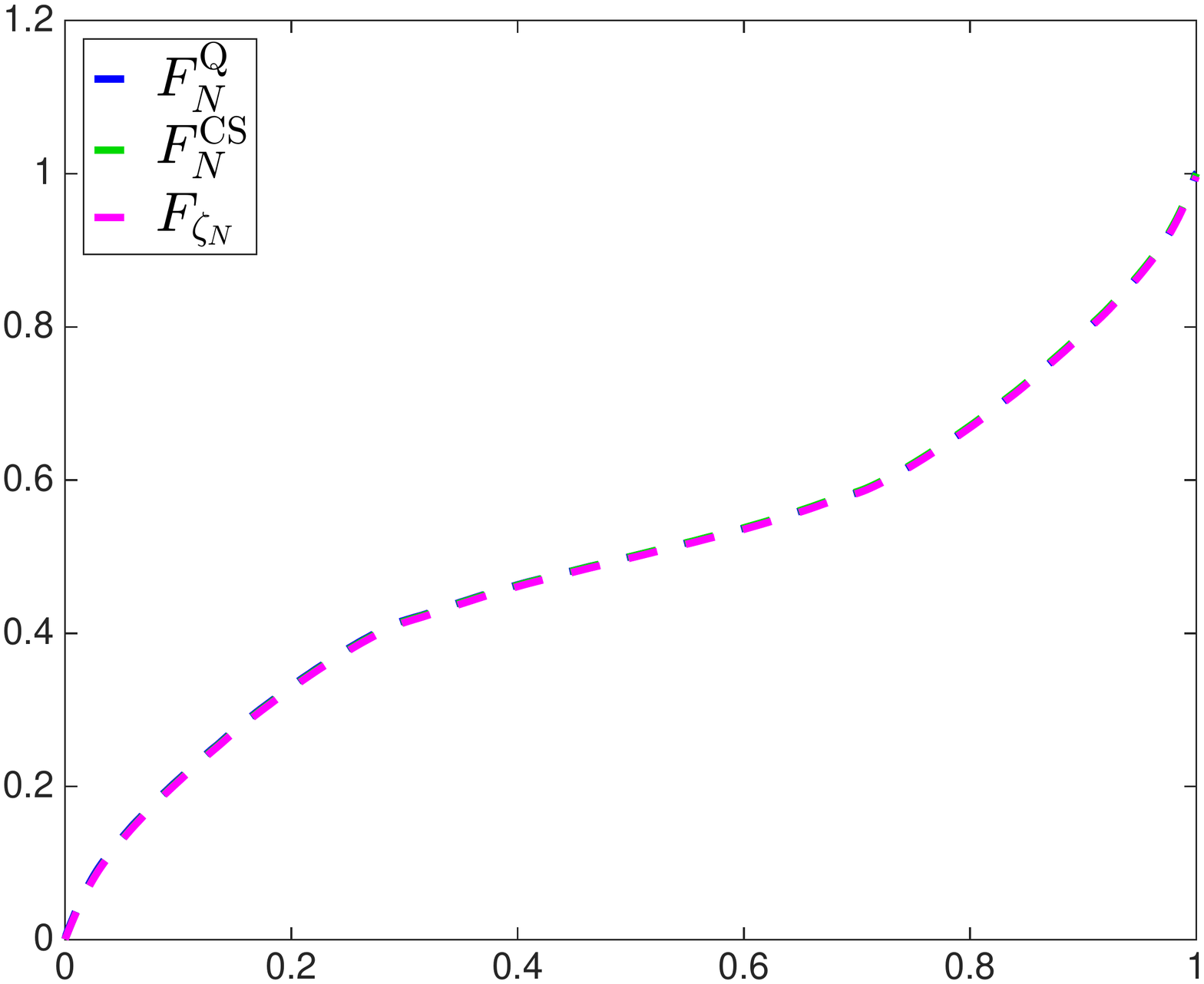}}
\put(300,120){\includegraphics[width=55mm]{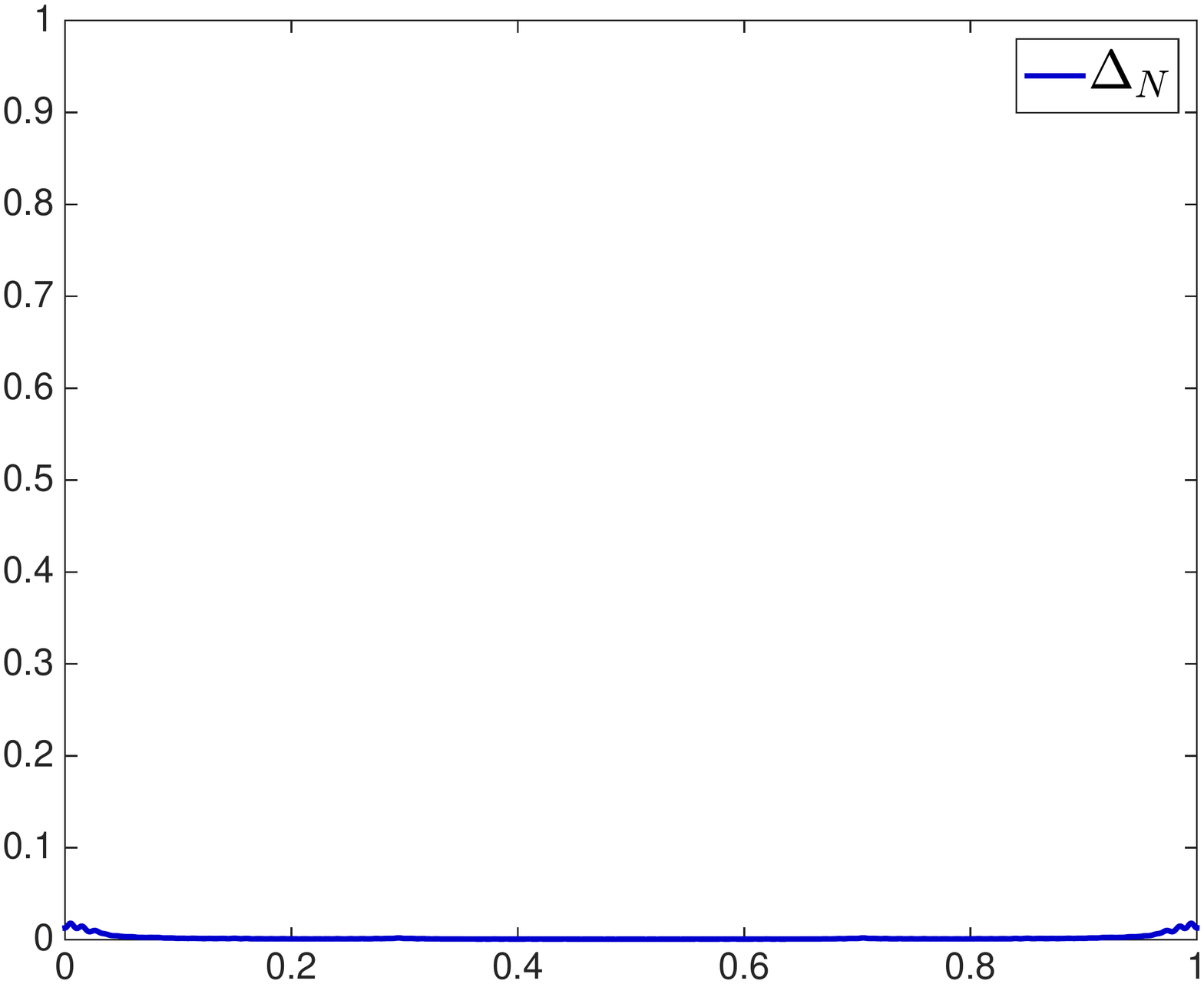}}

\put(0,240){\includegraphics[width=55mm]{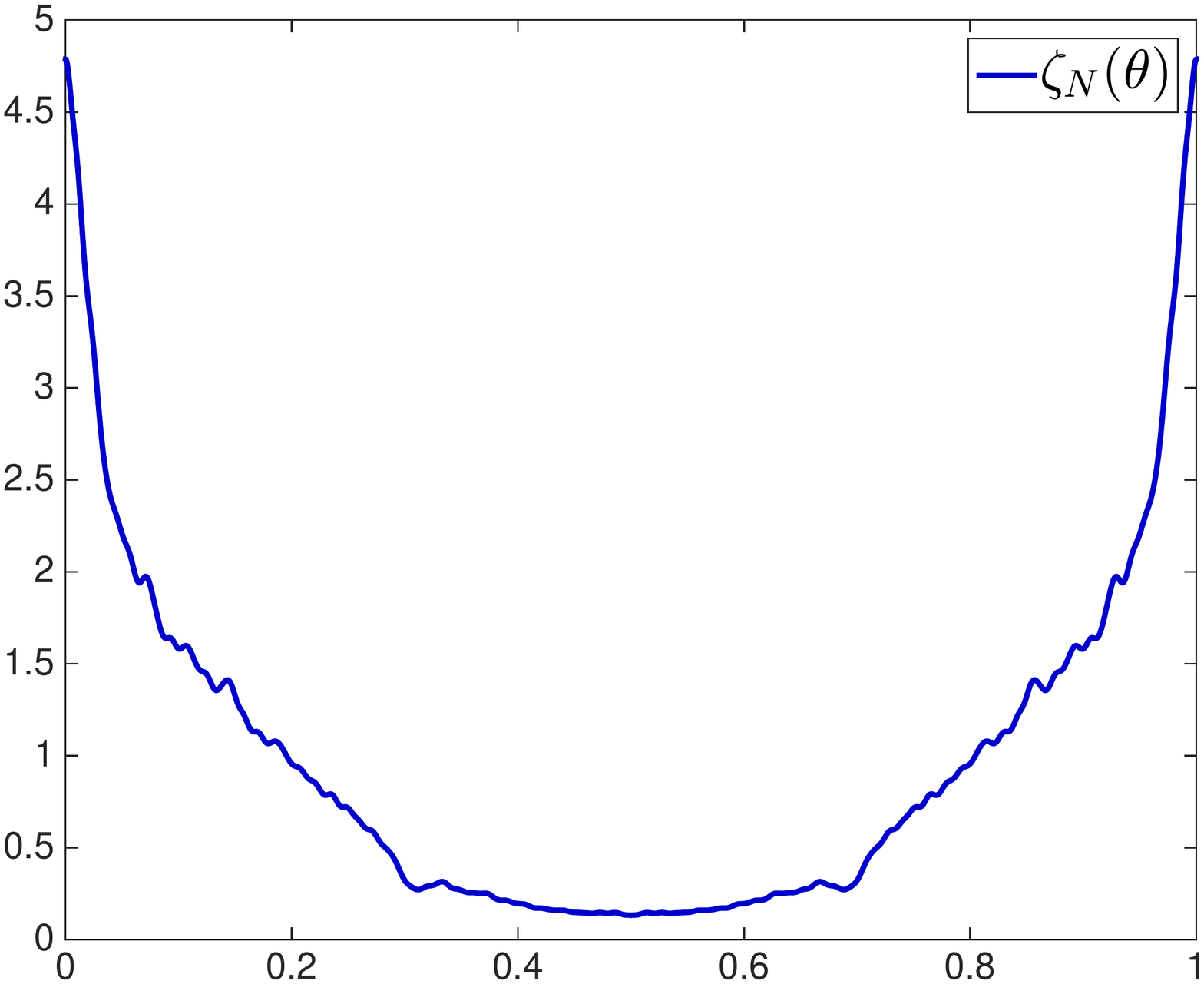}}
\put(150,240){\includegraphics[width=55mm]{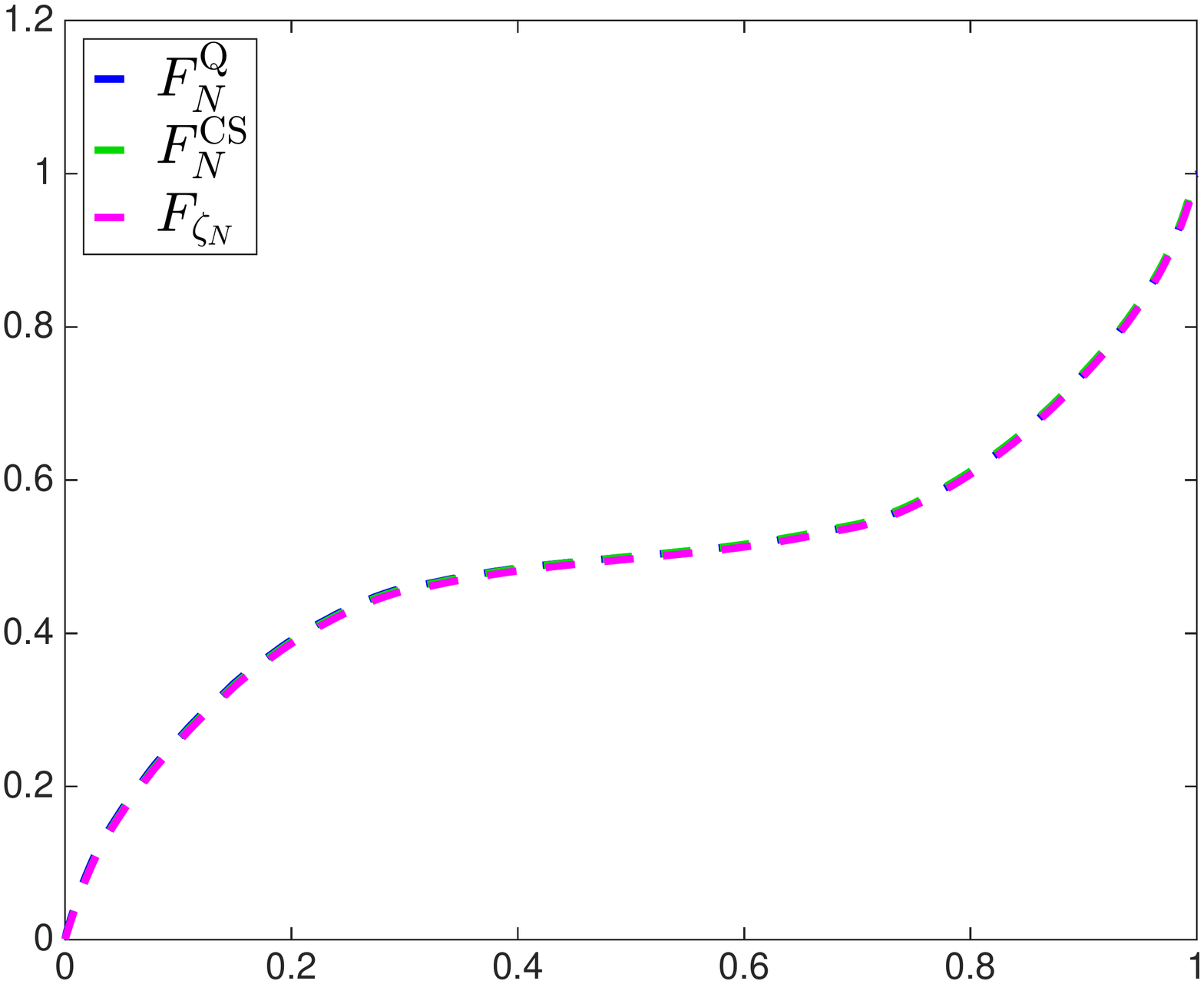}}
\put(300,240){\includegraphics[width=55mm]{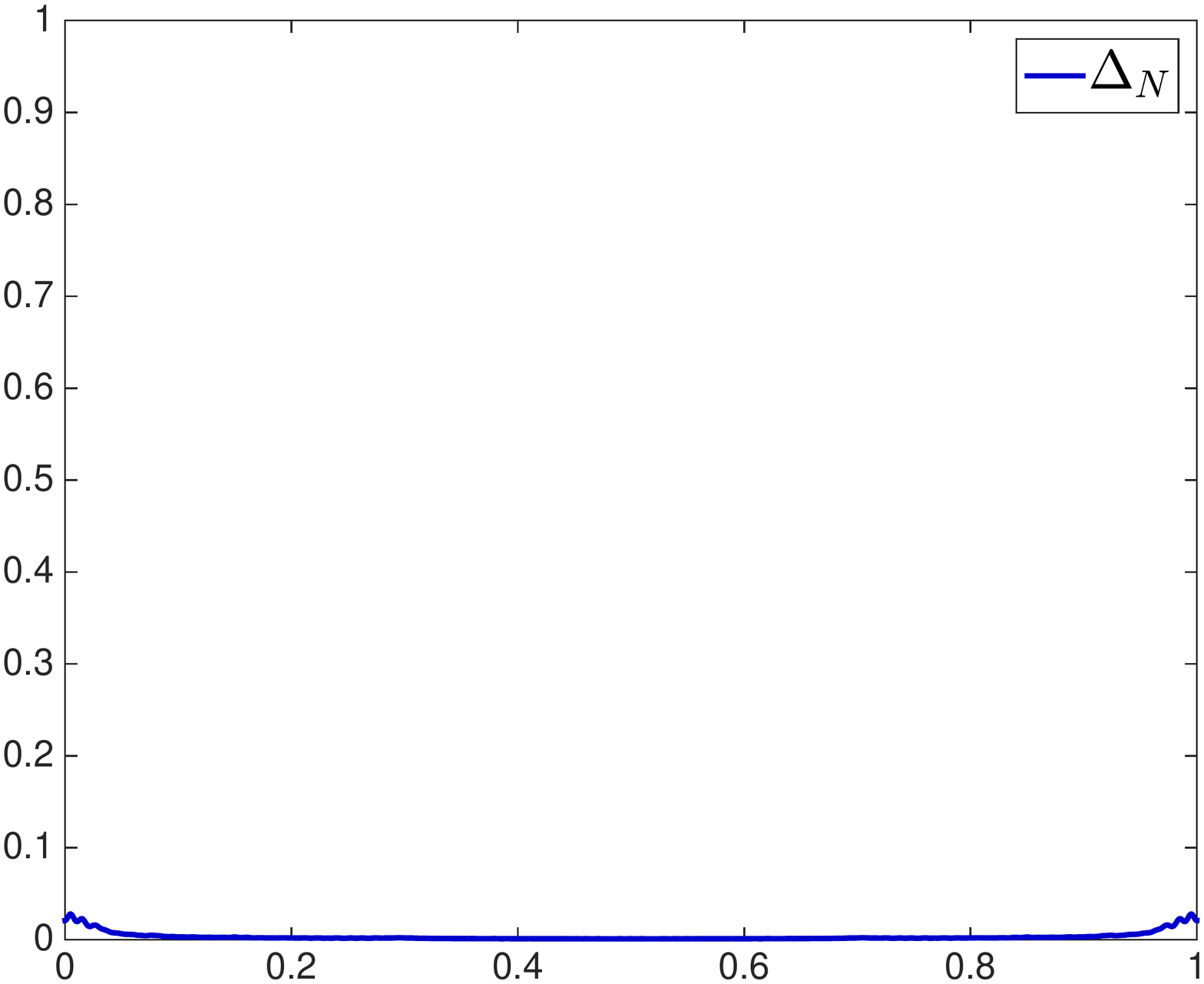}}

\put(79,0){\footnotesize $\theta$}
\put(229,0){\footnotesize $\theta$}
\put(379,0){\footnotesize $\theta$}

\put(1,270){\footnotesize \rotatebox{90}{Observable $x_1$}}
\put(1,150){\footnotesize \rotatebox{90}{Observable $x_2$}}
\put(1,30){\footnotesize \rotatebox{90}{Observable $x_3$}}

\end{picture}
\caption{\figCaptionSize Lorenz system -- $N = 100$. Left: density approximation with the CD kernel. Middle: distribution function approximation. Right: Singularity indicator $\Delta_N$ defined in~(\ref{eq:singul_indicator}).}
\label{fig:lorenz_x1x3x3}
\end{figure*}

\begin{figure*}[!th]
\begin{picture}(140,120)

\put(0,0){\includegraphics[width=55mm]{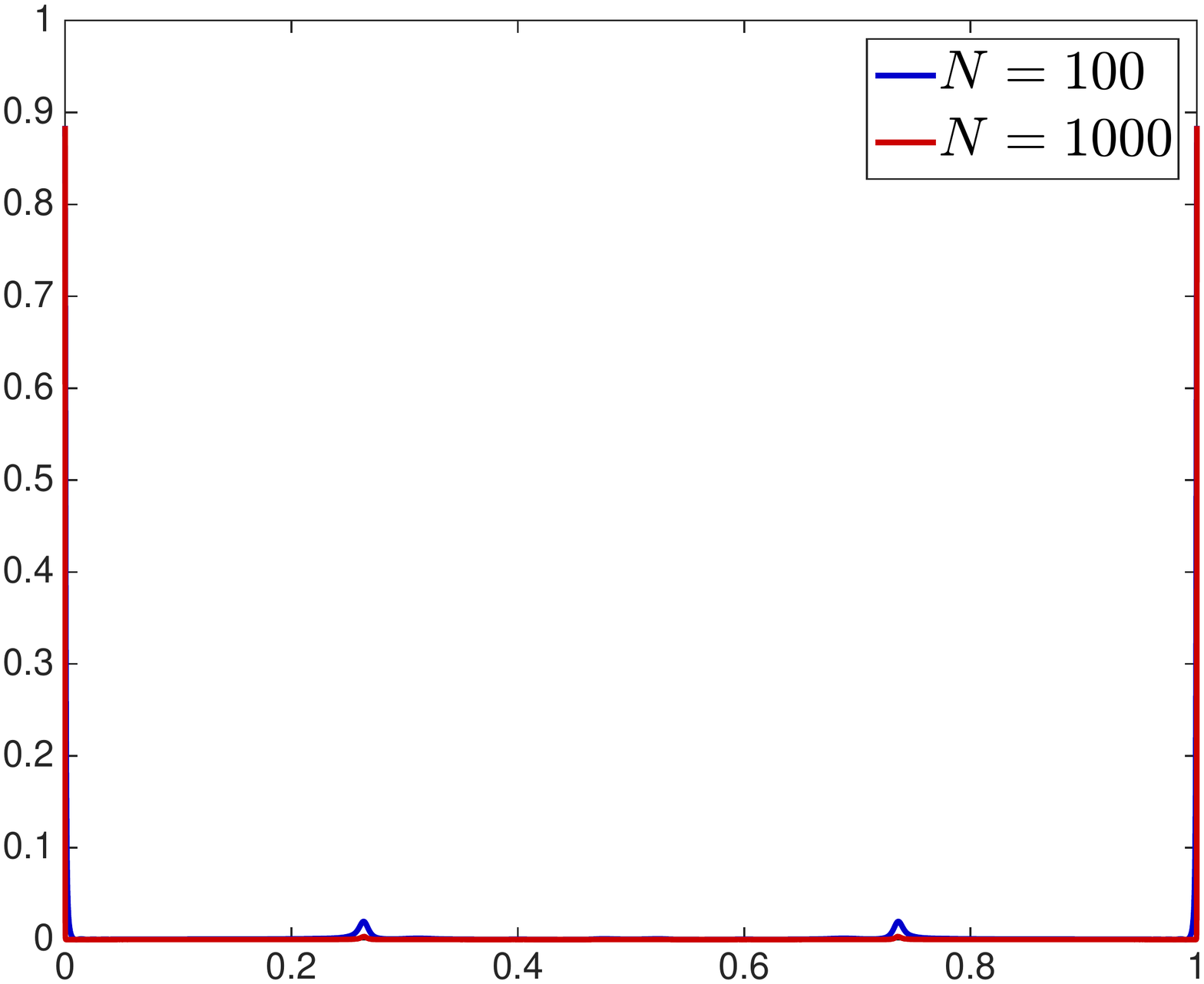}}
\put(150,0){\includegraphics[width=55mm]{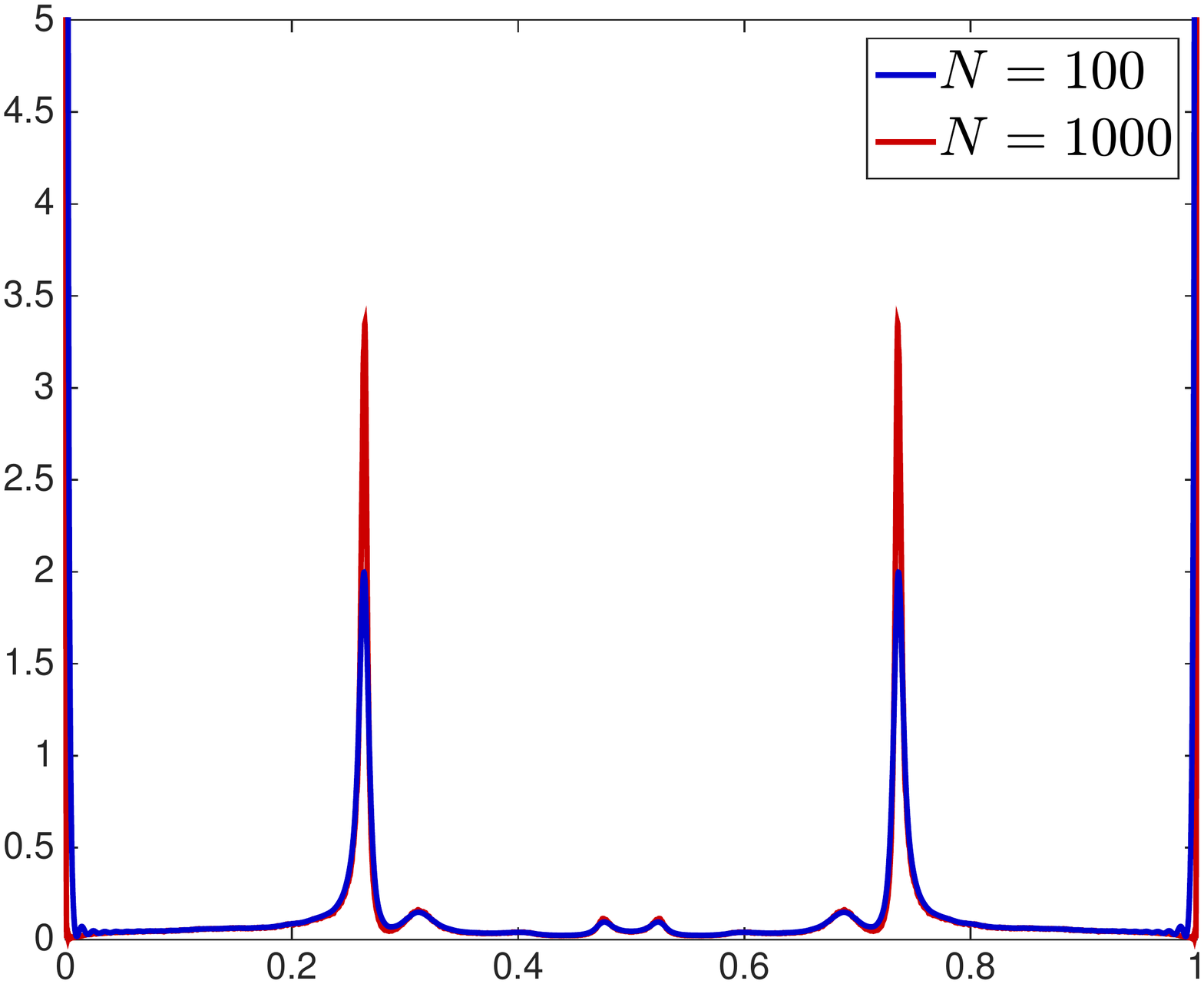}}
\put(300,0){\includegraphics[width=55mm]{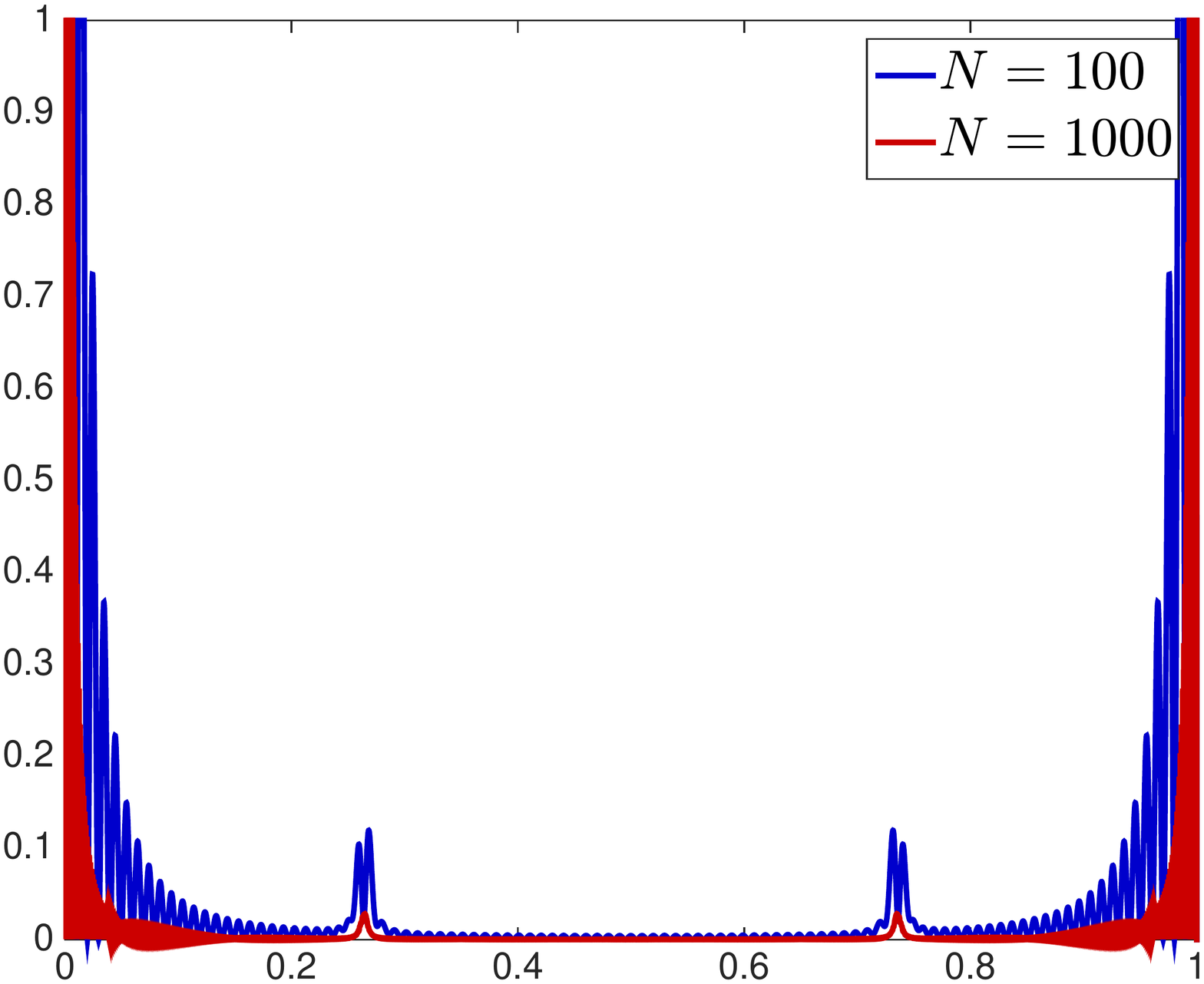}}

\put(25,25){\includegraphics[width=31mm]{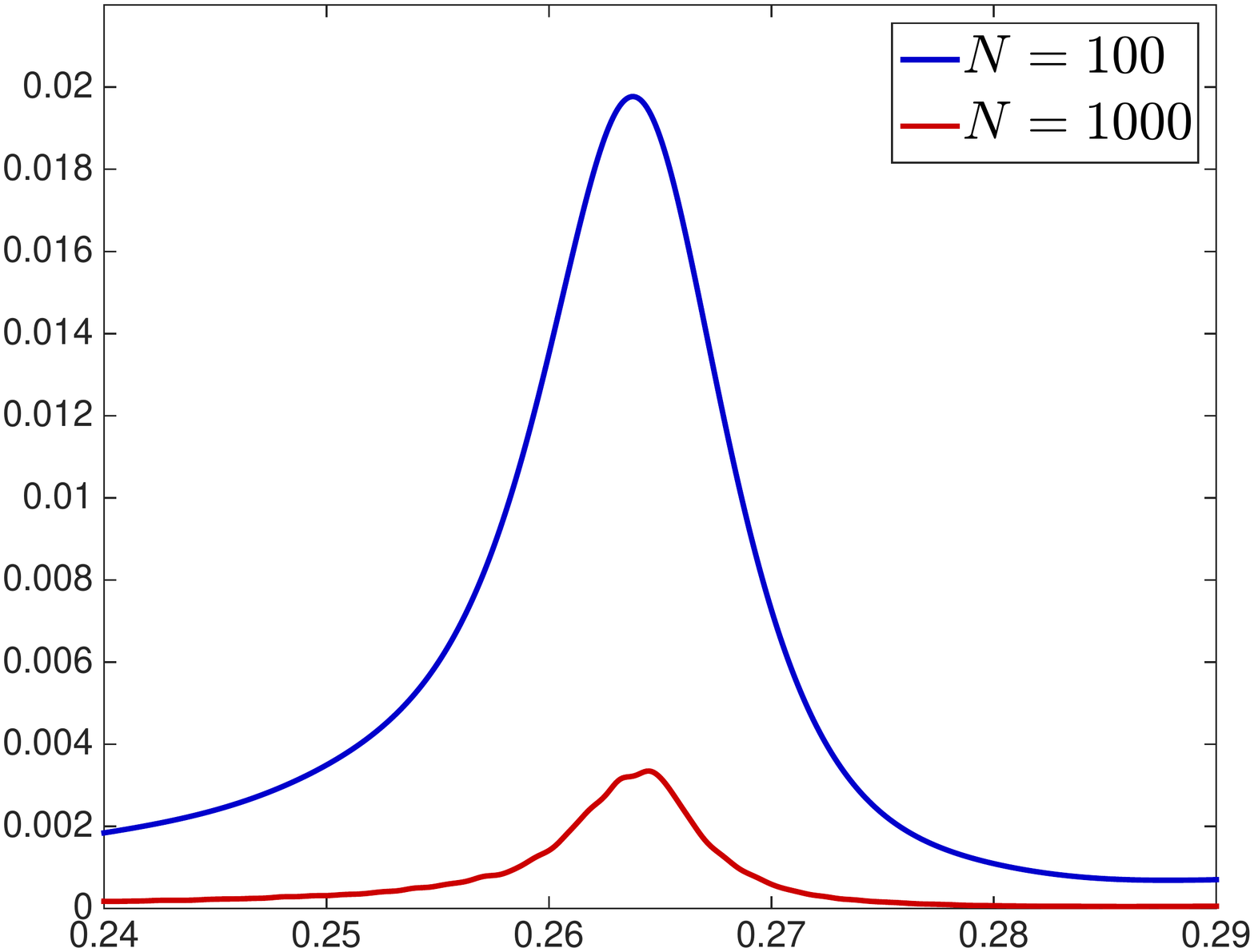}}
\put(325,25){\includegraphics[width=31mm]{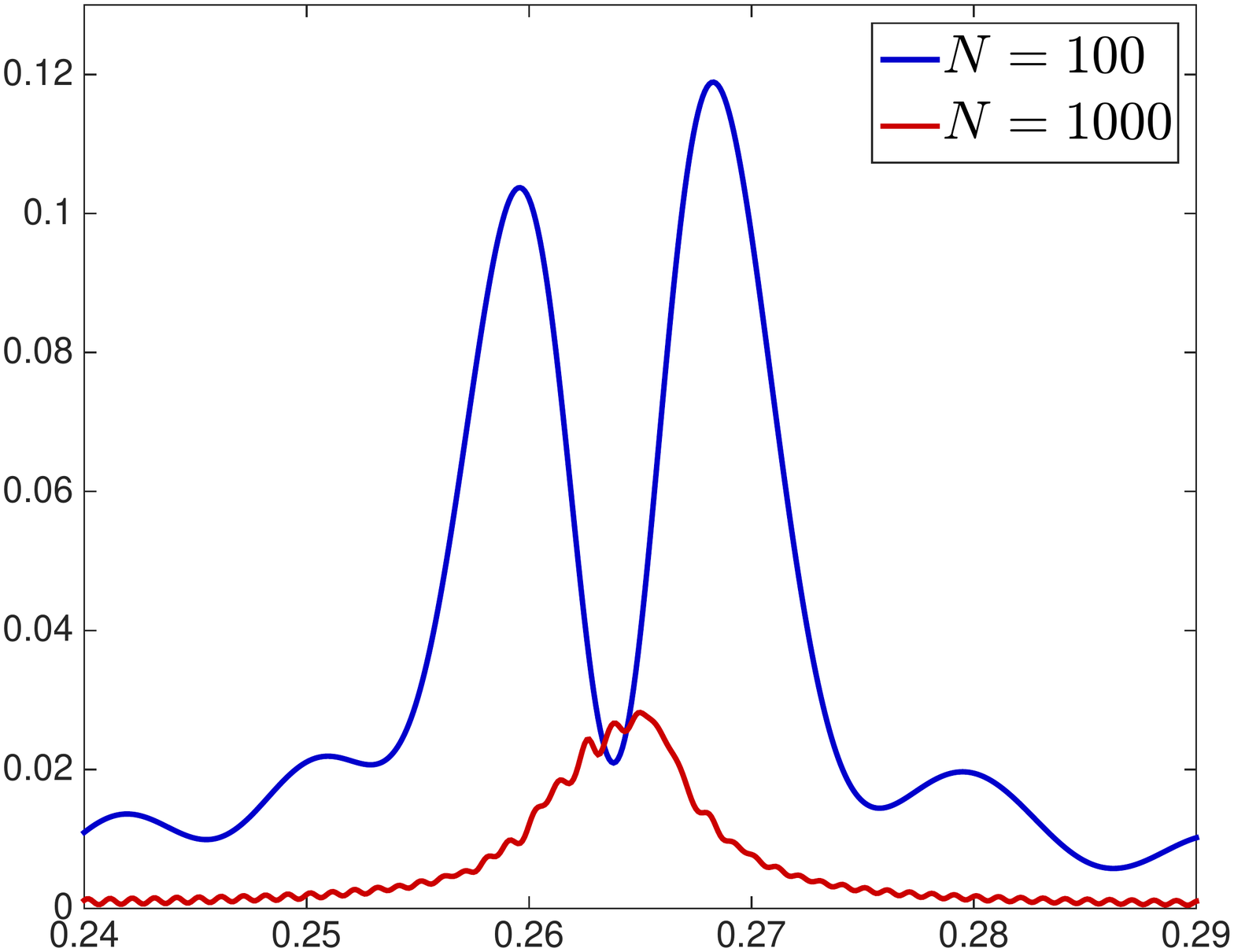}}
\put(70, 30){\vector(-1, -1){15}}
\put(370, 30){\vector(-1, -1){14}}

\put(79,0){\footnotesize $\theta$}
\put(229,0){\footnotesize $\theta$}
\put(379,0){\footnotesize $\theta$}

\end{picture}
\caption{\figCaptionSize Lorenz system -- Observable $f(x) = x_3$. Left: approximation of the atomic part $\zeta_N / (N+1)$. Middle: approximation of the density $\zeta_N$. Right: singularity indicator $\Delta_N$.}
\label{fig:lorenz:twoDiffN}
\end{figure*}

\begin{figure*}[th]
\begin{picture}(140,195)
\put(10,0){\includegraphics[width=90mm]{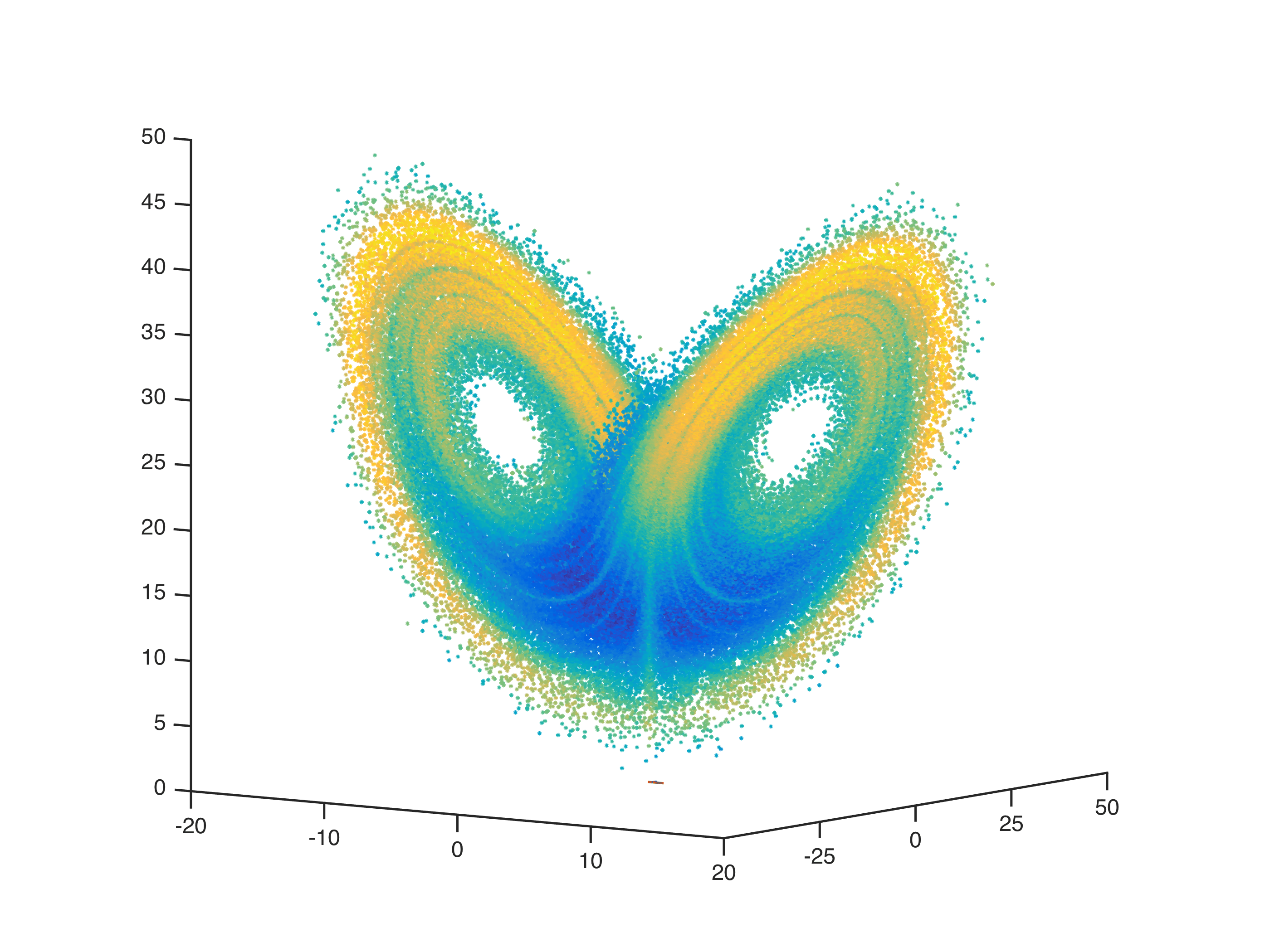}}
\put(225,0){\includegraphics[width=90mm]{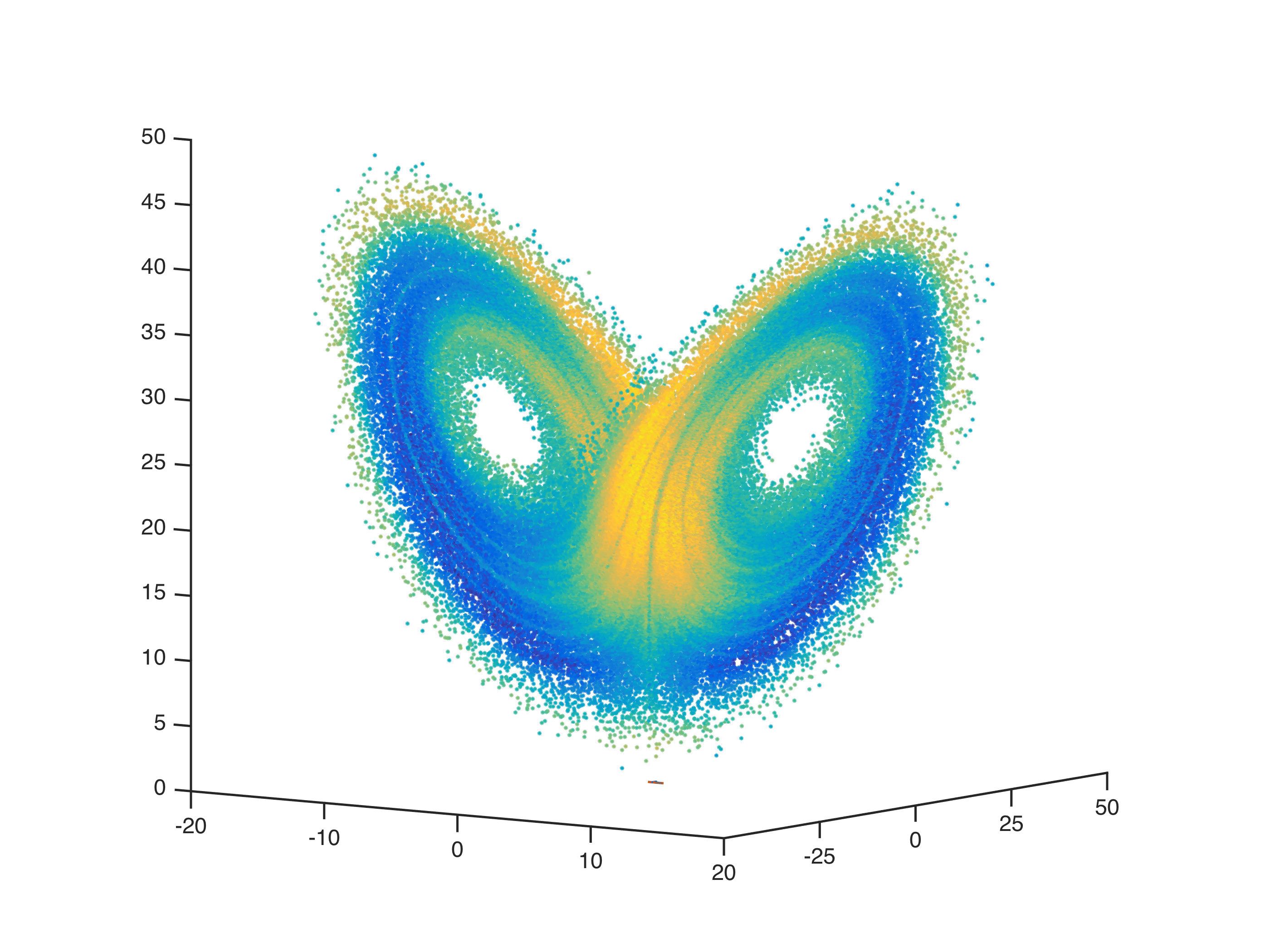}}

\put(92,7){\small $x_1$}
\put(307,7){\small $x_1$}

\put(184,10){\small $x_2$}
\put(405,10){\small $x_2$}

\put(22,90){\small $x_3$}
\put(235,90){\small $x_3$}
\put(118,170){\small Real part}
\put(323,170){\small Imaginary part}


\end{picture}
\caption{\figCaptionSize Lorenz system -- Approximation of the spectral projection $P_{[a,b)}f$ with $f(x) = x_3$ and $[a,b] = [0.24,0.28]$ and $N=100$, $M = 10^5$.}
\label{fig:lorenz_projection}
\end{figure*}

\subsection{Lorenz system}
Our second example investigates the Lorenz system.
\[
\begin{array}{llll}
\dot{x}_1 & = &  10(x_2-x_1) \\ 
\dot{x}_2 & = & x_1(28-x_3) - x_2 \\
\dot{x}_3 & = & x_1x_2 - \frac{8}{3}x_3.
\end{array}
\]
This is a continuous time system for which we define the Koopman semigroup $(U_t f)(x) = (f\circ \varphi_t)(x) $, where $\varphi_t$ is the flow of the dynamical system. The data is collected in the following way: We simulate one trajectory of length $MT_s$, where $T_s = 0.2$ is the sampling period and $M$ is the number of samples to be taken. Given an observable $f$, the data~(\ref{eq:data}) used for approximate moment computation using~(\ref{eq:moment:approx}) is then given by equidistant sampling, i.e., $y_j = f(jT_s)$. Therefore, our results on spectrum approximation pertain to $U_{T_s}$ (i.e., to one element of the Koopman semigroup). Figure~\ref{fig:lorenz_x1x3x3} shows the results for three different observables, $f_1 = x_1$, $f_2 = x_2$ and $f_3 = x_3$ with $N = 100$. For the first two observables we see purely absolutely continuous spectrum, as testified by the agreement of $F_{\zeta_N}$ with the remaining two estimates of the distribution function as well as by the singularity indicator $\Delta_N$ being very small. The observable $f = x_3$, on the other hand, has an atom at $\theta = 0$ as well as a peak of the density estimate at approximately $\theta = 0.26$, corresponding to the continuous time frequency $\omega = 2\pi \theta / T_s \approx 8.17\,\mr{rad}/\mr{s}$. The singularity indicator $\Delta_N$ suggests there may be a small singularity around this location. To investigate this further we compare in Figure~\ref{fig:lorenz:twoDiffN} the plots of the estimates of the atomic part $\zeta_N/(N+1)$ as well as the density estimate $\zeta_N$ and the singularity indicator $\Delta_N$ for $N= 100$ and $1000$. First we notice that the estimate of the atomic part decreases  around $\theta = 0.26$ as $N$ increases, which suggests that the peak around this point is not an atom. We also observe a decrease of the singularity indicator $\Delta_N$ expect for a very small neighborhood of the peak. This suggests that the peak is either purely absolutely continuous or that there may be a very small singular continuous contribution. This is in agreement with the fact that the Lorenz system is mixing (see \cite{luzzatto2005lorenz}) and hence there are no non-trivial eigenvalues of the Koopman operator. We believe that the peak is associated with an almost-periodic motion of the $x_3$ component during the time that the state resides in either of the two lobes, with switches between the lobes occurring in a chaotic manner. In Figure~\ref{fig:lorenz_projection} we depict the approximation of the spectral projection  $P_{[a,b)}f$ (see Section~\ref{sec:projComp}) with $[a,b) = [0.24,0.28]$  and $f(x) = x_3$, i.e., we are projecting on a small interval around the peak in the spectrum of $x_3$. This function will evolve almost linearly with frequency of the peak, i.e., $(P_{[a,b)}f)(x(t + \tau))\approx e^{i\omega\tau }(P_{[a,b)}f)(x(t)) $ with $\omega \approx 8.17 \,\mr{rad}/\mr{s}$.

\subsection{Cavity flow}
In this example we study the 2-D model of a lid-driven cavity flow; see~\cite{arbabi2017study} for a detailed description of the example and the data generating process. As in \cite{arbabi2017study}, the goal is to document the changes in the spectrum of the Koopman operator with increasing Reynolds number which are manifestations of the underlying bifurcations, going from periodic through quasi-periodic to fully chaotic behavior. For each Reynolds number,  the data available to us is in the form of the so called stream function of the flow evaluated on a uniform grid of points in the 2-D domain with equidistant temporal sampling. This leaves us with a very large choice of observables since the value of the stream function at any of the grid points (as well as any nonlinear function of the values of the stream function) is a candidate observable. In general, one wishes to choose the observable $f$ such that its spectral content is as rich as possible, preferably such that $f$ is $\ast$-cyclic (see Eq~(\ref{eq:cyclic})), which is, however difficult to test numerically. For example, for $Re = 13\cdot 10^3$, exhibit periodic behavior with a single (or very dominant) harmonic component and hence might not contain the full spectral content of the operator (i.e., $f$ is not $\ast$-cyclic). Therefore, for each value of the Reynolds number we chose as the observable the stream function at a grid point where the time evolution is complex and hence the spectral content of this observable is likely to be rich. A more careful numerical study, such as the one carried out in~\cite{arbabi2017study}, should analyze a whole range of observables (perhaps the values of the stream function at all grid points). However, here, already one suitably chosen observable allows us to draw interesting conclusions on the behavior of the spectrum of the operator as a function of the Reynolds number. The point spectrum approximation results $\zeta_N/(N+1)$ are depicted in Figure~\ref{fig:cavity_pointspec}. Since the observable $f$ is real, the spectrum is symmetric around the point $\theta = 0.5$ and hence we depict it only for $\theta \in [0,0.5]$; in addition, we change coordinates from $\theta$ to $\omega = 2\pi \theta/T_s$, where $T_s = 0.5\, \mr{s}$ is the sampling period. Finally, in order to better discern very small atoms, we also show the point spectrum approximation on a logarithmic scale. Based on Theorem~\ref{thm:atoms}, whether or not there is an atom at a given frequency $\omega$, can be assessed based on the proximity of the values of $\zeta_{N}/(N+1)$ for two different $N$: When there is an atom, we expect the two values to be closed to each other; otherwise we expect the value of $\zeta_{N}/(N+1)$ to be significantly smaller since in that case $\zeta_{N}/(N+1) \to 0$. Figure~\ref{fig:cavity_pointspec} suggests that there is a very strong atomic component of the spectrum for $Re = 13\cdot 10^3$ and $Re = 16\cdot 10^3$ and even for $Re = 19\cdot 10^3$ as the atomic part accounts for at least $80\,\%$ of the energy of the given observable (i.e., $80\,\%$ of the mass of $\mu_f$). This is confirmed by the approximations of the distribution function which are piecewise constant for these values of the Reynolds number. For $Re = 30\cdot 10^3$, on the other hand, the spectrum appears to be purely continuous. In order to assess whether the spectrum is purely absolutely continuous or has a singular continuous part, we also plot the $F_{\zeta_N}$ which does not entirely coincide with $F_N^{\mr{Q}}$ and hence there may be a singular continuous component of the spectrum present. However, a larger data set  and more observables would have to be investigated in order to ascertain that.


\begin{figure*}[b!]
\begin{picture}(140,500)

\put(0,405){\includegraphics[width=55mm]{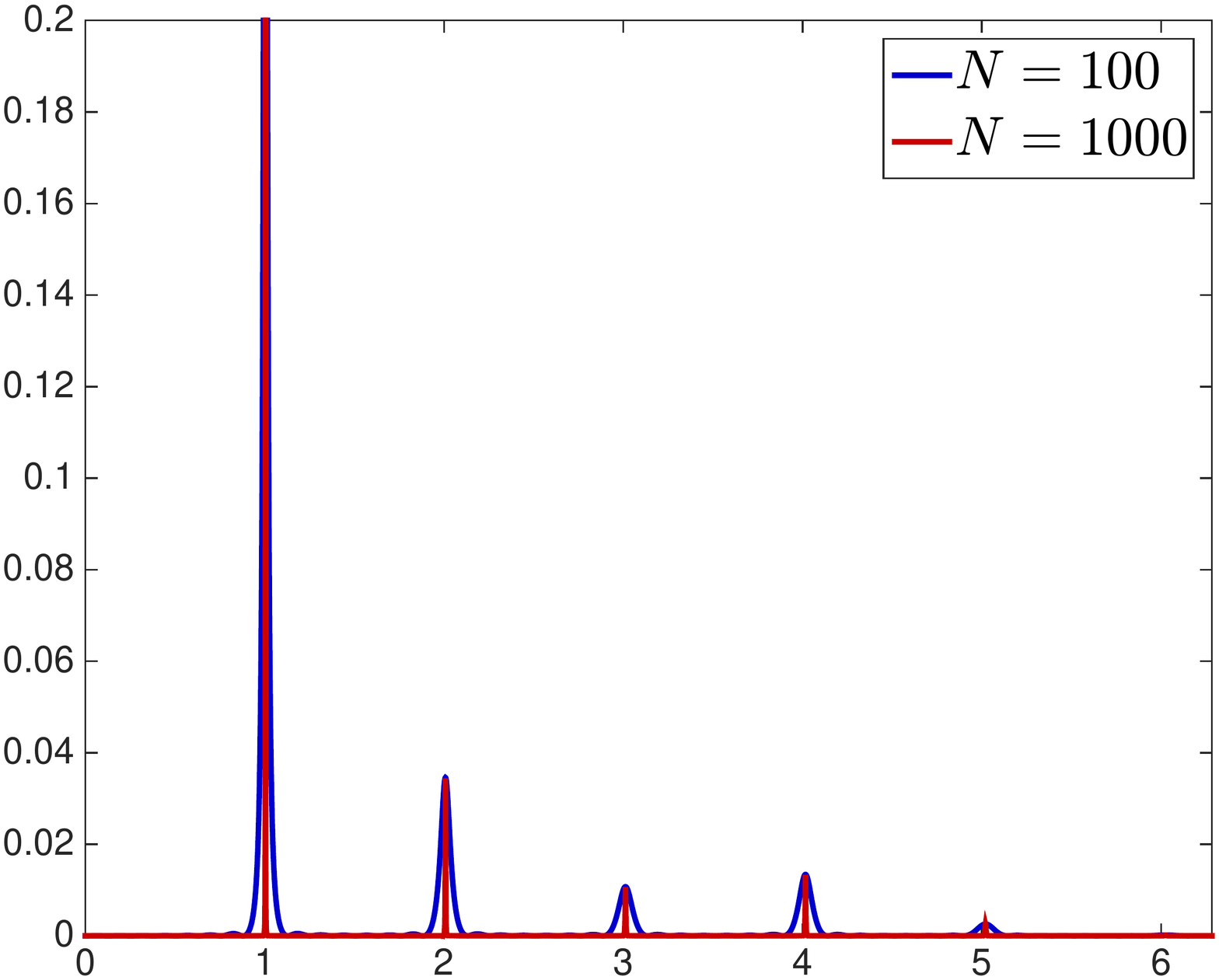}}
\put(150,405){\includegraphics[width=55mm]{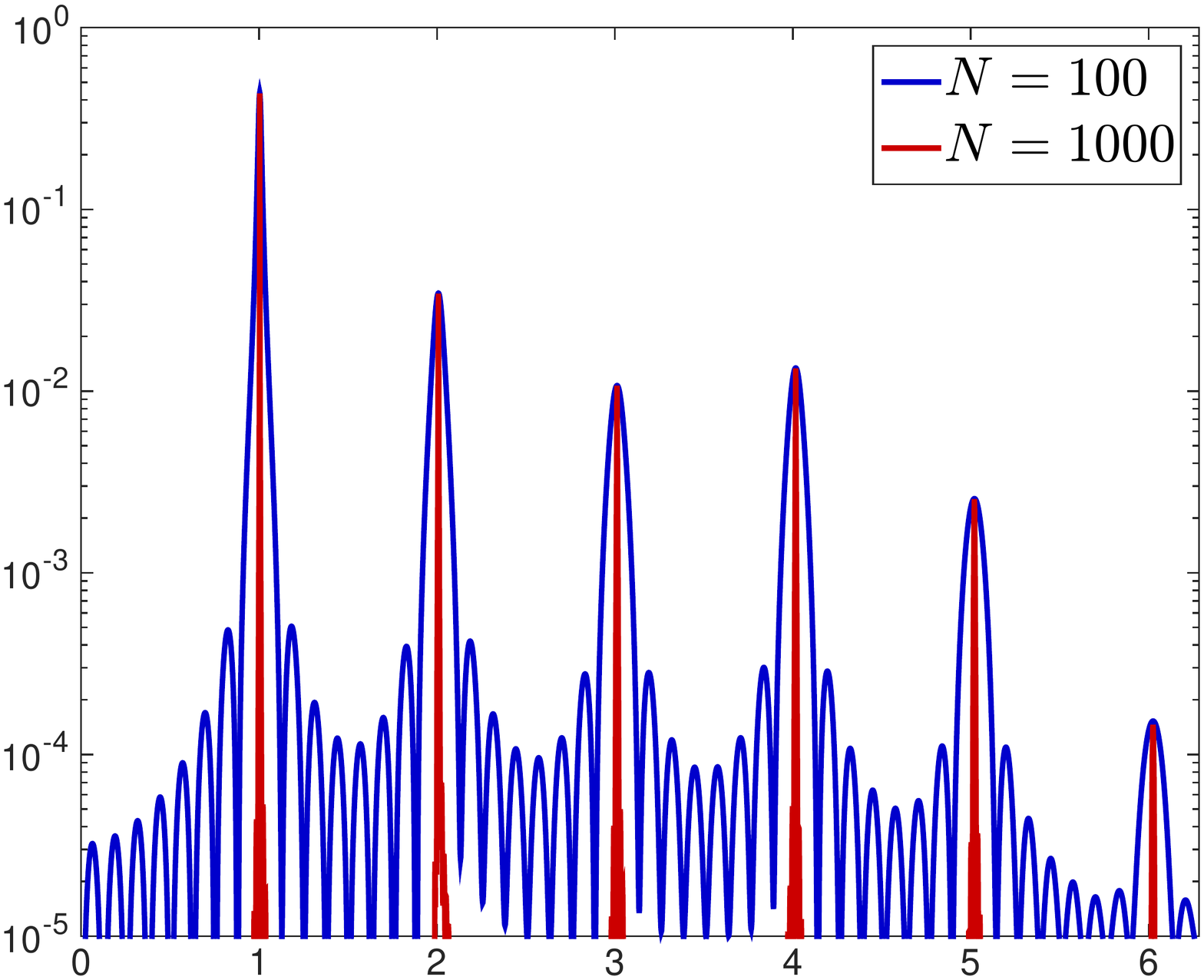}}
\put(300,405){\includegraphics[width=55mm]{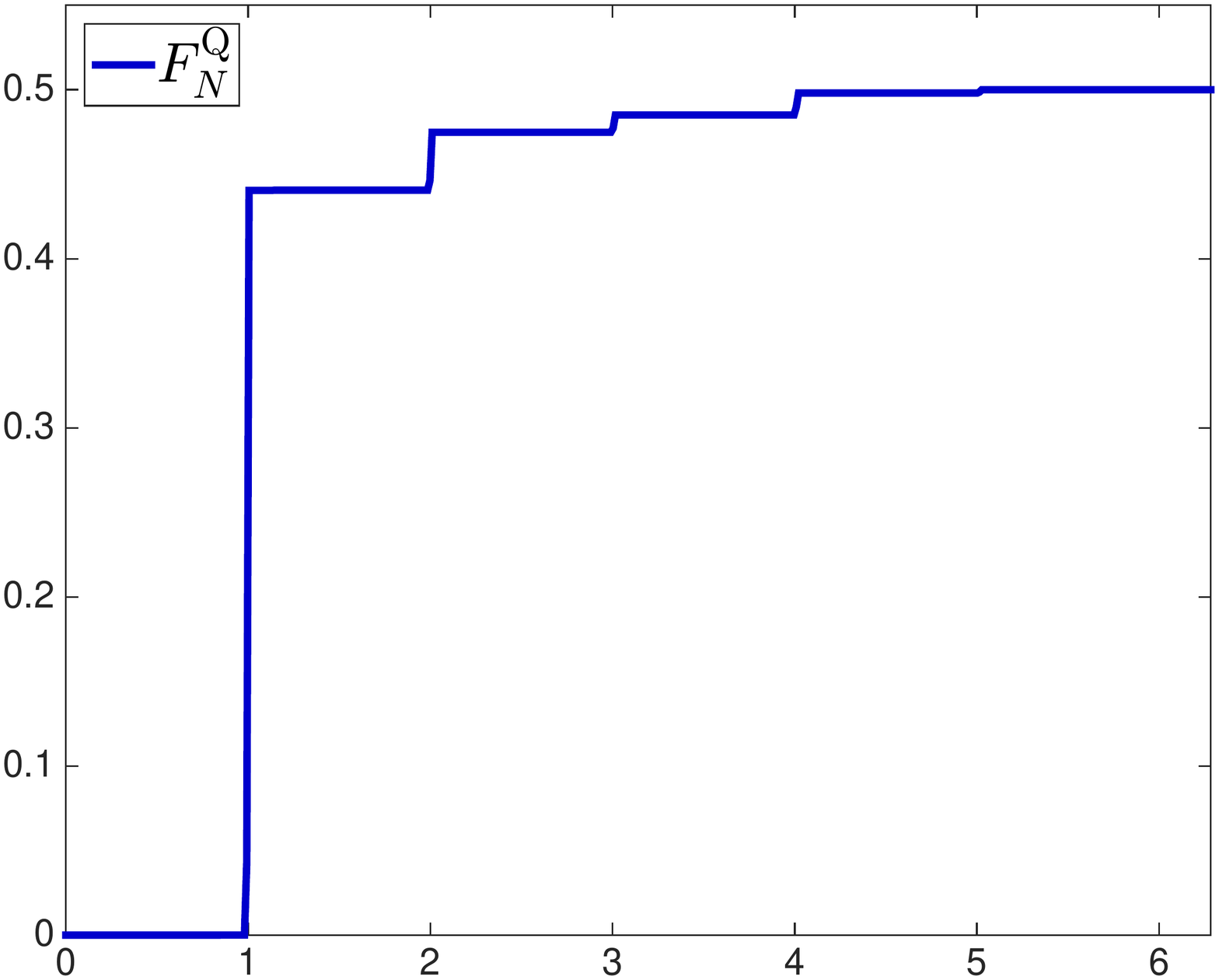}}

\put(0,270){\includegraphics[width=55mm]{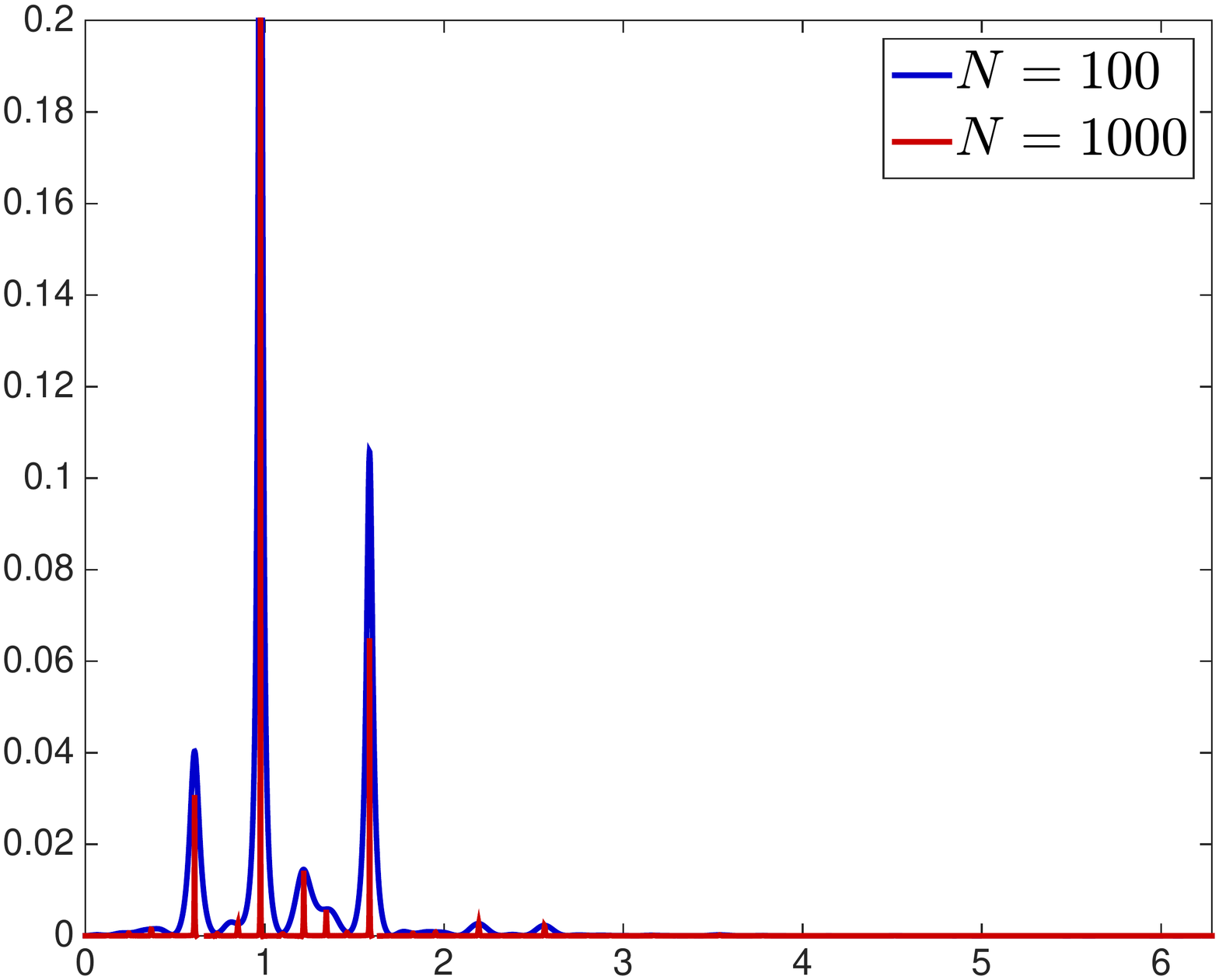}}
\put(150,270){\includegraphics[width=55mm]{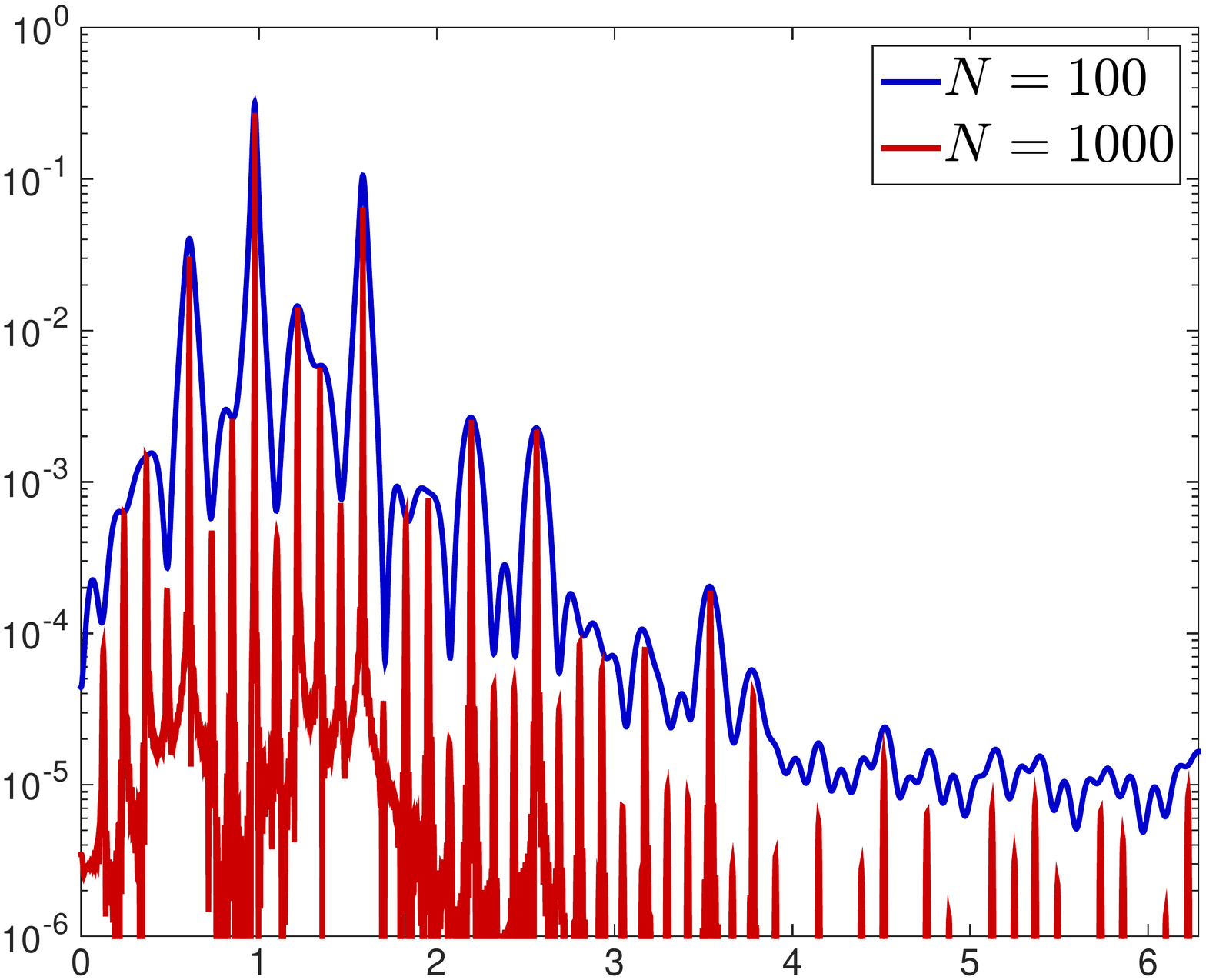}}
\put(300,270){\includegraphics[width=55mm]{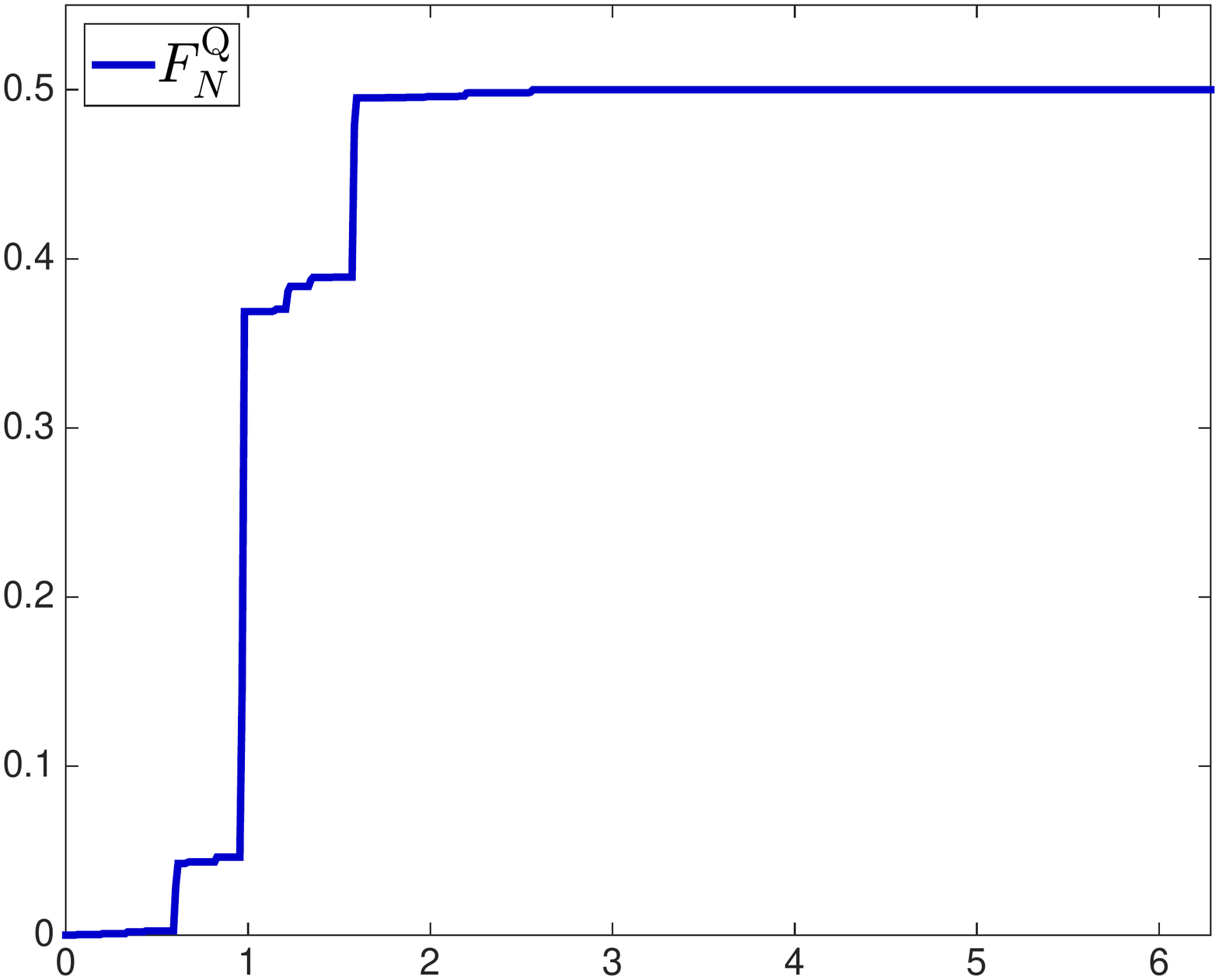}}

\put(57,295){\includegraphics[width=28mm]{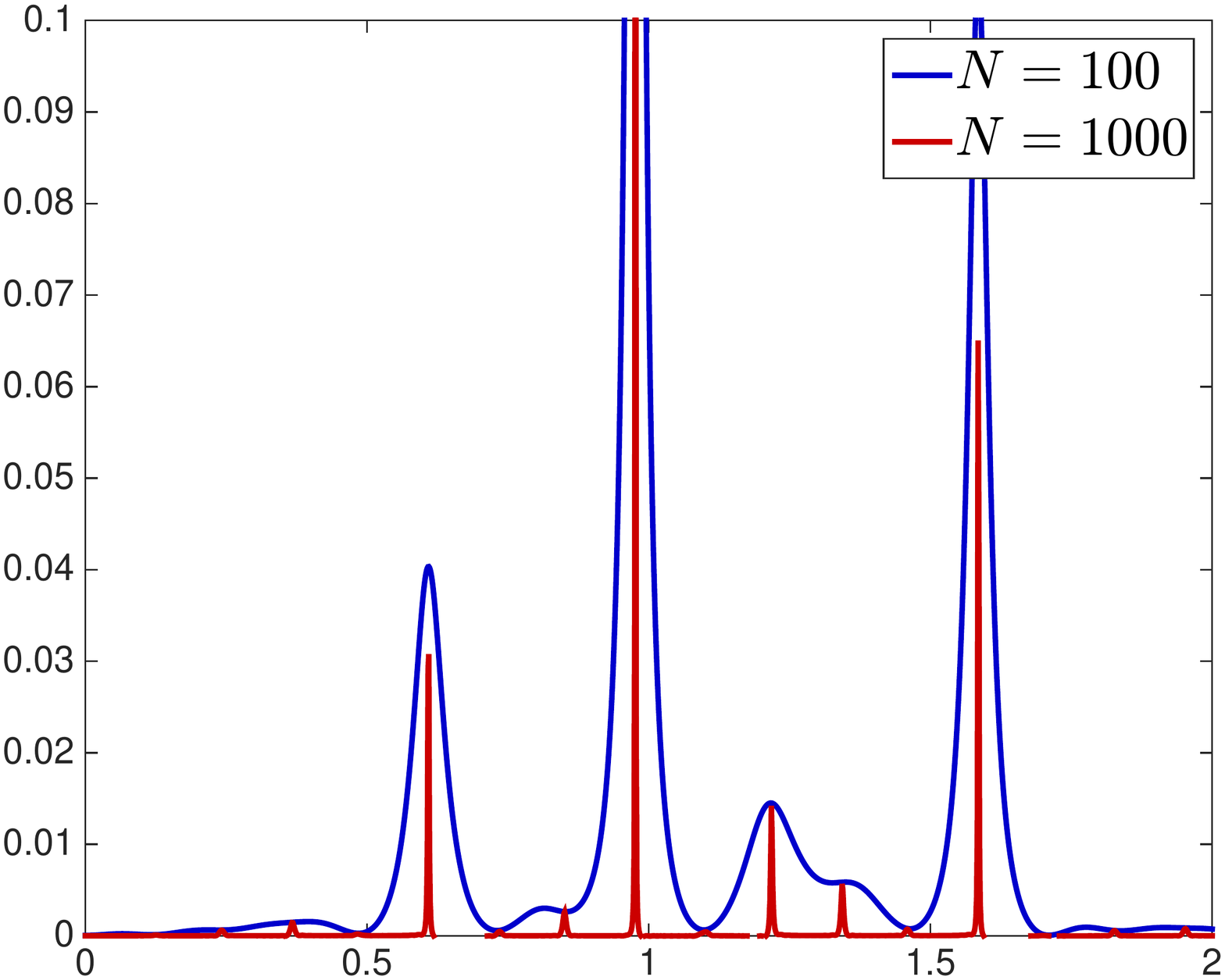}}
\put(65, 305){\vector(-1, -1){11}}

\put(0,135){\includegraphics[width=55mm]{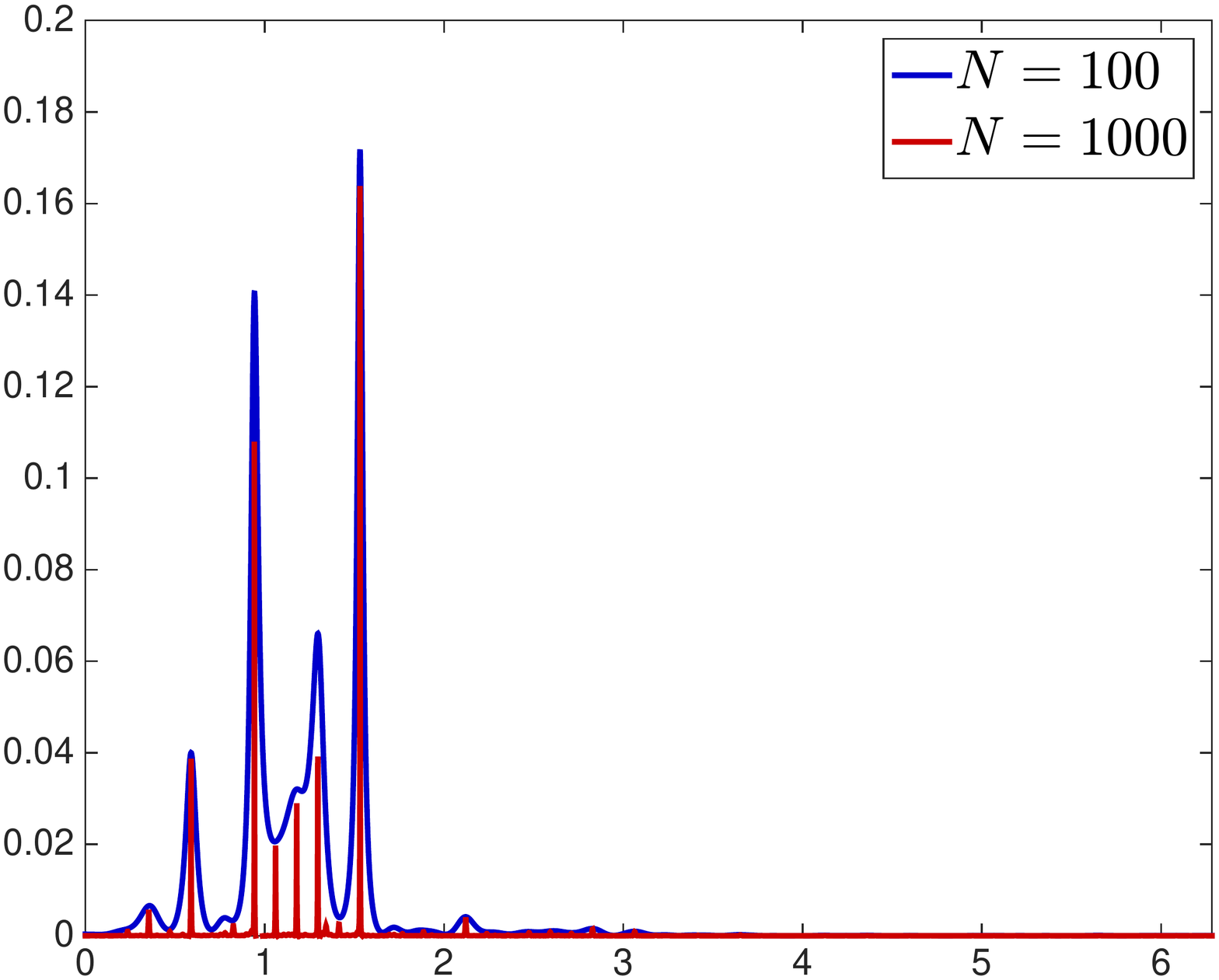}}
\put(150,135){\includegraphics[width=55mm]{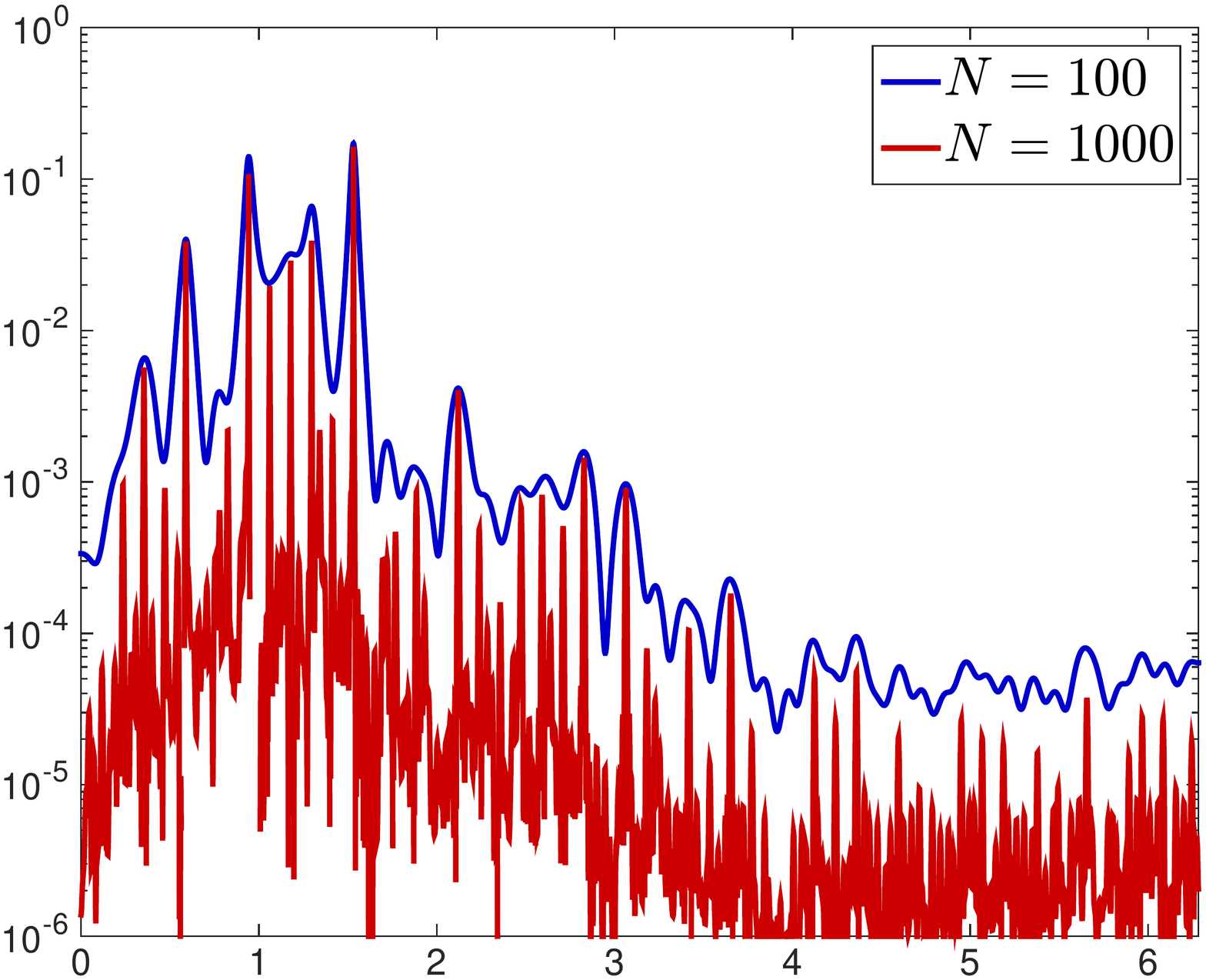}}
\put(300,135){\includegraphics[width=55mm]{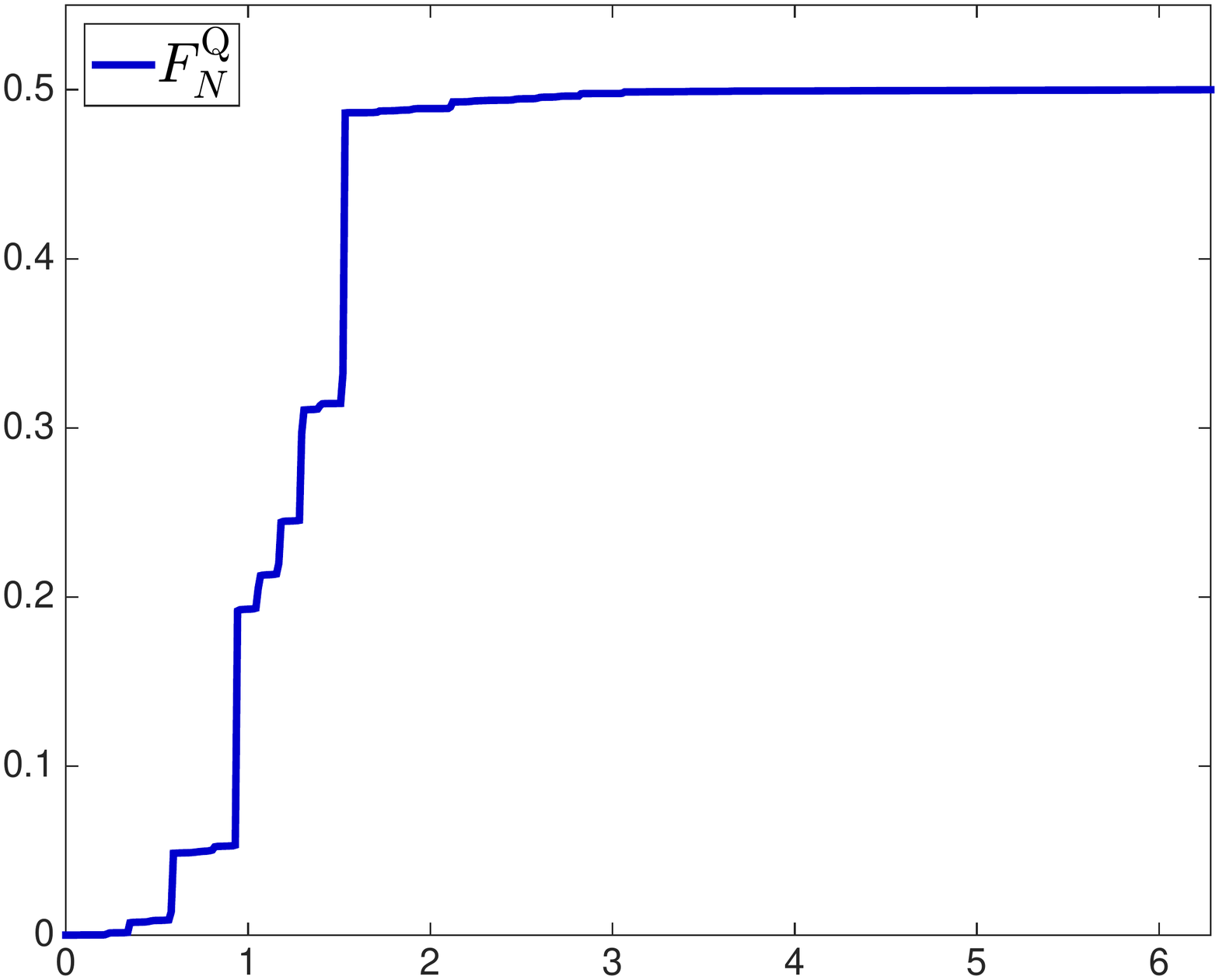}}

\put(57,160){\includegraphics[width=28mm]{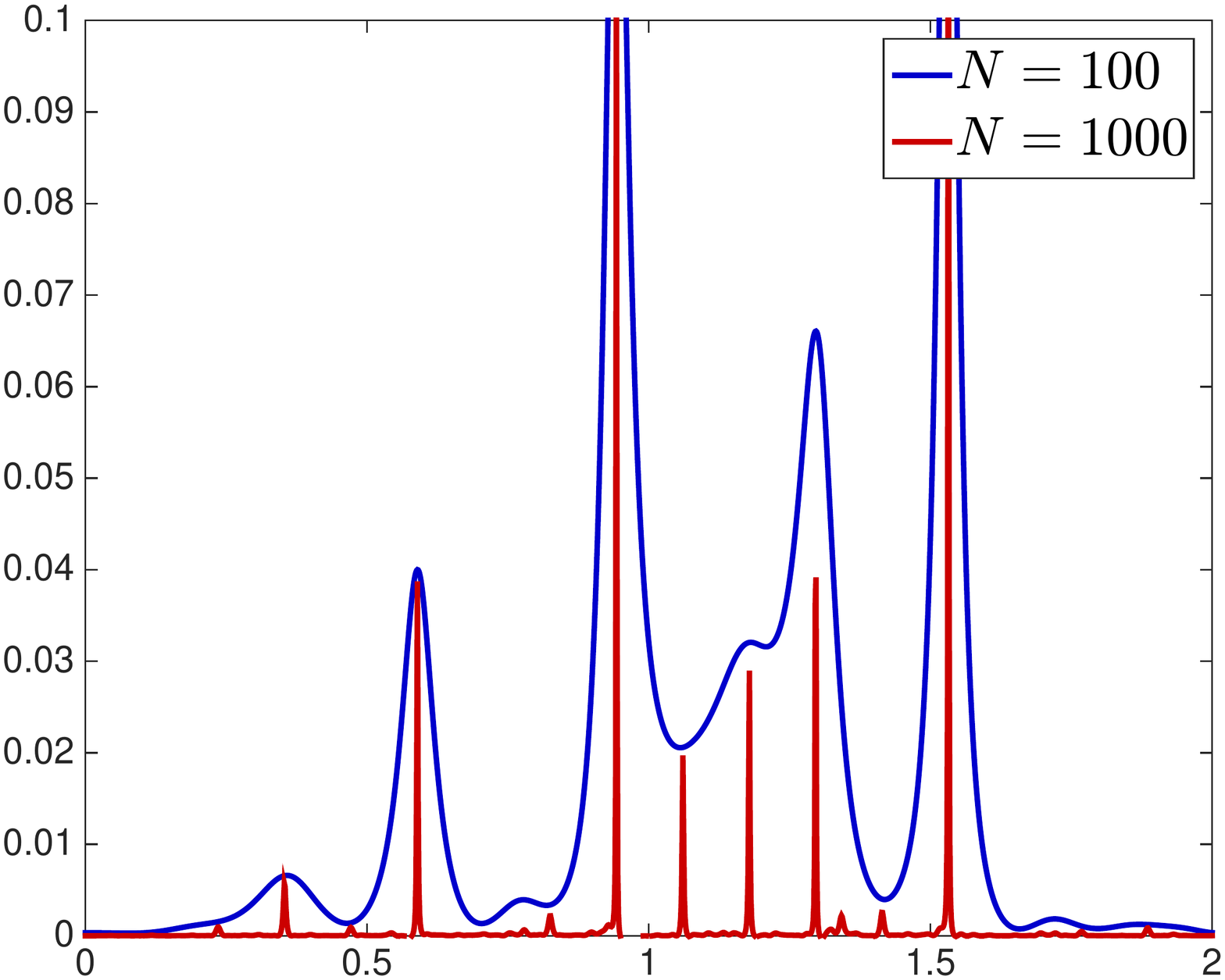}}
\put(65, 170){\vector(-1, -1){11}}

\put(352,150){\includegraphics[width=30mm]{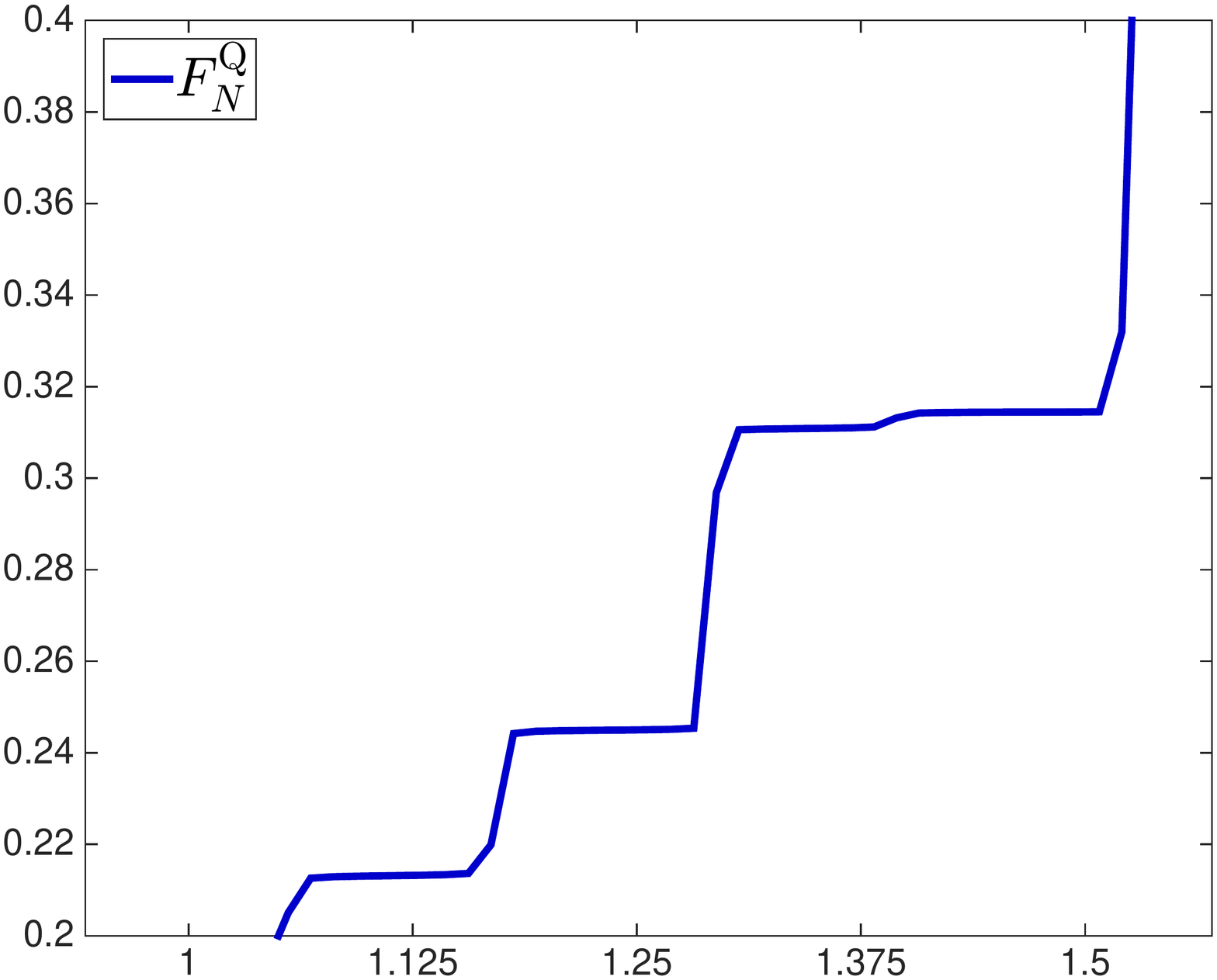}}
\put(360, 180){\vector(-2,1 ){11}}

\put(0,0){\includegraphics[width=55mm]{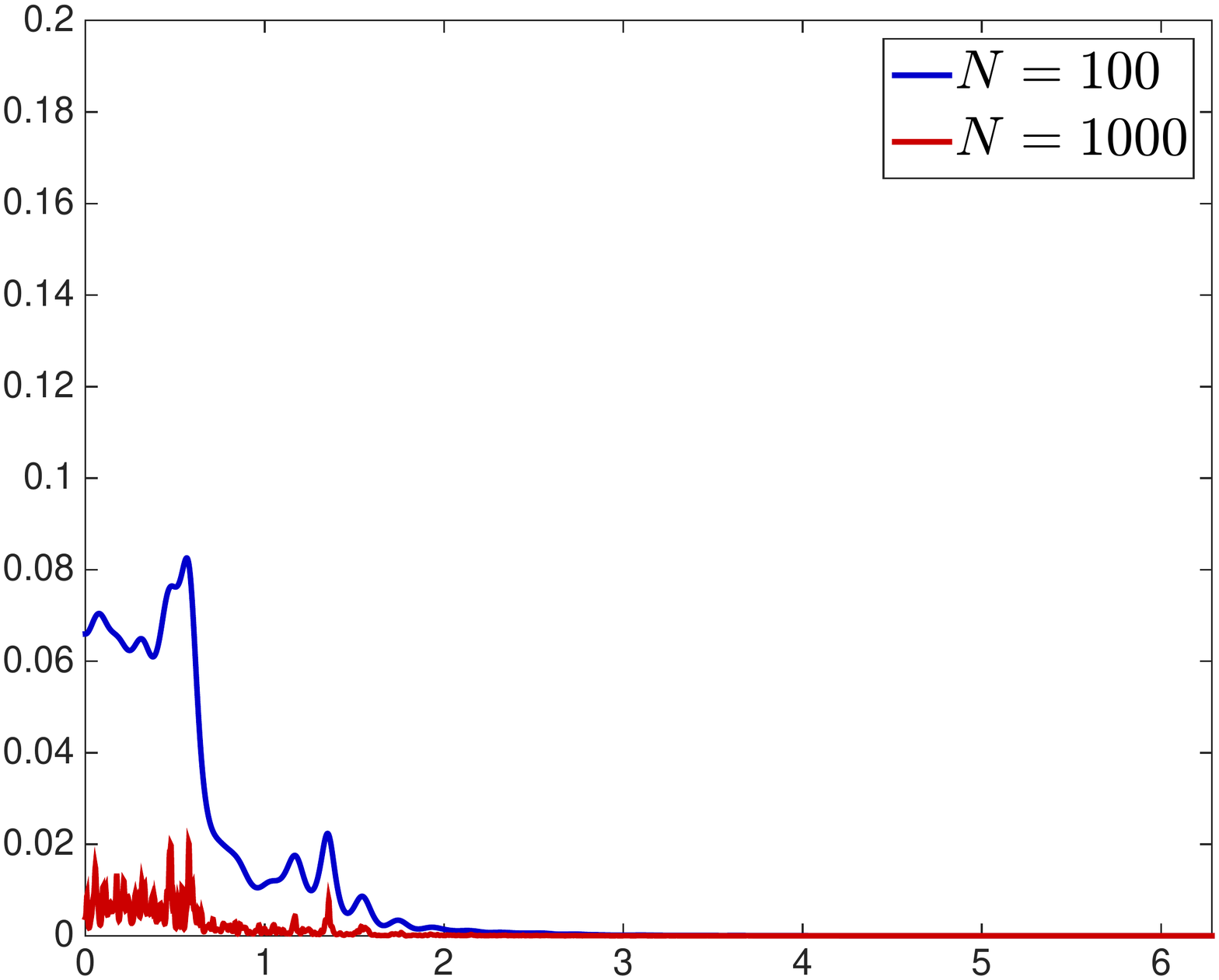}}
\put(150,0){\includegraphics[width=55mm]{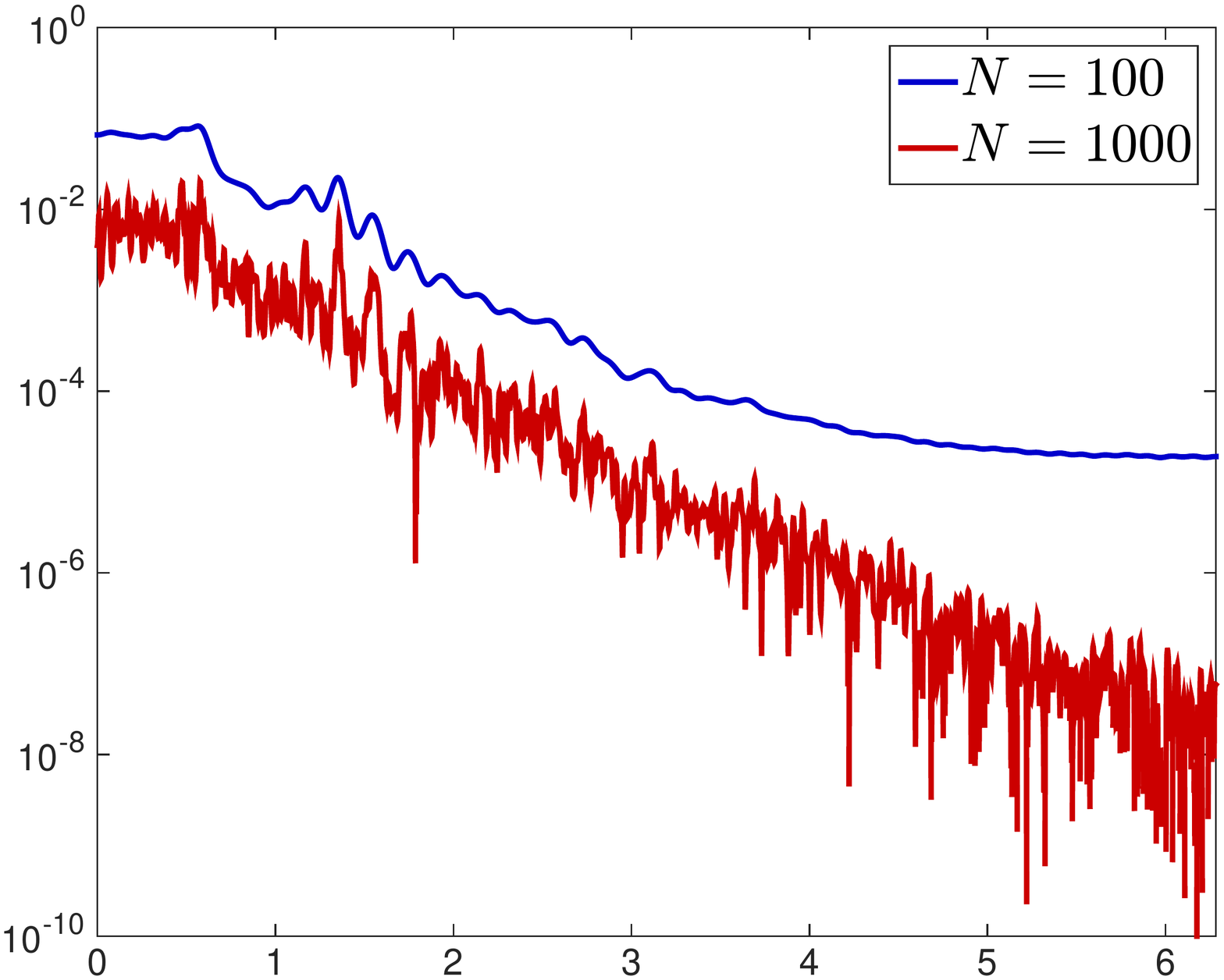}}
\put(300,0){\includegraphics[width=55mm]{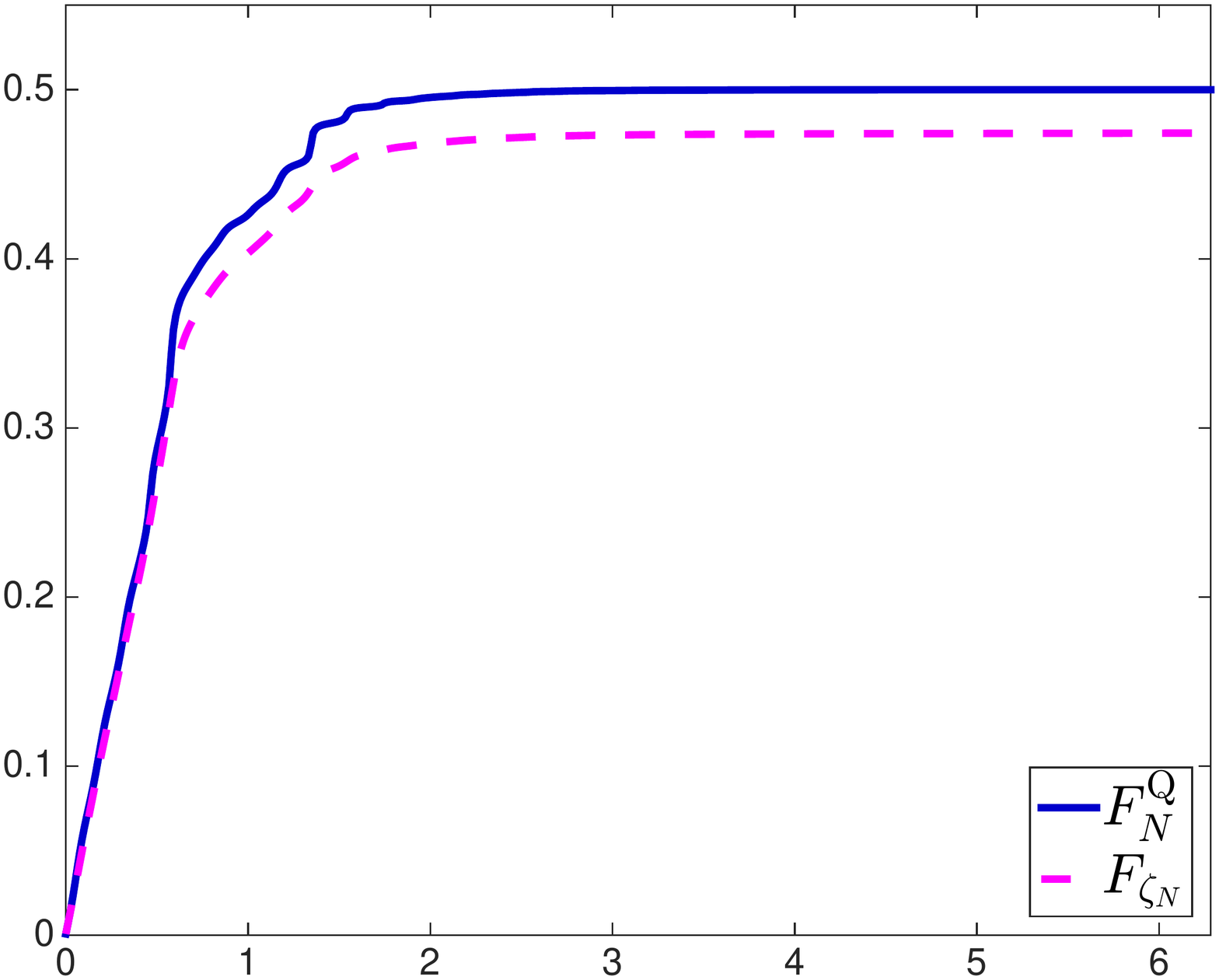}}

\put(57,25){\includegraphics[width=28mm]{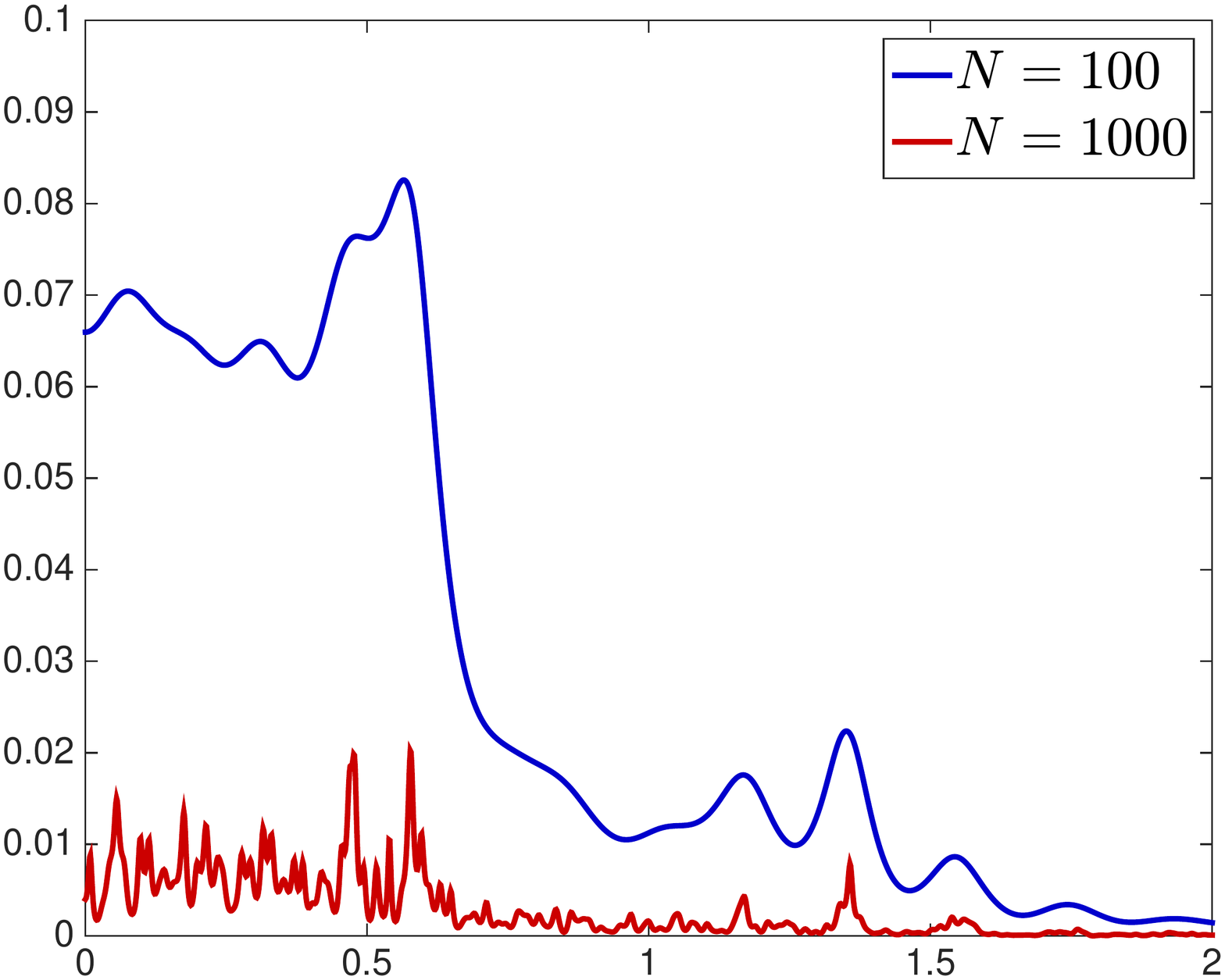}}
\put(58, 46){\vector(-1, -1){20}}

\put(57,0){\footnotesize $\omega\, [\mr{rad}/s]$}
\put(207,0){\footnotesize $\omega\, [\mr{rad}/s]$}
\put(357,0){\footnotesize $\omega\, [\mr{rad}/s]$}

\put(0,445){\footnotesize \rotatebox{90}{$Re = 13\cdot 10^3$}}
\put(0,310){\footnotesize \rotatebox{90}{$Re = 16\cdot 10^3$}}
\put(0,185){\footnotesize \rotatebox{90}{$Re = 19\cdot 10^3$}}
\put(0,35){\footnotesize \rotatebox{90}{$Re = 30\cdot 10^3$}}

\end{picture}
\caption{\figCaptionSize Cavity flow -- Point spectrum approximation $\zeta_N/(N+1)$ for $N=100$ and $N=1000$ (left: linear scaling; middle: logarithmic scaling). Right: approximation to the distribution function for $N = 100$. An atom is likely to be located where $\zeta_N/(N+1)$ shows convergence to a positive value, i.e., where the blue and red curves are close to each other.}
\label{fig:cavity_pointspec}
\end{figure*}

\clearpage

\section{Conclusion}
This work presented a method for data-driven approximation of the spectrum of the Koopman operator with the main contribution being the separation of the atomic and continuous parts of the spectra and their approximation with rigorous convergence guarantees. The atomic and absolutely continuous parts of the spectrum are approximated using the Christoffel-Darboux kernel; for the singular continuous part we develop a method for its detection as well as two different methods to approximate the entire spectral measure in a weak sense even in the presence of singular continuous spectrum. In addition, we propose a method for approximation of the Koopman operator based on spectral projections with guaranteed convergence in the strong operator topology as well as almost everywhere convergence of the spectral projections. See Figure~\ref{fig:BigPic} for an overview of the paper.  The approach is  simple and readily applicable to large-scale systems. The only limitation of the approach is, we believe, the class of the systems addressed, i.e., measure-preserving ergodic systems (or measure-preserving systems for which the preserved measure is known). One direction of future research is a generalization beyond this class of systems, e.g., to dissipative or unstable systems.

\begin{figure}[t!]
\begin{picture}(300,140)
 \put(0,0){\centerline{\includegraphics[width=.9\textwidth]{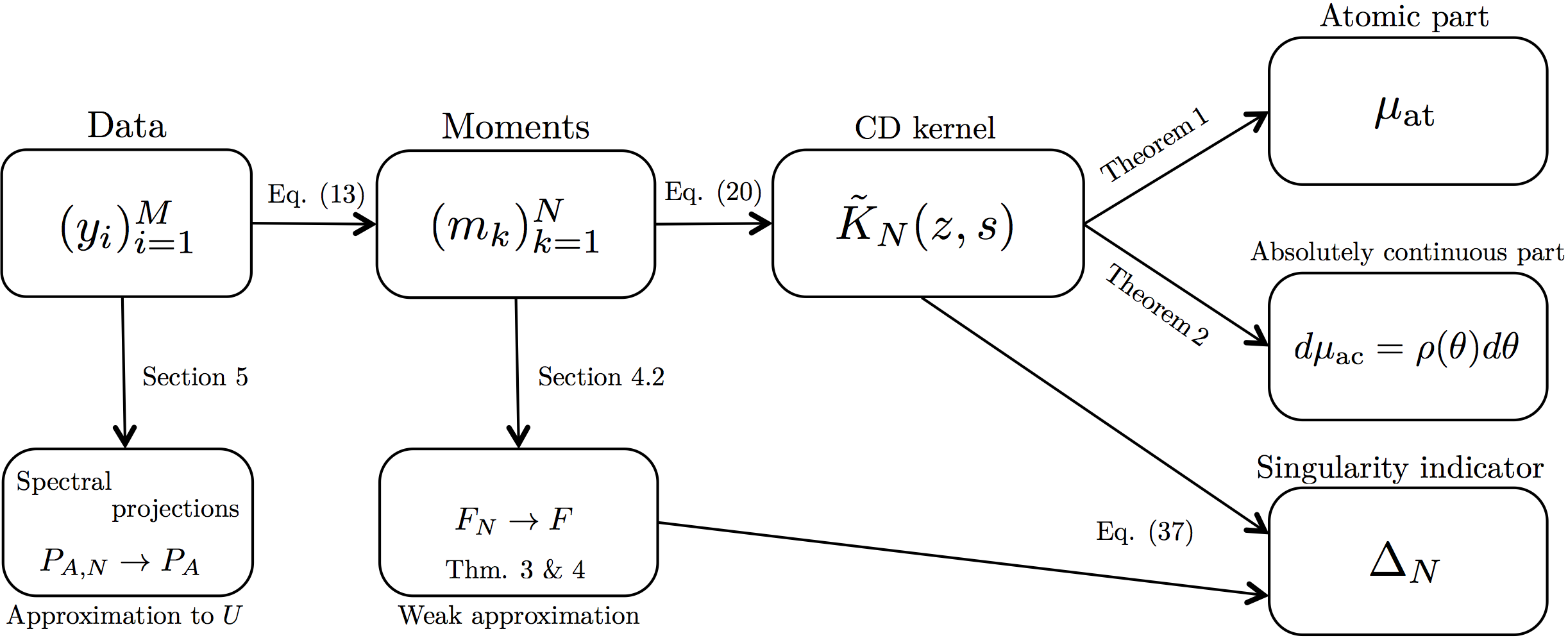}}}
\end{picture}
	\caption{\figCaptionSize{Overview of the paper.}}
	\label{fig:BigPic}
\end{figure}

\section{Acknowledgments}
The authors would like to thank the anonymous reviewers for their constructive comments that helped us improve the manuscript. The first author would also like to thank Hassan Arbabi for kindly providing the data for the cavity flow example, to Ryan Mohr and Marko Budi\v si\'c for sharing their initial thoughts on the topic and to Poorva Shukla for a careful reading of the manuscript and very helpful comments. The research was supported in part by the ARO-MURI grants W911NF-14-1-0359 and W911NF17-1-0306 as well as the DARPA grant HR0011-16-C-0116. The research of M. Korda was also supported by the Swiss National Science Foundation grant P2ELP2\_165166.

\bibliographystyle{abbrv}
\bibliography{./References}

\end{document}